\documentclass[twoside,11pt]{article}

\usepackage{jmlr2e}
\usepackage{amsmath}
\usepackage{diagbox}
\usepackage{placeins}
\usepackage[ruled,vlined]{algorithm2e}

\usepackage{float}
\usepackage[caption = false]{subfig}
\usepackage{graphicx}

\newcommand{\E}{\mathbf{E}}
\newcommand{\p}{\mathbf{P}}
\newcommand{\V}{\mathbf{Var}}
\newcommand{\indep}{\perp\!\!\!\!\perp} 
\newcommand{\e}{\mathrm{e}}
\newcommand{\diff}{\mathrm{d}}
\newcommand{\given}{\,\middle|\,}

\newtheorem{prop}[theorem]{Proposition}

\ShortHeadings{Inference in Single-Index Models}{}
\firstpageno{1}

\usepackage{lastpage}
\jmlrheading{22}{2021}{1-\pageref{LastPage}}{9/19; Revised
1/21}{2/21}{19-744}{Hamid Eftekhari, Moulinath Banerjee, and Ya'acov Ritov}
\ShortHeadings{Inference In High-dimensional Single-Index Models}{Eftekhari, Banerjee, and Ritov}

\begin{document}

\title{Inference In High-dimensional Single-Index Models Under Symmetric Designs}

\author{\name Hamid Eftekhari \email hamidef@umich.edu \\
       \name Moulinath Banerjee \email moulib@umich.edu \\
       \name Ya'acov Ritov \email yritov@umich.edu \\
       \addr Department of Statistics\\
       University of Michigan\\
       Ann Arbor, MI 48109, USA
}
\editor{Arnak Dalalyan}

\maketitle

\begin{abstract}
The problem of statistical inference for regression coefficients in a high-dimensional single-index model is considered. Under elliptical symmetry, the single index model can be reformulated as a proxy linear model whose regression parameter is identifiable.  We construct estimates of the regression coefficients of interest that are similar to the debiased lasso estimates in the standard linear model and exhibit similar properties: $\sqrt{n}$-consistency and asymptotic normality. The procedure completely bypasses the estimation of the unknown link function, which can be extremely challenging depending on the underlying structure of the problem. Furthermore, under Gaussianity, we propose more efficient estimates of the coefficients by expanding the link function in the Hermite polynomial basis. Finally, we illustrate our approach via carefully designed simulation experiments.
\end{abstract}

\begin{keywords}
sparsity, debiased inference, compressed sensing, Hermite polynomials
\end{keywords}

\section{Introduction and Background}
The single-index model (SIM) has been the subject of extensive investigation in both the statistics and econometrics literatures over the last few decades. It generalizes the linear model to scenarios where the regression function $\E[y \mid x]$ is not necessarily linear in the covariates but depends on them through an unknown transformation of $\langle x, \tau \rangle$, where $\tau$ is the vector of regression co-efficients: $\E[y \mid x] = g(\langle x, \tau \rangle)$. While allowing broad generality in the structure of the mean response, the single-index model also circumvents the curse of dimensionality by modeling the mean response in terms of a low-dimensional functional of $x$. This precludes the need for fitting high-dimensional nonparametric models. An excellent review of single-index models appears, for example, in the work of \cite{horowitz2009semiparametric}.

It is well-known that the parameter $\tau$ is, in general, unidentifiable in the single-index model, since any scaling of $\tau$ can always be absorbed into the function $g$. Therefore, some identifiability constraints are imposed for statistical estimation and inference, a popular choice being to set $\|\tau\|= 1$ and $\tau_1 \geq 0$.  General schemes for estimating $\tau$ involve optimizing an appropriate loss function (likelihood/pseudo-likelihood/least squares) in $(\underline \tau, \underline g)$ (generic parameter values) by alternately updating estimates of $\underline \tau$ and $\underline g$ \citep{carroll1997generalized}, or through a profile likelihood approach: see, e.g. the WNLS estimator in Section 2.5 of the work of \cite{horowitz2009semiparametric}. In any case, the estimation of $g$ figures critically in most estimation schemes. Inference on $\tau$ requires appropriate regularity assumptions on $g$, typically involving smoothness constraints, and while $g$ is only estimable at rate slower than $\sqrt{n}$, the parameter $\tau$ possesses $\sqrt{n}$-consistent asymptotically normal and efficient estimates, under certain regularity conditions. 

High-dimensional single index models have also attracted interest with various authors studying variable selection, estimation and inference using penalization schemes. \cite{ganti2015learning} use $\ell_1$ penalized estimates for learning high-dimensional index models were proposed and theoretical guarantees on excess risk for bounded responses was provided. \cite{foster2013variable} proposed an algorithm for variable selection in a
monotone single index model via the adaptive lasso. \cite{luo2016forward} proposed a penalized forward selection technique for high-dimensional single index models with a monotone link function. Cubic B-splines were employed by \cite{cheng2017bs} for estimating the single index model in conjunction with a SICA (smooth integration of counting and absolute deviation) penalty function for variable selection.  \cite{radchenko2015high} studied simultaneous variable selection and estimation in high-dimensional SIMs using a penalized least-squares criterion, with the link function estimated via B-splines, and provided theoretical results on the rate of convergence. \cite{yang17a} used a generalized version of Stein's lemma that allows estimation of the regression vector under known but not necessarily Gaussian designs. \cite{dudeja2018learning} and \cite{pananjady2019single} consider estimation in single- and multi-index models by expanding the unknown link function in the Hermite polynomial basis. 
More recently, \cite{hirshberg2018debiased} have proposed a method for average partial effect estimation in high-dimensional single-index models that is $\sqrt{n}$-consistent and asymptotically unbiased under sparsity assumptions on the regression coefficient, and to the best of our knowledge, this is the only work that provides asymptotic distributions in the high-dimensional setting. However, their method critically uses the form of the link function and also requires it to be adequately differentiable. 

In this paper, we develop an inference scheme for the regression coefficients of a high-dimensional single-index model with minimal restrictions on the (potentially random) link function---indeed, even discontinuous link functions are allowed---that \emph{completely bypasses} the estimation of the link. Thus, by dispensing with most regularity conditions on the link function, our approach can accommodate diverse underlying model structures. The price one pays for the feasibility of such an agnostic approach is an elliptically symmetric restriction on $X$. This assumption may be overly restrictive in applications where the entries of $X$ are determined by nature. However, in many applications one is allowed to design the measurement matrix. For example, in many applications of compressed sensing, Gaussian random matrices have been used as the measurement matrix, see for example the work of \cite{candes2008introduction} for an introduction to compressed sensing, the survey by \cite{li2018survey} on one-bit compressed sensing and the work of \cite{baraniuk2007compressive}, \cite{achim2014reconstruction} for applications in radar and ultrasound imaging.

Next we introduce our model. Consider the semiparametric single-index model:
\begin{align}\label{model}
y_i = f_i(\langle x_i , \tau \rangle) , \quad i \in \{1,2, \dots, n\},
\end{align}
where $f_i :\mathbf{R} \rightarrow\mathbf{R}$ are iid realizations of an unknown random function $f$, independent of $x_i$, and $\tau \in \mathbf{R}^p$ is an unknown parameter whose direction is the object of estimation\footnote{Equivalently, one can write $y_i = f(\langle x_i, \tau \rangle, u_i)$ for a deterministic function $f$ and iid standard uniform random variables $u_i \indep x_i$.}. Assume that $x_i \sim_{iid} N(0_p, \Sigma)$ for a positive definite matrix $\Sigma$. While $\tau$ is not identifiable in this model\footnote{Identifiability is discussed in more detail in Appendix \ref{identifiability}.}, an appropriate scalar multiple:  $\beta = \mu \tau$ (where $\mu$ will be defined shortly) is. As shown below, the new parameter $\beta$ turns out to be the vector of average partial derivatives of the regression function with respect to the covariates in the model, under certain conditions.

\textbf{Notation. }We will write $Y = (y_1, \dots, y_n)^T$, and let $X$ denote the matrix with $x_1^T, \dots, x_n^T$ in its rows. For subsets of indices $I\subset \{1,\dots, n\}, J \subset \{1, \dots, p\}$ we let $X_{I,J}$ be the submatrix of $X$ containing the rows with indices in $I$ and columns with indices in $J$. When $I$ or $J$ are singletons, we drop the brackets and identify $X_{I,k}$ with $X_{I,\{k\}}$ for $k=1,\dots,p$. Negative indices are used to exclude columns, so that for example $X_{i,-j}$ is the same as $X_{i,\{j \}^c}$. We denote by $\mathrm{e}_j$ the $j$-th element of the standard basis of $\mathbf{R}^p$. Finally, for two sequences $t_n,s_n$ we write $t_n \lesssim s_n$ to mean $t_n \leq c s_n$ for $n\geq 1$ and a constant $c > 0$ that does not depend on $n$. 

For Gaussian covariates, there is a specific feature of this model that obviates the need to estimate the link function that we now describe. Throughout the paper we assume $\|\Sigma^{\frac{1}{2}} \tau\|_2 = 1$, as otherwise we can rescale $\tau$ and $f$ appropriately without changing $\beta$. Define: 
\begin{align}
\mu &:= \E[ y_n \langle x_n, \tau \rangle ]= \E [f(\zeta) \zeta], \label{mu_def}\\
\beta &:= \mu\tau, \label{beta_def}\\
z &:= Y - X \beta, \label{z_def}
\end{align}
where $\zeta$ is a standard normal variable independent of $f$. 

The parameter of interest $\beta$ can be viewed as an average partial effect. To see this, write $g(\langle x, \tau \rangle ) = \E[y|X = x]$ and assume that $g$ is differentiable. The average partial effect with respect to the $j$-th covariate is then defined as: 
\begin{align*}
\E_x\left[ \frac{\partial}{\partial x_j} \E[y \mid x] \right] &= \tau_j \E_x [g'(\langle x, \tau \rangle)].
\end{align*}
By Stein's lemma \citep{Stein:1981}, $\E [g'(\langle x, \tau \rangle)] = \E \left[\langle x, \tau \rangle g(\langle x, \tau \rangle)\right]$, assuming both expectations exist and using the fact that $\langle x, \tau \rangle \sim N(0,1)$. Thus we have
\begin{align*}
\E_x\left[ \frac{\partial}{\partial x_j} \E[y \mid x] \right] &= \tau_j \E_x\left[  \langle x, \tau \rangle \E[y \mid x]\right] = \tau_j\E [y \langle x, \tau \rangle ]
= \mu \tau_j = \beta_j.
\end{align*}

As we seek to make inference on the importance/relative importance of the components of $\tau$ via estimates 
of $\beta$, we assume henceforth that $\mu \neq 0$.  It can be shown (see \ref{orthoproof} in the appendix) that  
\begin{align}\label{orthogonality}
\E [X^T z] = 0,
\end{align}
which is equivalent to $\E [z_i x_i] = 0$ for all $i = 1,\dots, n$. We therefore have the representation $Y = X\beta + z$ with $X$ and $z$ uncorrelated, which implies that $\beta \in \arg\min_{\beta'} \E\|Y - X\beta'\|_2^2$, thus motivating the use of (penalized) least squares methods for estimation. The above equation will be referred to as the orthogonality property of $X$ and $z$ and will be used subsequently at several places. 

The orthogonality property appears to have been first noted in the work of \cite{brillinger1982generalized} who used it to study the properties of the least squares estimator $\hat{\beta}_{ls}$ in the classical fixed $p$ setting and showed that the estimator was asymptotically normal. \cite{plan2016generalized} studied the estimation of $\beta$ in the $p > n$ setting using a generalized constrained lasso. While their results are quite general, they require knowing the constraint set $K$ over which least squares is performed. Under a different set of assumptions, \cite{thrampoulidis2015lasso} obtained an asymptotically exact expression for the estimation error of the regularized generalized lasso when $X$ has iid $N(0,1)$ entries and $\beta$ is generated according to a density in $\mathbf{R}^p$ with marginals that are independent of $p$.
\newline
Recall that a random vector $V$ is called spherically symmetric if its distribution is invariant to all possible rotations, i.e. $V \equiv_d PV$ for all orthogonal matrices $P$. Say that $X$ has an elliptically symmetric distribution if for some fixed vector $\mu$ and positive definite matrix $\Sigma$, $\Sigma^{-1/2}(X-\mu)$ is spherically symmetric. It turns out that the proxy linear model representation also holds \emph{more generally} when $X$ follows an elliptically symmetric distribution since the proof of the orthogonality property (\ref{orthogonality}) only uses the linearity of conditional expectations:
\begin{align*}
\E[x \mid \langle x, \tau \rangle] = \langle x, \tau \rangle b,
\end{align*}
for a non-random vector $b \in \mathbf{R}^p$. The latter is a well-known property of elliptically symmetric distributions,
 see Appendix \ref{Elliptical-Design} for the 
details. This fact appears in the work of \cite{li1989regression}, who then use the linear model representation to construct asymptotically normal and unbiased estimators of the regression coefficients in a 
fixed dimension parameter setting. More recently, \cite{goldstein2018structured} have studied structured signal recovery (in high dimensions) from a single-index model with elliptically symmetric $X$, which can be viewed as an extension of \cite{plan2016generalized}. 
\newline
Given the linear representation of the model as $ Y = X\beta + z$, it is natural to ask the following questions:
\begin{itemize}
\item
Can the debiasing techniques introduced in the setting of high-dimensional linear models be used for inference in single-index models?
\item
Can one improve on these procedures by going beyond a linear approximation?
\end{itemize} 
Section \ref{Gaussian-design} answers the first question in the affirmative by showing that under Gaussian design variants of the debiased Lasso estimator are consistent and asymptotically normally distributed. The second question is answered in Section \ref{Efficient-Inference} where, under Gaussian design, we improve the estimator in Section \ref{Gaussian-design} by using an estimate of the link function in the debiasing procedure to improve the asymptotic variance. Our simulations show that this reduction in variance can be significant. Finally, in Appendix \ref{Elliptical-Design} we extend the aforementioned results to the case of elliptically symmetric designs with subgaussian tails. 

\section{Inference under Gaussian Design} 
\label{Gaussian-design} 
In this section we apply the debiasing technique to obtain $\sqrt{n}$-consistent estimators of individual coordinates of $\beta$ under the assumption of Gaussian design. Our main theorems assume the existence of pilot estimators of $\beta$ that possess sufficiently fast ($\ell^1$-norm) rates of convergence. Subsection \ref{pilot-estimate} discusses the construction of these pilot estimators. The extension of the results in this section to the case of elliptically symmetric design is considered in Appendix \ref{Elliptical-Design}.

\subsection{Background on Debiased Lasso}

The debiased lasso procedure proposed by \cite{javanmard2014confidence,zhang2014confidence,vandegeer2014} can be motivated as correcting the bias of low-dimensional projection estimators using the lasso estimate. More precisely, suppose that in a linear model $y = \langle x, \beta \rangle + \varepsilon$ the goal is to conduct inference on the first coordinate of $\beta$. In the low-dimensional scenario where $p < n$ and $\mathrm{rank}(X) = p$, the OLS estimate of $\beta_1$ can be written as 
\begin{align*}
\hat{\beta}_1^{OLS} = \frac{u^T Y}{u^T X_{\cdot, 1}},
\end{align*}
 where $u$ is the projection of the first column $X_{\cdot,1}$ on the orthocomplement of the span of $X_{\cdot, 2}, \dots, X_{\cdot, p}$. In the high-dimensional setting where $p > n$, this projection is typically zero, but one may still be able to find a vector $u$ for which $u^T X_{\cdot, 1}$ is large while $\max_{j>1} |u^T X_{\cdot, j}|$ has a slow rate of growth. The bias of this projection estimator can further be reduced by using the Lasso estimate \citep{tibshirani1996regression} of $\beta$:
 \begin{align*}
 \tilde{\beta}_1 &:= \frac{ u^T (Y - X_{\cdot, -1} \hat{\beta}_{-1}^{\text{lasso}} )}{u^T X_{\cdot, 1}} = \hat{\beta}_1^{\text{lasso}} + \frac{ u^T (Y - X \hat{\beta}^{\text{lasso}} )}{u^T X_{\cdot, 1}} ,\\
 \hat{\beta}^{\text{lasso}} &:= \arg \min_{\beta'} \left\{ \frac{1}{2n} \| Y - X\beta' \|_2^2 + \lambda \| \beta' \|_1 \right\}.
 \end{align*}
A natural choice for the projection vector $u$ is $u = X \Sigma^{-1}\mathrm{e}_1$ when $\Sigma$ is known,  used for example by \cite{javanmard2018debiasing}. Our estimator in Section 2.2 uses this choice up to a constant multiple. The analysis of this estimator is similar to the linear model setting and is included here as a stepping stone to the more involved analysis of later sections. 
\newline
The estimator in Section 2.3 uses node-wise lasso to estimate (a multiple) of the first row of $\Sigma^{-1}$. The use of node-wise lasso was suggested previously by \cite{zhang2014confidence,vandegeer2014}. Our estimator is slightly different in that we use sample splitting to break the dependence between the estimate of $\Sigma^{-1}$ and the approximation error $z$, whereas this is not necessary in linear models since the noise and the design are typically assumed to be independent. 
\newline
Finally, in Section 3, we use a higher order expansion of the link function in the Hermite polynomial basis to further reduce the magnitude of the approximation error $z$ to obtain a more efficient estimator than the linear debiased estimator. Hermite polynomials provide a natural basis in our setting because of their orthogonality property under the Gaussian measure. Our estimator combines ideas from non-parametric and high-dimensional statistics and to the best of our knowledge has not been studied previously.

\subsection{Inference on $\beta$ when $\Sigma$ is Known}

We first consider the case where $\Sigma$ is known. In this case, the distribution of $x_i \sim N(0, \Sigma)$ is fully known, and we can therefore compute the $L_2$ projection of any covariate, $x_{n,1}$, over the remaining ones, $x_{n,-1}$. Specifically, define 
\begin{align*}
\gamma := \left(\E [x_{n,-1} x_{n,-1}^T] \right)^{-1} \E[ x_{n,-1} x_{n,1}].
\end{align*}
In words, $\gamma$ is the vector of coefficients when regressing the first covariate on the rest, at the population level. The resulting residuals 
$r_i := x_{i,1} - \langle \gamma, x_{i,-1} \rangle$
satisfy
\begin{align*}
\E [r_i x_{i, -1}^T] = \E[ x_{i,1} x_{i,-1}^T] - \gamma^T \E [x_{i, -1} x_{i,-1}^T] = 0.
\end{align*}

Suppose a sample size of size $2n$ is given\footnote{To avoid the loss of efficiency due to sample splitting, one can swap the role of the two subsamples and take the mean of the resulting estimates, see Remark \ref{sample-splitting} for details.}. Then
\begin{itemize}
\item Compute the residuals $r_i$ as defined above on the first sub-sample $(x_i, y_i)_{i=1}^n$.
\item Compute the lasso \citep{tibshirani1996regression} estimator $\hat{\beta}$ on the second subsample $(x_i, y_i)_{i=n+1}^{2n}$:
\begin{align*}
\hat{\beta} \in \arg \min_{\beta'} \left\{ \frac{1}{2n} \sum_{i=n+1}^{2n} (y_i - \langle x_i, \beta' \rangle )^2 + \lambda \|\beta' \|_1 \right\}
\end{align*}
with\footnote{In practice (and in our simulations) we use cross-validation to choose the regularization parameter $\lambda$.} $\lambda \gtrsim \sigma_x \sigma_z \sqrt{\log(p) / n}$, where $\sigma_x, \sigma_z$ are the subgaussian constants of $x_n$ and $z_n$, respectively.

\end{itemize}

Using these residuals and the pilot estimator $\hat{\beta}$, define a debiased estimator of $\beta_1$ as 
\begin{align}\label{est-def-known}
\tilde{\beta}_1 := \hat{\beta}_1 +  \frac{ \sum_{i=1}^n r_i( y_i - \langle x_i,\hat{\beta} \rangle)}{\sum_{i=1}^n r_i x_{i,1}}.
\end{align}

The following theorem characterizes the asymptotic distribution of $\tilde{\beta}_1$.

\begin{theorem}\label{thm_known_sig}
	Suppose $(x_i, y_i)_{i=1}^{2n}$ follow model (\ref{model}) with $x_i \sim N(0,\Sigma)$ and let 
	\begin{align*}
	\nu^2 =\frac{\E [r_n^2 z_n^2]}{(\E [r_n^2])^2}.
	\end{align*}
	Assume also that the following conditions are satisfied:
	
	\begin{enumerate}
		\item The sparsity of $\beta$ satisfies $s = o(\frac{n}{\log^2(p)})$.
		\item There exist $0 < c, C < \infty$ such that $c \leq \lambda_{\min}(\Sigma) \leq \lambda_{\max}(\Sigma) \leq C$. 
		\item  $\E|y_n|^{2 + \alpha} \leq M < \infty$ for some $\alpha, M > 0$ and all $n \geq 1$.
        \item The pilot estimator $\hat{\beta}$ satisfies
        \begin{align*}
            \|\hat{\beta} - \beta\|_2 \lesssim \sqrt{\frac{s \log(p)}{n}},\quad \text{ with probability } 1-o(1).
        \end{align*}
	\end{enumerate}	
	
	Then
	\begin{align*}
	\sqrt{n}(\tilde{\beta}_1 - \beta_1) = \nu \cdot \Xi_n, \quad \text{ where }\quad \Xi_n \rightarrow_d N(0,1).
	\end{align*}
	
\end{theorem}

\textbf{Discussion of the Assumptions.} The following points clarify the assumptions of the preceding theorem.
\begin{enumerate}
\item (Gaussianity / Elliptical Symmetry). Under the quadratic loss criterion, the population loss $\E[(y - \langle x, \beta' \rangle)^2]$ is minimized at $\beta^\star = \beta + \Sigma^{-1}\E [ (y - \langle x, \beta \rangle )x]$, where $\beta = \mu \tau$ is to be estimated. The bias term $ \Sigma^{-1}\E [ (y - \langle x, \beta \rangle )x]$ vanishes whenever 
\begin{align}
\E[x \mid \langle x, \beta \rangle ] = c \langle x, \beta \rangle \label{cond-linearity}
\end{align}
for a fixed (non-random) vector $c \in \mathbf{R}^p$. This latter condition is satisfied for elliptically symmetric distributions and is the basis of the Lasso procedure used in our work. Details are provided in Appendix \ref{departure}. Section 6 in the work of \cite{li1989regression} analyzes estimation in low dimensional single-index models under departures from elliptical symmetry using arbitrary convex loss functions and establishes upper bounds on the bias of the corresponding M-estimators in terms of
appropriate measures of such departures. 
The authors further consider the empirical counterparts of these measures as practical ways to verify the elliptical symmetry assumption. It is conceivable that similar diagnostic measures can work in our setting. For example, one can use sample splitting to estimate $\beta$ from a subsample and verify an empirical version of (\ref{cond-linearity}) on an independent subsample. These considerations are important for practice and can be the subject of future work. Another interesting extension would be to quantify the asymptotic distribution of the least squares estimator of $\beta$ in our setting under structured (as in adequately parametrized) violations of symmetry. 

\item
If the design matrix $X$ is not centered, but $\E [x_n]$ is known, then one can first center $X$ by using $\tilde{X} = X - \E X$ and redefining the link function:
	\begin{align*}
	\tilde{f}_i( \cdot ) = f_i(\cdot + \langle \E [x_i] , \tau \rangle).
	\end{align*}
	Even though $f_i$ now depends on the mean of $x_i$, it is independent of $x_i$, and hence all the theoretical results in this paper continue to hold.  	
	In simulations we consider a case where $X$ has nonzero mean and is centered on the sample, prior to constructing the estimator (since, in reality, $\E[ X]$ will not be known). While this, strictly speaking, does not fall within the purview of our approach (since the rows of $X - n^{-1} \mathbf{1}_n \mathbf{1}_n^T X$ are no longer independent), our inference procedure is still seen to produce satisfactory results in simulations. 

\item
Assumption 3 on the spectrum of $\Sigma$ is typically used in the high-dimensional inference literature, for example see the works of \cite{vandegeer2014,javanmard2018debiasing,javanmard2014confidence}. This assumption simplifies concentration arguments and also ensures that the design matrix satisfies the restricted eigenvalue condition with high probability.

\end{enumerate}

\begin{remark}
\begin{enumerate}
	\item In the special case of a standard linear model, i.e. when $y_i = \langle x_i, \beta \rangle + \varepsilon_i$ with $\varepsilon_i \indep x_i$, the asymptotic variance of $\sqrt{n}(\tilde{\beta}_1 - \beta_1)$ reduces to $\nu^2 = \sigma^2 / \E [r_n^2]$ which is the variance of the OLS estimator of $\beta_1$ in the low-dimensional case and the debiased estimator \citep{vandegeer2014} in the high-dimensional case.

	\item From the subgaussianity of $y_n$ and the Cauchy-Schwarz inequality, it follows that $\E[ r_n^2 z_n^2]$ is uniformly bounded above. Together with $\E [r_n^2] \geq \lambda_{\min}(\Sigma)$ (see Lemma \ref{residual_prop} in the appendix for a proof of this), this shows that $\nu^2$ is uniformly bounded above and thus the rate of convergence of $\tilde{\beta}_1$ is indeed $\sqrt{n}$. On the other hand, if $\E[ r_n^2 z_n^2 ]= o(1)$, then the rate of convergence of $\tilde{\beta}_1$ is faster than $\sqrt{n}$, that is,  we obtain $\sqrt{n}(\tilde{\beta}_1 - \beta_1) \rightarrow_p 0$. 
\end{enumerate}
\end{remark}
\subsection{Inference when $\Sigma$ is Unknown}
In this section we consider the problem of debiasing an estimate $\hat{\beta}_1$ of $\beta_1$ when the precision matrix $\Sigma^{-1}$ is unknown but estimable. Suppose that we have a sample $\mathcal{S} = (x_i, y_i)_{i=1}^{2n}$ of size $2n$, and that we use the second sub-sample $\mathcal{S}_2 = (x_i,y_i)_{i=n+1}^{2n}$ to find an estimate $\hat{\gamma}$ of $\gamma$ using node-wise lasso (as proposed by \citealp{vandegeer2014} in the setting of linear models):
\begin{align}\label{nodewise}
\hat{\gamma} \in \arg\min_{\gamma'} \left\{ \frac{1}{2n} \sum_{i=n+1}^{2n} (X_{i,1} - \langle \gamma', X_{i,-1} \rangle)^2 + \lambda_{\text{node}} \| \gamma'\|_1 \right\}.
\end{align}
This estimate of $\gamma$ is then used to obtain estimates $\hat{r}_i$ of $r_i$ on the first sub-sample $\mathcal{S}_1 = (x_i, y_i)_{i=1}^n$:
\begin{align}\label{resid_unkn_sig}
\hat{r}_i := x_{i, 1} -\langle  x_{i, -1}, \hat{\gamma}\rangle, \quad i \in \{1,\dots, n\}.
\end{align}
The debiased estimator of $\beta_1$ on the first subsample is then defined as
\begin{align}\label{est-def-unknown}
\tilde{\beta}_1 := \hat{\beta}_1 + \frac{\sum_{i=1}^n \hat{r}_i (y_i - \langle x_i, \hat{\beta}\rangle)}{\sum_{i=1}^n{\hat{r}_i x_{i,1}}},
\end{align}
where the pilot estimate $\hat{\beta}$ is is the lasso estimator:
\begin{align*}
\hat{\beta} \in \arg \min_{\beta'} \left\{ \frac{1}{2n} \sum_{i=1}^{n} (y_i - \langle x_i, \beta' \rangle )^2 + \lambda_{\text{pilot}} \|\beta' \|_1 \right\}
\end{align*}
with $\lambda_{\text{pilot}} \gtrsim \sigma_x \sigma_z \sqrt{\log(p) / n}$, and where $\sigma_x, \sigma_z$ are the subgaussian constants of $x_n$ and $z_n$, respectively.
The asymptotic distribution of the above estimator is characterized in the following theorem.
\begin{theorem}\label{thm-unknown-sigma}
	Suppose $(x_i, y_i)_{i=1}^{2n}$ follow the model (\ref{model}) with $x_i \sim N(0,\Sigma)$ and that the following conditions are satisfied:
	\begin{enumerate}
		\item We have $s := \|\beta\|_0 \lor \|\gamma\|_0 = o(\frac{\sqrt{n}}{\log(p)})$.
		
		\item The estimate $\hat{\gamma}$ of $\gamma$, depending on data in the second sub-sample $\mathcal{S}_2$, satisfies
		\begin{align*}
		\p \left(\|\hat{\gamma} - \gamma \|_1 \leq c_\gamma s\sqrt{\frac{ \log p }{n}}\right) \rightarrow 1,
		\end{align*}
		for a constant $c_\gamma$ not dependent on $n$.
		\item There exist $0 < c, C < \infty$ such that $c \leq \lambda_{\min}(\Sigma) \leq \lambda_{\max}(\Sigma) \leq C$. 
		\item  $z_n$ is subgaussian with $\|z_n\|_{\psi_2} \leq \sigma_z$ for some $\sigma_z<\infty$ not depending on $n$.
\end{enumerate}	
	
Then
	\begin{align*}
	\sqrt{n}(\tilde{\beta}_1 - \beta_1) = \nu \cdot \Xi_n, \quad \text{ where }\quad \Xi_n \rightarrow N(0,1).
	\end{align*}
	
\end{theorem}

\textbf{Discussion of assumptions. } We provide a discussion of the assumptions of the theorem.
\begin{enumerate}
    \item (Sufficient condition for assumption 4.) The approximation error $z_n = y_n - \langle x_n, \beta \rangle$ is subgaussian whenever $y_n$ is subgaussian, since
\begin{align*}
\| z_i \|_{\psi_2} &\leq \| y_n \|_{\psi_2} + |\mu| \cdot  \| \langle x_i, \tau \rangle \|_{\psi_2},
\end{align*}
and $|\mu|$ is bounded by
\begin{align*}
| \mu | &= | \E [y_n \langle x_n, \tau \rangle] | \\
&\leq \sqrt{ \E [y_n^2] } \sqrt{\E[ \langle x_n, \tau \rangle^2 ]}\\
&\lesssim \| y_n \|_{\psi_2}.
\end{align*}
Thus $z_n$ is subgaussian and $\|z_n\|_{\psi_2} \lesssim \|y_n\|_{\psi_2}$. A similar argument shows the converse is also true and $\|y_n\|_{\psi_2} \leq \|z_n\|_{\psi_2}$.

	\item Our approach to approximately de-correlating the design matrix uses sample splitting for estimation of $\gamma$, and the supporting argument is somewhat different from the ones in the work of \cite{zhang2014confidence}, \cite{javanmard2014confidence} and \cite{vandegeer2014}.  
	This is because in our linear model $y_i = \langle x_i, \beta \rangle + z_i$, the error $z_i$ is not independent of, but only uncorrelated with, $x_i$. Since our proof relies on the $\ell_1$ consistency of $\hat{\gamma}$, it necessitates sparsity of the rows of $\Sigma^{-1}$.  It is not clear if the argument could be changed to justify the use of standard debiasing techniques (introduced in the aforementioned papers) that do not require sparsity assumptions on $\Sigma^{-1}$.


	\item It is straightforward to see that given assumptions 1 and 3 of Theorem \ref{thm-unknown-sigma}, assumption 2 is satisfied with high probability under a Gaussian design. 
First note that the eigenvalues of $\Sigma_{-1,-1}:= \E[x_{n,-1}x_{n,-1}^T]$ are within the interval $[\lambda_{\min}(\Sigma), \lambda_{\max}(\Sigma)]$ and hence by assumption 3 are bounded away from $0$ and $+\infty$. Using Lemma \ref{RE_condition}, this implies that the design matrix $X$ satisfies the restricted eigenvalue (RE) condition\footnote{See Definition \ref{def:RE} in Subsection \ref{pilot-estimate} for the definition of restricted eigenvalues.}
 (with a restricted eigenvalue that is bounded away from zero) with probability $1- o(1)$ as $n \rightarrow \infty$.
Next, by the definition of $\gamma$ and $r_i$ we can write
	\begin{align*}
	x_{i,1} = \langle \gamma, x_{i,-1} \rangle + r_i, \quad i = n+1, \dots, 2n,
	\end{align*}
where by construction $r_i$ and $x_{i,-1}$ are uncorrelated (and hence independent). Thus standard results for the lasso estimator \citep[e.g. Theorem 7.2 of][]{bickel2009simultaneous} guarantee with high probability that $\| \hat{\gamma} - \gamma\|_1 \lesssim s\sqrt{\log(p) / n}$ as long as the tuning parameter satisfies $\lambda_{\text{node}} \gtrsim \lambda_{\max}(\Sigma) \sqrt{\log(p) / n}$.
Obviously, all considerations here go through if instead of estimating $\beta_1$, we are interested in estimating $\beta_k$ for some other fixed $k$: we replace 1 by $k$ at the pertinent places. 
	
	%
\end{enumerate}

\textbf{Examples.} In the following we provide examples that satisfy the assumptions of the above theorem.
\begin{itemize}
\item (Noisy one-bit compressed sensing). Suppose that $y = \operatorname{sign}(\langle x, \tau \rangle) \cdot \epsilon$ where $\tau$ is a s-sparse vector with $\|\tau\|_2 = 1$, $\epsilon \in \{ 1, -1\}$ is a random sign with $\p(\epsilon = 1) = p \in [0,1]$ independent of $x \sim N(0, I)$. Variants of this model have been widely studied in the compressed sensing literature. See the survey by \cite{li2018survey} for applications to wireless sensor networks, radar, and bio-signal processing. It is clear that assumptions 1 and 3 of Theorem \ref{thm-unknown-sigma} hold for this model. Note that in this model the rows of $\Sigma^{-1} = I$ are 1-sparse, and therefore the third remark above applies for the consistency of the node-wise lasso estimate $\hat{\gamma}$, implying the second assumption holds as well. To check the fourth assumption, use the boundedness of $y$ to write
\begin{align*}
\|z \|_{\psi_2} \leq \|y\|_{\psi_2} + |\mu| \cdot \| \langle x, \tau \rangle \|_{\psi_2} \lesssim 1 + | \mu| \leq 1+  \sqrt{\frac{2}{\pi}}.
\end{align*}
The fifth assumption is also satisfied in light of Proposition \ref{prop1}.

\item Suppose that $y = g( \langle x, \tau \rangle) + \varepsilon$ where $g$ is a fixed but unknown bounded or Lipschitz function, $x \sim N(0,\Sigma)$ with $\Sigma_{ij}=\rho^{|i-j|}$ for some fixed $\rho \in (0,1)$, and $\varepsilon$ is a mean-zero $\sigma$-subgaussian error independent of $x$. Then as long as $\|\tau\|_0 = o(\sqrt{n}/\log(p) )$, the assumptions of the theorem are satisfied. Note that in this case, $\Sigma$ is the correlation matrix of an $\mathrm{AR}(1)$ process, which is well-known to have a $2$-banded inverse covariance matrix.
Therefore in this case $\| \gamma \|_0 \leq 2$ and assumption 1 is satisfied. It can also be shown that $(1-\rho)/(1+\rho) \leq \lambda_{\min}(\Sigma) \leq \lambda_{\max}(\Sigma) \leq (1 + \rho) / (1 - \rho)$ (for example using Lemma 6 in the work of \cite{gray2006toeplitz}), showing that assumption 3 is satisfied. The sub-gaussianity of $z = y - \langle \beta, x \rangle$ easily follows from the Lipschitz or boundedness assumption on the link function. For example, if $g$ is L-Lipschitz, then
\begin{align*}
\| z \|_{\psi_2} &\leq \| g( \langle x, \tau \rangle) \|_{\psi_2} + \| \mu \langle x, \tau \rangle \|_{\psi_2} + \sigma \\
&\leq  \| L \cdot \langle x, \tau \rangle \|_{\psi_2} + \| \mu \langle x, \tau \rangle \|_{\psi_2} + \sigma \\
&\lesssim L + | \mu | + \sigma.
\end{align*}
Furthermore, $\mu$ is bounded by $L$ as well:
\begin{align*}
|\mu| &= \left| \E[g(\langle x, \tau \rangle) \langle x, \tau \rangle ] \right| \\
&\leq \E[ \left|g(\langle x, \tau \rangle)\right|  \left| \langle x, \tau \rangle \right| ] \\
&\leq L \E[ \langle x, \tau \rangle^2 ] = L.
\end{align*}
Therefore as long as $\sigma$ and $L$ are uniformly bounded (in $n$), the subgaussian norm of $z$ is also uniformly bounded and assumption 4 is satisfied. The consistency of the lasso pilot estimator $\hat{\beta}$ and the node-wise estimate $\hat{\gamma}$ now follow from Proposition \ref{prop1} and the third remark above.

\end{itemize}

\begin{remark}\label{sample-splitting}
	In order to avoid loss of efficiency due to sample splitting, one can change the roles of the two sub-samples in the theorem to compute two estimates $\tilde{\beta}_1^1, \tilde{\beta}_1^2$, and use the average of $\tilde{\beta}^1_1$ and $\tilde{\beta}^2_1$ as the final estimator. The proof of Theorem \ref{thm-unknown-sigma} shows that 
	\begin{align*}
	\frac{\sqrt{n}(\tilde{\beta}_1^m - \beta_1)}{\nu} = G^m + o_p(1), \quad m = 1,2,
	\end{align*}
	where 
	\begin{align*}
	G^m = \frac{1}{\sqrt{n}} \frac{\sum_{i\in \mathcal{S}_m} r_i z_i}{\nu \E [r_n^2]}
	\end{align*}
	depends only on the $m$-th subsample $\mathcal{S}_m$ so that $G^1$ and $G^2$ are independent. Moreover, we have $G^m \rightarrow_d N(0,1)$ as $n \rightarrow \infty$, and so by independence
	\begin{align*}
	G^1 + G^2 \rightarrow_d N(0,2).
	\end{align*}
	Consequently, for the average estimator $\tilde{\beta}_1^{\text{avg}} = (\tilde{\beta}_1^1 + \tilde{\beta}_1^2)/2$ we have
	\begin{align*}
	 \frac{\sqrt{2n} (\tilde{\beta}_1^{\text{avg}} - \beta_1)}{\nu} &= \frac{1}{\sqrt{2}} \left( \sqrt{n}\frac{(\tilde{\beta}_1^1 - \beta_1)}{\nu} + \sqrt{n}\frac{(\tilde{\beta}_1^2 - \beta_1)}{\nu} \right)\\
	&= \frac{1}{\sqrt{2}} (G^1 + G^2 + o_p(1)) \\
	&\rightarrow_d N(0,1),
	\end{align*}
	showing that the lost efficiency due to sample splitting is regained by switching the roles of sub-samples. This technique is well-known, see for example the work of \cite{chernozhukov2018double} for a similar application.
	
%

\end{remark}

\subsection{A Pilot Estimator of $\beta$} \label{pilot-estimate}


In this subsection the construction of a pilot estimator for $\beta$ is discussed. Consistency and rates of convergence of (generalized) constrained lasso estimators for single-index models under Gaussian or elliptically symmetric design have been established in other works, for example by \cite{plan2016generalized} and \cite{goldstein2018structured}. In what follows we state and prove the consistency of penalized lasso under assumptions similar to the ones typically used in high dimensional linear models.

Let $\hat{\beta}$ be a solution to the penalized form of the lasso problem:
\begin{align}\label{pilot-def}
\hat{\beta} = \arg \min_{\beta'} \left\{ \frac{1}{2n} \| Y - X \beta' \|_2^2 + \lambda \| \beta' \|_1\right \}.
\end{align}

In what follows we give sufficient conditions for consistency of $\hat{\beta}$, following the arguments in the work of \cite{bickel2009simultaneous}. Before we state the proposition, we review the concept of restricted eigenvalues. 
\begin{definition}[Restricted eigenvalue condition]\label{def:RE}
A matrix $A$ is said to satisfy the restricted eigenvalue condition with parameters $(s,\kappa, \alpha)$, if for all $S \subset \{1,\dots, p\}$ with $|S| \leq s$ and all $\theta \in \mathbf{R}^p$ with $\|\theta_{S^c}\|_1 \leq \kappa \|\theta_S\|_1$ we have 
$\| A \theta\|_2 \geq \alpha \|\theta_S\|_2$. In this case we call $\alpha$ a restricted eigenvalue of $A$ (corresponding to parameters $(s,\kappa)$).

\end{definition}

\begin{prop}\label{prop1}
	Suppose that model (\ref{model}) holds with $x \sim N(0,\Sigma)$ and let $\sigma_x^2 = 4\max_j \Sigma_{jj}$. Assume that
	\begin{enumerate}
		\item $z_i = y_i - \langle x_i, \beta \rangle$ is subgaussian with $\|z_i\|_{\psi_2} \leq \sigma_z$,
		\item $\beta$ is a $s$-sparse vector, i.e. $|\{j: \beta_j \neq 0\} |\leq s$,
		\item $\Sigma^{\frac{1}{2}}$ satisfies the restricted eigenvalue condition with parameters $(s,9, 2\alpha)$ for some $\alpha > 0$, and that $\alpha$ and $\lambda_{\max}(\Sigma)$ are bounded away from $0,\infty$.
	\end{enumerate}
	Then there exists an absolute constant $c_0 > 0$ such that for $\lambda > c_0 \sigma_z \sigma_x \sqrt{\log(p)/n}$ and $n \geq c_0 (1 \lor \sigma_x^4) s \log(p/s)$ we have
	\begin{align*}
	\|\hat{\beta} - \beta \|_1 &\leq \frac{12 s \lambda}{\alpha^2},  \\
	\| \hat{\beta} - \beta \|_2 &\leq \frac{3 \sqrt{s}\lambda}{\alpha^2},\\
	 \frac{\|X(\hat{\beta} - \beta )\|_2}{\sqrt{n}} &\leq \frac{3 \sqrt{s} \lambda}{\alpha},
	\end{align*}
with probability no less than $1 - 2p^{-1}- \exp(-c_0 n^2/\sigma_x^4)$.
\end{prop}

\begin{remark}
 Assumptions (2) and (3) in Proposition \ref{prop1} are standard  for the consistency of the lasso estimator, see the work of \cite{bickel2009simultaneous}. Assumption (1) is the analogue of subgaussian errors in linear models and is satisfied whenever $y_n$ is subgaussian. This is a strong assumption on the approximation error $z_i$, but is nevertheless satisfied in several interesting cases such as when $y_i$ is bounded almost surely by a constant or when the model can be written as $y_i = g(\langle x_i , \tau \rangle) + e_i$, where $g$ is an unknown Lipschitz function and $e_i$ is a mean-zero subgaussian error independent of $x_i$.
\end{remark}

\section{Towards More Efficient Inference}
\label{Efficient-Inference} 

In previous sections, a linear approximation of the link function was used to obtain estimates of $\beta_1$. While this approach avoids the estimation of the link function, the (scaled) variance of the resulting estimator, $\nu^2 = \E [r_n^2 (y_n - \langle x_n, \beta \rangle)^2] / (\E[ r_n^2])^2$, depends heavily on the quality of this linear approximation. 

Let us write $\E[y_n \mid x_n = x] = g(\langle x, \tau \rangle)$ and $e_n = y_n - g(\langle x_n, \tau \rangle)$, so that $\E[e_n\mid x_n]= 0$. 
In this section we show how, in the known $\Sigma$ regime and under smoothness assumptions on the link function $g$, we can go beyond a linear approximation and obtain more efficient estimators of $\beta_1$. To this end, we use an expansion of the link function in terms of Hermite polynomials, as the latter form an orthonormal basis of the Hilbert space $L^2(\mathbf{R}, N(0,1))$ and are thus particularly useful in our setting.

The use of Hermite polynomials is readily motivated once we write $\langle x_n, \beta \rangle = \mu \langle x_n, \tau \rangle = \mu \cdot h_1(\langle x_n, \tau \rangle)$, where $h_1(\xi) = \xi$ is the first-order Hermite polynomial, and $\mu$ is by definition
\begin{align*}
\mu = \E [y_n \langle x_n, \tau \rangle] = \E [g(\xi) h_1(\xi)], \quad \xi \sim N(0,1).
\end{align*}
Thus $\mu$ is the inner product of $g$ and $h_1$ in $L^2(\mathbf{R}, N(0,1))$, and the debiasing procedure of Section \ref{Gaussian-design} uses only the projection of the link function on $h_1$ to linearize the model. 

Assume that $g$ can be expanded as $g(\xi) = \sum_{j=0}^\infty \mu_j h_j(\xi)$, where $h_j$ is the normalized Hermite polynomial of $j$-th degree:
\begin{align}\label{hermite_def}
h_j (\xi) = \frac{(-1)^{j}}{\sqrt{j!}} e^{\frac{\xi^2}{2}} \frac{d^j}{d\xi^j} e^{-\frac{\xi^2}{2}}.
\end{align}
The rest of this section considers an estimator of the form
\begin{align*}
	\tilde{\beta}_1 = \hat\beta_1 + \frac{\sum_{i=1}^n r_i(y_i - \hat{g}_m(\langle x_i, \hat{\tau}\rangle))}{\sum_{i=1}^n r_i x_{i1}},
\end{align*}
where $\hat{g}_m = \sum_{j=0}^{m} \hat{\mu}_j h_j$ is a higher order estimate of $g$.

\begin{figure}[t]
\subfloat[Distribution of $\tilde{\beta}_1 - \beta_1$]{\includegraphics[width = 3in, height = 2.1in]{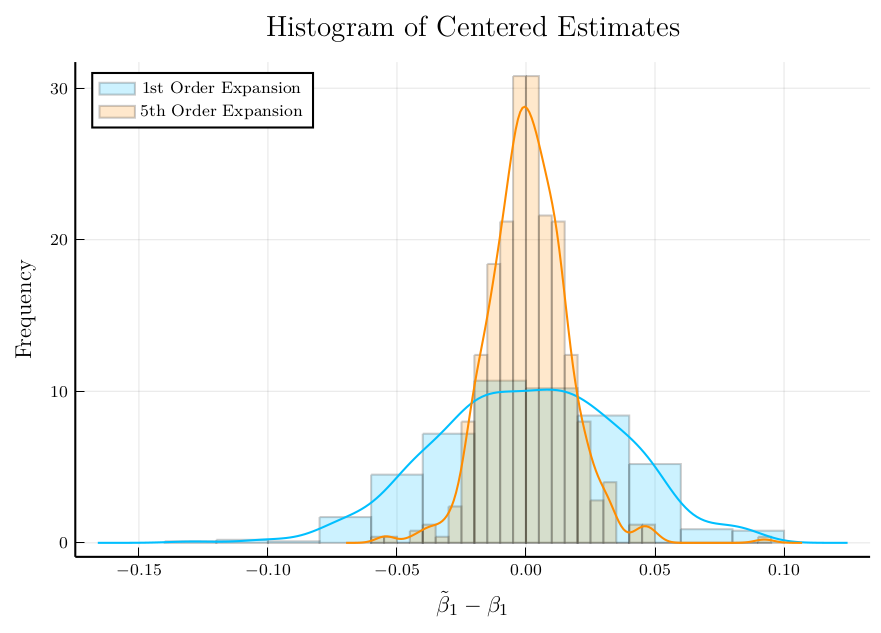}} 
\subfloat[A cubic link function and its estimate]{\includegraphics[width = 3in, height = 2.1in]{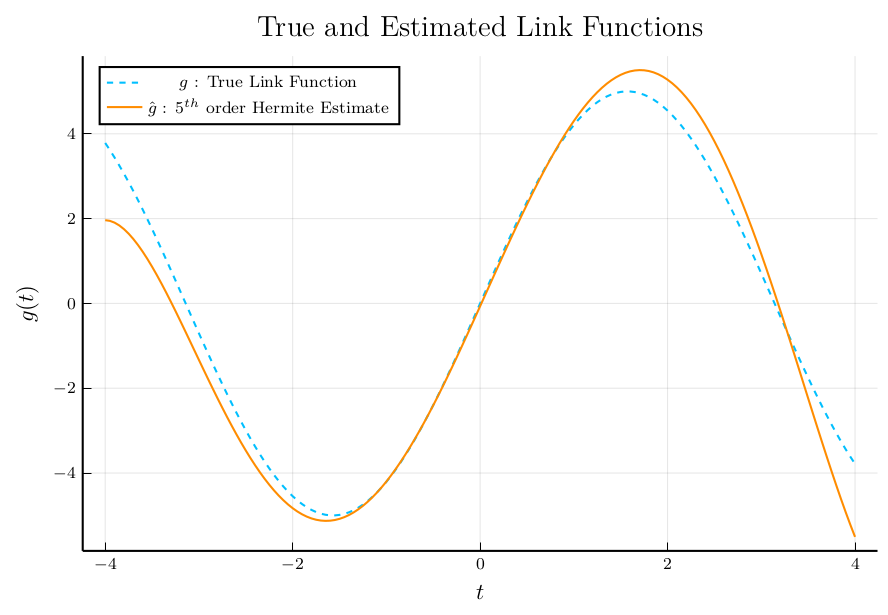}}
\caption{Debiasing with Hermite polynomial expansions with link function $g(t)=5\sin(t)$. (a). The distribution of debiased estimator centered at $\beta_1$ is approximately Gaussian and centered at zero. The use of a 5th order expansion leads to significant improvement in variance. (b). The 5th order Hermite estimate of the link function (n = 2000, p = 3000).}
\label{hist_link}
\end{figure}

In order to simplify notation, assume that we have a sample $(x_i, y_i)_{i=1}^{2n}$ of size $2n$. For a given $m$, we compute $\{\hat{\mu}_j\}_{j=1}^m$ and $\hat{\tau}$ on $\mathcal{S}_2 = (x_i, y_i)_{i=n+1}^{2n}$ as in Algorithm \ref{alg:efficient-known}:

\begin{algorithm}[t]
\SetAlgoLined
\KwData{An iid sample $(x_i, y_i)_{i=1}^{2n}$} \KwIn{Tuning parameters $m,\lambda_\text{pilot}, $}
\begin{enumerate}
	\item Set $\hat{\beta} := \arg\min_{\beta'} \left\{ \lfloor n/2 \rfloor^{-1}    \sum_{i=n+1}^{n + \lfloor n/2 \rfloor} (y_i - \langle x_i, \beta' \rangle)^2 +\lambda_\text{pilot} \| \beta' \|_1 \right\}$.	\item Define $\hat{\mu}_1 := \|\Sigma^{\frac{1}{2}} \hat{\beta}\|_2$ and $\hat{\tau} := \hat{\mu}_1^{-1} \hat{\beta}$.
	\item For $0 \leq j \leq m$ and $ j \neq 1$ set 
	$\hat{\mu}_j := \frac{1}{\lceil n/2 \rceil} \sum_{i= n + \lfloor n/2 \rfloor +1 }^{2n} y_i h_j(\langle x_i, \hat{\tau} \rangle)$.
	\item \KwRet
	\begin{align}\label{efficient-def}
	\tilde{\beta}_1 := \hat\beta_1 + \frac{\sum_{i=1}^n r_i(y_i - \sum_{j=0}^m \hat{\mu}_j h_j(\langle x_i, \hat{\tau} \rangle))}{\sum_{i=1}^n r_i x_{i1}}\,.
	\end{align}
	
\end{enumerate}
 \caption{Hermite Estimator with Known $\Sigma$}\label{alg:efficient-known}
\end{algorithm}

%

\begin{theorem}\label{thm-efficient}
	Suppose that $(x_i, y_i)_{i=1}^{2n}$ are i.i.d. observations from the model
	\begin{align*}
	y_i = g(\langle \tau, x_i \rangle) + e_i, \quad x_i \sim N_p(0,\Sigma), \quad \E[e_i | x_i] = 0.
	\end{align*}
	Let $m = \lfloor \log^\frac{2}{3}(n) \rfloor$ and suppose that $\hat{\beta}, \hat{\tau}, \{\hat{\mu}_j\}_{j=0}^m$ are computed as in Algorithm \ref{alg:efficient-known}. Assume also that the following conditions are satisfied:
	\begin{enumerate}
		\item There exist $0 < c, C < \infty$ such that $c \leq \lambda_{\min}(\Sigma) \leq \lambda_{\max}(\Sigma) \leq C$. 
		\item There exists a constant $c' > 0$ such that $|\mu_1| = |\E [y_n \langle x_n, \tau \rangle]| > c'$ for all $n \geq 1$.
		\item The response $y_n$ has a finite fourth moment: $\E[ y_n^4] < C_y^4$.
		\item The link function $g$ is differentiable with
		\begin{align*}
		\|g'\|_{L_2}^2 = \underset{\xi \sim N(0,1)}{\E}| g'(\xi)|^2 < L^2 < \infty,
		\end{align*}
		for a constant $L$ not depending on $n$.
		\item The sparsity of $\beta$ satisfies $s = o(\frac{n}{\log^2(p)})$.
		\item $\E |e_n|^{2 + \alpha} < M < \infty$ for some $\alpha, M>0$ and all $n \geq 1$.
		\item The pilot estimator $\hat{\beta}$ satisfies
		\begin{align*}
		    \|\hat{\beta} - \beta \|_2 \lesssim \sqrt{\frac{s \log(p)}{n}} \quad \text{ with probability } 1-o(1).
		\end{align*}
	\end{enumerate}
	Then
	\begin{align*}
	\sqrt{n}(\tilde{\beta}_1 - \beta_1) = \nu \cdot \Xi_n, \quad \text{ where } \Xi_n \rightarrow_d N(0,1), \text{ and } \nu^2 = \frac{\E[ r_n^2 e_n^2]}{(\E[ r_n^2])^2}.
	\end{align*}
\end{theorem}

\begin{remark}
	1. Inspecting the proof of Theorem \ref{thm-efficient} shows that the argument in remark (\ref{sample-splitting}) applies here as well, so that changing the role of the two subsamples $\mathcal{S}_1$ and $\mathcal{S}_2$ leads to efficient use of the full sample.
	\\
	2. Our second assumption requiring that $|\mu_1|$ be bounded away from zero ensures that $\hat{\tau} := \hat{\mu}_1^{-1} \hat{\beta}$ enjoys similar consistency properties as $\hat{\beta}$. Even though our proof breaks down if $\mu_1 = o(1)$, other pilot estimates of $\hat{\tau}$ exist that do not require a non-vanishing $\mu_1$. An interesting example is the work of \cite{dudeja2018learning} proposing to estimate $\tau$ by maximizing $\hat{F}_l (\tau') = n^{-1}\sum_1^n y_i h_l(\langle \tau', x_i \rangle)$ over the unit sphere for an appropriately chosen $l \in \{1,2,\dots\}$. The consistency of this estimator instead requires $\mu_l \neq 0$ since $\E[ \hat{F}_l(\tau') ] = \mu_l \langle \tau, \tau' \rangle^l$ when $\Sigma = I$ (see the identity (\ref{hermite_inner_prod}) below for inner products of Hermite polynomials).
\end{remark}

\begin{remark}[Efficiency]
1. Using $y_n = g(\langle x_n, \tau \rangle) + e_n$ and $\E[e_n | x_n] = 0$ we can write
\begin{align*}
\E[ r_n^2 z_n^2] = \E[ r_n^2(y_n - \langle x_n, \beta \rangle)^2] = \E [r_n^2( g(\langle x_n, \tau \rangle) - \langle x_n, \beta \rangle)^2] + \E [r_n^2 e_n^2] \geq \E [r_n^2 e_n^2],
\end{align*}
with equality happening if and only if $\E [r_n^2( g(\langle x_n, \tau \rangle) - \langle x_n, \beta \rangle)^2] = 0$. This shows that the estimator in Theorem \ref{thm-efficient} is strictly more efficient than the one in Theorem \ref{thm_known_sig}, unless $g(\langle x_n, \tau \rangle) = \mu \cdot \langle x_n, \tau\rangle$ almost everywhere, in which case the asymptotic variances are equal.

2. (Lower bound on variance reduction). Let $\tilde{\beta}_1^L$ denote the debiased estimate obtained by linearization (as in Theorem \ref{thm_known_sig}) and $\tilde{\beta}_1^H$ be the estimate from Hermite expansion. Also assume that $g,g',g'' \in L^2(N(0,1))$.  The reduction in asymptotic variance (scaled by $n$) is given by
\begin{align*}
n (\V[\tilde{\beta}_1^L] - \V[\tilde{\beta}_1^H]) \approx  \sum_{\stackrel{ j=0}{ j \neq 1}}^\infty \left[ \frac{\mu_{j}^2}{\E[r_n^2]} + 2\tau_1^2 \left( j \mu_{j}^2 +   \sqrt{(j+1)(j+2)} \mu_j \mu_{j+2} \right) \right].
\end{align*}
A proof is given in Proposition \ref{variance_reduction} in the appendix. To gain some intuition, consider the special case where $\Sigma = I$ so that $\|\tau \|_2 = \E[r_n^2] = 1$. Then simple algebra after using the inequality $2 \sqrt{(j+1)(j+2)} \mu_j \mu_{j+2} \geq - (j+1) \mu_j^2 - (j+2) \mu_{j+2}^2 $ shows that the right-hand side of the above display is lower bounded by 
\begin{align*}
(1- \tau_1^2) \sum_{j=0, j \neq 1}^\infty \mu_j^2 = (1- \tau_1^2) \E[(g(\langle x, \tau \rangle) - \langle x, \beta \rangle)^2].
\end{align*}

\end{remark}

	\textbf{The growth rate of $m$.} In Theorem \ref{thm-efficient} we let the number of basis functions $\{h_j\}_{j=0}^m$ grow with $n$ at the rate $m = \lfloor \log^{\frac{2}{3}}(n) \rfloor$. The theorem continues to remain valid for slower rates of growth of $m$, as long as $m \rightarrow \infty$. The slow growth rate of $m$ is used to bound the variance of $\hat{\mu}_j$ via the fourth moment of $h_j(\langle \tau, x \rangle)$ which is exponential in $j$ (see Proposition \ref{hermite-prop}). While this rate of growth guarantees improvement in the $n\rightarrow \infty$ asymptotic regime, providing lower/upper bounds on the MSE in terms of $m$ for a fixed sample size $n$ is more challenging as these bounds would depend on the Hermite coefficients of the particular link function. To see this, observe that in Figure \ref{MSE}, the MSE (and similarly, the variance) initially oscillates on odd/even ordered Hermite expansions. In the particular case of $g(t) = 5 \sin(t)$ this can be explained by the fact that $\sin(t)$ is an odd function, so the even ordered Hermite coefficients are all zero. Thus adding even-ordered terms, for a fixed sample size, only adds noise without explaining any of the variance.
\newline
For implementation in practice, one can use a jackknife estimate of the variance to find the $m$ that minimizes the variance for a given sample size $n$. Note that, since the bias is typically negligible compared to the standard error (see Tables \ref{hermite_accuracy_known} and \ref{hermite_accuracy_unknown}), the Mean Squared Error (MSE) is largely determined by the variance so that minimizing the variance will also minimize the MSE. Figure \ref{MSE} shows agreement between the optimal $m$ for the MSE and its jackknife based estimate. 

\begin{figure}[!ht]
\subfloat[Mean Squared Error of Debiased Estimates]{\includegraphics[width = 3in, height = 2.1in]{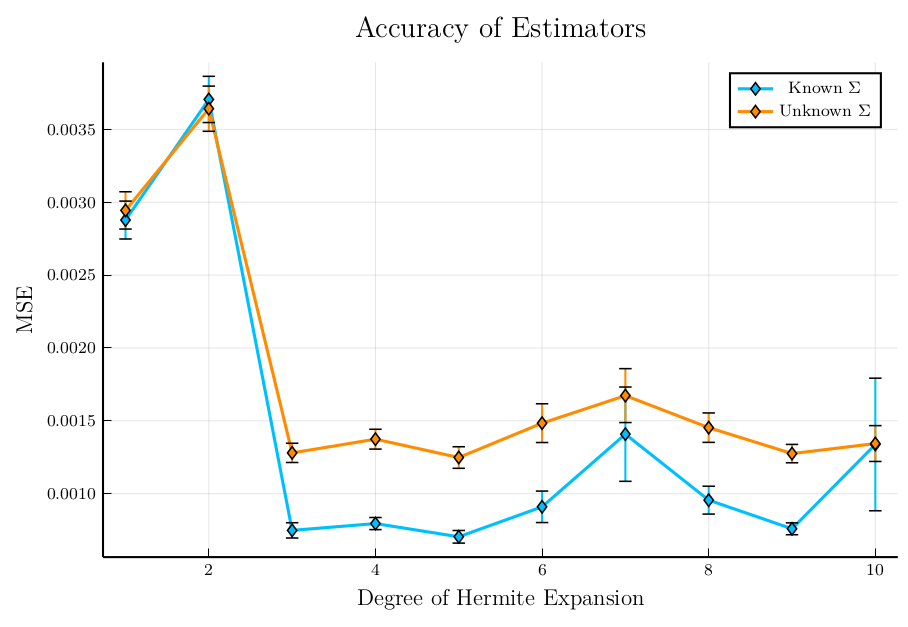}} 
\subfloat[Jackknife Estimates of Standard Error]{\includegraphics[width = 3in, height = 2.1in]{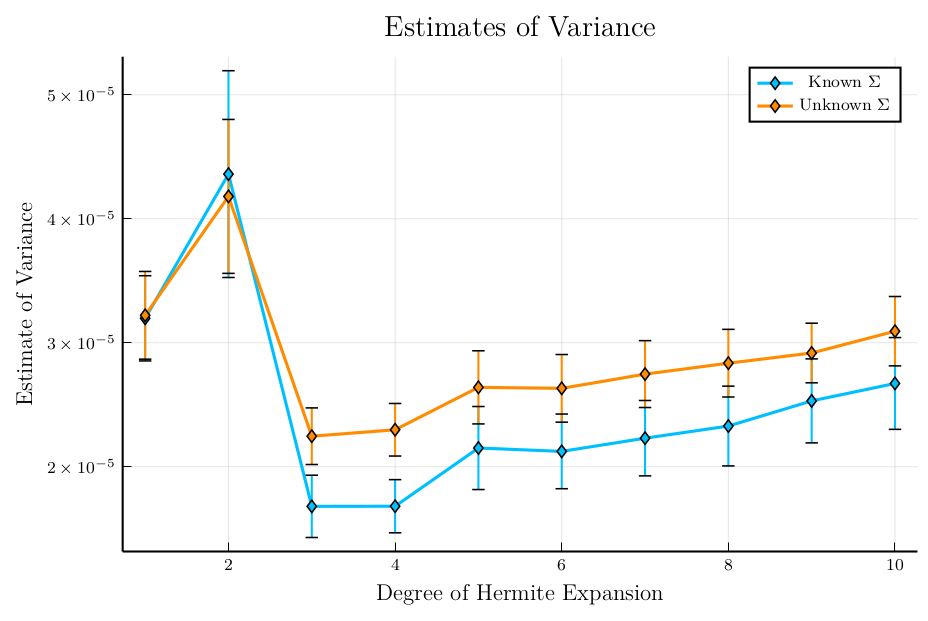}}
\caption{(a). The effect of higher order expansions on the Mean Squared Error (MSE). For each degree of expansion, the empirical MSEs were computed based on 1000 replications and are shown by the solid lines while the respective error bars denote their empirical standard errors.  (b). The Jackknife (leave-$k$-out) estimates of variance (up to a factor that does not depend on $m$) can be used to tune the order $m$ of the Hermite expansion. In this case, the solid lines depict the mean (computed over 10 realizations) of the jackknife estimate of the variance of the debiased estimator, and the error bars as before reflect the standard error of this quantity.}
\label{MSE}
\end{figure}

\subsection{Inference when $\Sigma$ is unknown} The estimator proposed in the previous section uses $\Sigma$ and $\gamma$ in the definition of $\hat{\mu}_1$ and $r_i$ respectively. When $\Sigma$ is unknown, it is natural to instead set $\hat{\mu}_1 := \|\hat{\Sigma}^\frac{1}{2} \hat{\beta}\|_2$. Similarly, one can define $r_i = x_{i,1} - \langle \hat{\gamma} , x_{i,-1} \rangle$ where $\hat{\gamma}$ is obtained using node-wise lasso as described in Section 2.2. 

The main difficulty in the case of unknown $\Sigma$ is that it is not possible to exactly standardize $\langle \hat{\tau}, x \rangle $ to have a $N(0,1)$ distribution. This is important in our analysis of the bias and variance of the estimator because the Hermite polynomials are only orthogonal with respect to standard normal random variables. More specifically, in the case of known $\Sigma$ we heavily use the following identity of \cite{o2014analysis} which is only valid for unit vectors\footnote{When $x \sim N_p(0,\Sigma)$ we use the identity with $\alpha = \Sigma^\frac12 \tau$ and $\hat{\alpha}=\Sigma^\frac12 \hat{\tau}$. Then $\langle \tau, x \rangle = \langle \alpha, \Sigma^{-\frac12} x \rangle \sim N(0,1)$ since by assumption $\Sigma^\frac12 \tau$ is a unit vector. On the other hand, $\langle \hat{\tau}, x \rangle$ has a standard normal distribution only if $\hat{\alpha} = \Sigma^\frac12 \hat{\tau} $ is a unit vector, a normalization that requires $\Sigma$ to be known.}
 $\alpha, \hat{\alpha} \in \mathbf{R}^p$:
\begin{align}\label{hermite_inner_prod}
\E[h_j( \langle \alpha, x \rangle ) h_k(\langle \hat{\alpha} , x \rangle)] = \langle \alpha, \hat{\alpha} \rangle^j \mathbf{1}[j = k] ,\quad \text{ for } x \sim N_p(0, I). 
\end{align}
Our Proposition \ref{non-unit-hermite-prop} in the appendix generalizes this identity to non-unit vectors $\alpha, \hat{\alpha}$ as follows
\begin{align*}
\E[ h_j( \langle \alpha, x \rangle ) h_k( \langle \hat{\alpha}, x \rangle ) ] = \sqrt{j! k! }  \sum_{ \stackrel{ 0 \leq i \leq \min \{ j,k \} }{ j \equiv_2 i \equiv_2 k}} \frac{ \langle \alpha, \hat{\alpha} \rangle^i \left(\frac{\| \alpha \|_2^2 - 1}{2}\right)^\frac{j-i}{2}\left(\frac{\| \hat\alpha \|_2^2 - 1}{2}\right)^\frac{k-i}{2}}{i! (\frac{j-i}{2})! (\frac{k-i}{2})! },
\end{align*}
where $x \sim N_p(0, I)$ and $i \equiv_2 j$ means that $i-j$ is an even integer. As suggested by the above display, using $\hat{\Sigma}$ to studentize $\langle \hat{\tau}, x \rangle$ leads to non-orthogonality of the Hermite polynomials, which in turn introduces extra bias terms in the error decomposition of the debiased estimator. In order to control these extra errors a slower rate of growth of $m=\lfloor \log^\frac{1}{4}(n) \rfloor$ (compared with $m = \lfloor \log^\frac{2}{3}(n) \rfloor$ in Theorem \ref{thm-efficient})  is assumed in Theorem \ref{thm-efficient-unknown}.

\begin{algorithm}[t]
\SetAlgoLined
\KwData{An iid sample $(x_i, y_i)_{i=1}^{2n}$} \KwIn{Tuning parameters $m,\lambda_\text{pilot}, \lambda_\text{node}$}
 \begin{enumerate}
 	\item Set $\check{\beta} := \arg\min_{\beta'} \left\{ (2n)^{-1}    \sum_{i=1}^{n} (y_i - \langle x_i, \beta' \rangle)^2 +\lambda_\text{pilot}' \| \beta' \|_1 \right\}$
 	\item Set $\hat{\beta} := \arg\min_{\beta'} \left\{ (2\lfloor n/3 \rfloor)^{-1}    \sum_{i=n+1}^{n + \lfloor n/3 \rfloor} (y_i - \langle x_i, \beta' \rangle)^2 +\lambda_\text{pilot} \| \beta' \|_1 \right\}$.
	 \item Set $\hat{\gamma} := \arg\min_{\gamma'} \left\{  \lfloor n/3 \rfloor^{-1}   \sum_{i= n + \lfloor n/3 \rfloor + 1}^{n + 2\lfloor n/3 \rfloor} (x_{i1} - \langle x_{i,-1}, \gamma' \rangle)^2 +\lambda_\text{node} \| \gamma' \|_1 \right\}$.
	\item Set $\hat{\Sigma} := \lfloor n/3 \rfloor^{-1} \sum_{i= n + \lfloor n/3 \rfloor + 1}^{n + 2\lfloor n/3 \rfloor} x_i x_i^T$.
	\item Set $\hat{\mu}_1 := \| \hat{\Sigma}^{\frac{1}{2}} \hat{\beta}\|_2$ and $\hat{\tau} := \hat{\mu}_1^{-1} \hat{\beta}$.
	\item For $0 \leq j \leq m$ and $ j \neq 1$ set
	$\hat{\mu}_j :=  \sum_{i= n +2 \lfloor n/3 \rfloor +1 }^{2n} y_i h_j(\langle x_i, \hat{\tau} \rangle) / (n -2 \lfloor n/3 \rfloor) $.
	\item \KwRet
	$\label{efficient-def-unknown}
	\tilde{\beta}_1 := \check\beta_1 + \frac{\sum_{i=1}^n \hat{r}_i(y_i - \langle \check{\beta} - \hat{\beta} , x_i \rangle - \sum_{j=0}^m \hat{\mu}_j h_j(\langle x_i, \hat{\tau} \rangle))}{\sum_{i=1}^n \hat{r}_i x_{i1}}\,$.
\end{enumerate}
 \caption{Hermite Estimator with Unknown $\Sigma$}\label{alg:efficient-unknown}
\end{algorithm}

\begin{theorem}\label{thm-efficient-unknown}
	Suppose that $(x_i, y_i)_{i=1}^{2n}$ are i.i.d. observations from the model
	\begin{align*}
	y_i = g(\langle \tau, x_i \rangle) + e_i, \quad x_i \sim N_p(0,\Sigma), \quad \E[e_i | x_i] = 0.
	\end{align*}
	Let $m = \lfloor \log^\frac{1}{4}(n) \rfloor$ and suppose that $\hat{\beta}, \hat{\tau}, \{\hat{\mu}_j\}_{j=0}^m$ are computed as in Algorithm \ref{alg:efficient-unknown}. Assume also that the following conditions are satisfied:
	\begin{enumerate}
		\item There exist $0 < c, C < \infty$ such that $c \leq \lambda_{\min}(\Sigma) \leq \lambda_{\max}(\Sigma) \leq C$. 
		\item There exists a constant $c' > 0$ such that $|\mu_1| = |\E [y_n \langle x_n, \tau \rangle]| > c'$ for all $n \geq 1$.
		\item The link function $g$ is differentiable with
		\begin{align*}
		\|g'\|_{L_2}^2 = \underset{\xi \sim N(0,1)}{\E}| g'(\xi)|^2 < L^2 < \infty,
		\end{align*}
		for a constant $L$ not depending on $n$.
		\item We have $s := \|\beta\|_0 \lor \|\gamma\|_0 = o(\frac{\sqrt{n}}{\log(p)})$.
		\item The estimate $\hat{\gamma}$ of $\gamma$, satisfies
		\begin{align*}
		\p \left(\|\hat{\gamma} - \gamma \|_1 \leq c_\gamma s\sqrt{\frac{ \log p }{n}}\right) \rightarrow 1, \text{ and } \,
		\p \left(\|\hat{\gamma} - \gamma \|_2 \leq c_\gamma' \sqrt{\frac{ s\log p }{n}}\right) \rightarrow 1,
		\end{align*}
		for constants $c_\gamma, c_\gamma'$ not dependent on $n$.

		\item $\E |e_n|^{2 + \alpha} < M < \infty$ for some $\alpha, M>0$ and all $n \geq 1$.
		\item $z_n$ is subgaussian with $\| z_n \|_{\psi_2} \leq \sigma_z$ for some $\sigma_z<\infty$ and all $n\geq 1$.	\end{enumerate}
	Then
	\begin{align*}
	\sqrt{n}(\tilde{\beta}_1 - \beta_1) = \nu \cdot \Xi_n + o_p(1), \quad \text{ where } \Xi_n \rightarrow_d N(0,1), \text{ and } \nu^2 = \frac{\E [r_n^2 e_n^2]}{(\E [r_n^2])^2} + \frac{\mu_3^2 \tau_1^2}{(\E[ r_n^2])^2}.
	\end{align*}
	
\end{theorem}

\begin{remark}
The variance of the $\tilde{\beta}_1$ in Theorem \ref{thm-efficient-unknown} has an extra term $\frac{\mu_3^2 \tau_1^2}{(\E[ r_n^2])^2}$ compared with the case of known $\Sigma$.  The inflated variance is due to the error in estimation of $\Sigma$ by $\hat{\Sigma}$. Since by assumption $\E[r_n^2] \geq \lambda_{\min}(\Sigma)$ is bounded away from zero, we obtain the same variance as in Theorem \ref{thm-efficient} if $\mu_3 \tau_1 = o( 1)$.
\end{remark}

\subsection{Simulations}

Next we describe our setup for the simulations\footnote{Julia code for these simulations is available at \url{https://github.com/ehamid/sim_debiasing}.}. We consider the effect of up to the tenth order Hermite expansions on the bias and variance of the debiased estimator. We take the nonlinear link function to be:
\begin{align*}
g(t) = 5\sin(t),
\end{align*}
and define $\tau$ by
\begin{align*}
\tau = \frac{\iota}{\sqrt{\iota^T \Sigma \iota}} 
\text{ where } 
\iota_i = \begin{cases}
	11 - i \, \text{ if } i \leq 10,\\
	0 \, \text { else,}
\end{cases}
\end{align*}
and $\Sigma_{ij} = 0.5^{|i - j|}$ for $1 \leq i,j \leq p = 2000$. A simple calculation shows that $\mu = 5/\sqrt{e}$ in this case, so that $\beta = 5 \tau / \sqrt{e}$.  For each one of 1000 Monte Carlo replications, $n = 1000$ observations were generated from $y = g(\langle x, \tau \rangle ) + 0.1 \varepsilon$ where $\varepsilon \sim N(0,1)$ and $x \sim N(0, \Sigma)$ and $\tilde{\beta}_1$ was computed as described in Section \ref{Efficient-Inference} for Hermite expansions of up to the tenth degree. Table \ref{hermite_accuracy_known} compares the bias, standard error and root mean squared error of these estimators while Table \ref{hermite_accuracy_unknown} shows the same information for the case of unknown $\Sigma$.

\textbf{Jackknife Estimate of Variance.} Given a sample of size $n = 1000$, we use a leave-$k$-out procedure (for $k = 10$) to obtain an estimate of the variance of $\tilde{\beta}_1$ as follows. For each $j = 1, \dots, 100$, we leave out the observations indexed by $10(j-1)+1, \dots, 10(j-1)+10$, and compute the debiased estimate on the remaining observation to obtain $\tilde{\beta}_1^{(j)}$. The estimate of variance of $\tilde{\beta}_j$ is then (up to a scaling factor that does not depend on $m$):
\begin{align*}
    \widehat{Var}(\tilde{\beta}_1) = \frac{1}{100}\sum_{j=1}^{100} (\tilde{\beta}_1^{(j)} - \bar{\beta}_1)^2,\quad  \text{ where } \quad \bar{\beta}_1 = \frac{1}{100} \sum_{j=1}^{100} \tilde{\beta}_1^{(j)}.
\end{align*}
Figure \ref{MSE}.(b) shows the relationship between these estimates and the order of Hermite expansion $m$.

\begin{table}[!htbp]
\setlength\tabcolsep{4pt}
\renewcommand{\arraystretch}{1.3}
	\small
	\centering
	\begin{tabular}{|cl|cccccccccc|}
		\hline
		\multicolumn{2}{|l|}{\diagbox[]{\scriptsize Measure}{ \scriptsize Degree}} & 1 & 2 & 3 & 4 & 5 & 6 & 7 & 8 & 9 & 10 \\ 
  \hline
Bias &  &-0.052 &0.014 &0.022 &-0.031 &0.004 &0.012 &0.03 &0.013 &-0.007 &0.059\\
Std. Error & & 1.696 &1.925 &0.865 &0.891 &0.839 &0.954 &1.186 &0.977 &0.871 &1.155 \\
RMSE & & 1.696 &1.925 &0.865 &0.892 &0.839 &0.954 &1.187 &0.977 &0.871 &1.157 \\
   \hline
	\end{tabular}
	\caption{Accuracy of $\tilde{\beta}_1$ with Hermite Expansions of orders 1 to 10 for Debiasing when $\Sigma$ is known. (The numbers have been scaled by $\sqrt{n}$ for better readability.)}
	\label{hermite_accuracy_known}
\end{table}

\begin{table}[!htbp]
\setlength\tabcolsep{4pt}
\renewcommand{\arraystretch}{1.3}
	\small
	\centering
	\begin{tabular}{|cl|cccccccccc|}
		\hline
		\multicolumn{2}{|l|}{\diagbox[]{\scriptsize Measure}{\scriptsize Degree}} & 1 & 2 & 3 & 4 & 5 & 6 & 7 & 8 & 9 & 10\\ 
  \hline
Bias &  & 0.176 &-0.06 &-0.115 &-0.183 &-0.097 &-0.12 &-0.097 &-0.134 &-0.121 &-0.077\\
Std. Error & & 1.707 &1.907 &1.126 &1.158 &1.113 &1.212 &1.29 &1.198 &1.123 &1.157\\
RMSE & & 1.716 &1.908 &1.131 &1.172 &1.117 &1.218 &1.293 &1.206 &1.129 &1.159 \\
   \hline
	\end{tabular}
	\caption{Accuracy of $ \tilde{\beta}_1$ with Hermite Expansions of orders 1 to 10 for Debiasing when $\Sigma$ is unknown. (The numbers have been scaled by $\sqrt{n}$ for better readability.)}
		\label{hermite_accuracy_unknown}

\end{table}

Appendix \ref{simulations} provides more simulations that illustrate the coverage of confidence intervals and the error rates of tests constructed using the results in this paper.



\section{Conclusions}
We have shown that in the generic single-index model, it is easy to obtain $\sqrt{n}$-consistent estimators of finite-dimensional components of $\beta$ in the high-dimensional setting using a procedure that is perfectly agnostic to the link function, provided we have a Gaussian (or more generally, elliptically symmetric) design. Even though this rate can be achieved under minimal assumptions on the link function, we also showed that using an estimate of the link function to refine the debiased estimator enhances efficiency. Some words of caveat are in order. First, the independence of $f$ and $x$ is critical to our development. Indeed, if $f$ depends upon $x$, there is no guarantee that one can estimate individual co-efficients at $\sqrt{n}$ rate. As an example, consider the binary choice model $\Delta = 1(\beta^T X + \epsilon > 0)$ where $\epsilon$ given $X$ depends non-trivially on $X$ \citep{manski1975maximum,manski1985semiparametric}. The recent results in the work of \cite{mukherjee2019non} (see Theorem 3.4) imply that when $p$ is fixed, the co-efficients of $\beta$ can be estimated at a rate no faster than $n^{1/3}$, with the maximum score estimator of Manski attaining this rate \citep{kim1990cube}. It is clear that we should not expect the debiasing approach of our paper to work in this model. Second, if $f$ is independent of $X$ and discontinuous, e.g. $f(t) = 1(t > 0) + \epsilon$ where $\epsilon > 0$, this becomes a multi-dimensional change-point problem (a change-plane problem to be precise) and the work of \cite{wei2018change} implies that the co-efficients are estimable even at rate $n$ (for the fixed $p$-case), and the $\sqrt{n}$ rate derived in this paper is sub-optimal. These two examples serve to illustrate the fact that while the debiasing scheme is attractive, it can fail under model-misspecification, and may not produce optimal convergence rates in certain cases. 
\newline
Of course the $\sqrt{n}$ rate will be typically optimal when $f$ is sufficiently smooth, e.g. $f(t) = P(t) + \epsilon$ where $\epsilon$ is independent of $X$ and $P$ is a polynomial of fixed degree. In this case, the debiased estimator has to be rate-optimal since even if we knew $P$ there is no way we can estimate the coefficients at a rate faster than $\sqrt{n}$. 

Several interesting questions remain open for future research. First, inference for high-dimensional single-index models beyond elliptically symmetric designs remains to be fully explored. Second, a finite-sample analysis of our estimator based on Hermite polynomials that elucidates the relationship between the order of expansion and the MSE would be useful in practice. Third, extending this estimator using more general bases than the Hermite polynomial basis for elliptically symmetric designs (alluded to at the end of Appendix \ref{Elliptical-Design}) remains to be studied.

\acks{We thank the anonymous referees for several helpful comments and corrections. This work was done under the auspices of NSF Grant DMS-1712962.}

\newpage

\appendix

\section{Inference for a General Elliptically Symmetric design} 
\label{Elliptical-Design} 
In this section, extensions of Proposition \ref{prop1}, Theorem \ref{thm_known_sig} and Theorem \ref{thm-unknown-sigma} to the more general setting of elliptically symmetric design are considered. We start by reviewing the definitions of elliptically symmetric and sub-gaussian vectors.

\begin{definition}
	A centered random vector $x \in \mathbf{R}^p$ follows an elliptically symmetric distribution with parameters $\Sigma$ and $F_v$ if 
	\begin{align*}
	x \stackrel{d}{=} v B U,
	\end{align*}
	where the random variable $v \in \mathbf{R}$ has distribution $F_v$, the random vector $U\in \mathbf{R}^p$ is uniformly distributed over the unit sphere $\mathbb{S}^{p-1}$ and is independent of $v$, and $B$ is a matrix satisfying $\Sigma = BB^T$. In this case we write $x \sim \mathcal{E}(0, \Sigma, F_v)$.
\end{definition}

Note that the matrix $B$ and the random variable $v$ in the above definition are not uniquely determined. In particular, for any orthogonal matrix $Q$ and $ t > 0$, if the pair $(B, v)$ satisfies the definition then so does the pair $(t BQ, v/t)$. For comparability with the case of Gaussian random vectors, in this work we assume that in this representation $\E [v^2] = p$, so that the variance-covariance matrix of $x$ is equal to $\Sigma$, i.e. $\E [x x^T] = \Sigma$.

It is well-known that elliptically symmetric distributions generalize the multivariate normal distribution, and in particular, include distributions that have heavier or lighter tails than the normal distribution. More precisely, in the above definition, if $v = \sqrt{u}$ where $u \sim \chi_p^2$, then $\mathcal{E}(0, \Sigma, F_v) = N(0, \Sigma)$.

\begin{definition}
	A centered random vector $x \in \mathbf{R}^p$ is subgaussian with subgaussian constant $\sigma$ if for all unit vectors $u \in \mathbb{S}^{p-1}$ we have that $\langle u, x \rangle$ is a subgaussian random variable with $\|\langle u, x \rangle \|_{\psi_2} \leq \sigma$. In this case we write $\|x\|_{\psi_2} \leq \sigma$.
\end{definition}

Under an elliptically symmetric design $x_i \sim\mathcal{E}(0, \Sigma, F_v)$, the linear representation $y_i = \mu \langle x_i, \tau \rangle + z_i$ is still valid with $\E [z_i x_i] = 0$, when $\mu$ and $z_i$ are defined by
\begin{align*}
\mu &= \E [y_i \langle x_i, \tau \rangle], \\
z_i &= y_i - \mu \langle x_i, \tau \rangle,
\end{align*}
and we use the normalization $\| \Sigma^\frac{1}{2} \tau \|_2 = 1$. The argument for $\E [z_i x_i] = 0$ is exactly as in the case of Gaussian design, since, as far as the distribution of $x_i$ is concerned, the proof in Section \ref{orthoproof} only requires the normalization $\| \Sigma^\frac{1}{2} \tau \|_2 = 1$ and the fact that the conditional expectation of $x$ given $\langle x, \tau \rangle$ is linear in $\langle x, \tau \rangle$, that is, there exists a (non-random) vector $b$ such that $\E[x \mid \langle x, \tau \rangle] = \langle x, \tau \rangle b$. The latter property also holds for elliptically symmetric random vectors \cite[Corollary 2.1]{goldstein2018structured}.

Besides the orthogonality property $\E [z_i x_i] = 0$, our proofs rely on controlling the tail probabilities of certain random variables, such as $\| z^T X \|_\infty$ in Proposition \ref{prop1} and $\max_{j \neq k} | \sum_{i=1}^n r_i z_i|$ in Theorems \ref{thm_known_sig} and \ref{thm-unknown-sigma}. In addition, in the case of unknown $\Sigma$, subgaussian tails of $x_i$ were used to control the moments of $r_i - \hat{r}_i = \langle x_{i,-1}, \hat{\gamma} - \gamma \rangle$. The assumption of sub-gaussianity of $x_i$ allows the same proofs go through in the case of elliptically symmetric designs.



\begin{remark}
	A sufficient condition for $x \sim \mathcal{E}(0, \Sigma, F_v)$ to be subgaussian is that in the representation $x = v B U$ the random variable $v$ is subgaussian. This follows because for all unit vectors $w \in \mathbb{S}^{p-1}$,
	\begin{align*}
	 |\langle w, x \rangle| &= | \langle w, BU \rangle| \cdot |v| \\
	 &\leq \|B^T w\|_2 \cdot |v|\\
	 &\leq \sqrt {\lambda_{\max}(BB^T)} |v| , \quad \text{a.s.}.
	 \end{align*}
	 Thus $\langle w, x \rangle$ is subgaussian if $v$ is subgaussian. Moreover, in this work we assume that $\lambda_{\max}(\Sigma)$ is uniformly (in $n,p$) bounded above, which implies that up to an absolute constant $x$ and $v$ have the same subgaussian constant.
\end{remark}

With these definitions, Proposition \ref{prop1} and Theorems \ref{thm_known_sig}, \ref{thm-unknown-sigma} can be extended as follows.

\begin{prop}\label{prop1'}
	Let $\hat{\beta}$ be the penalized lasso estimator defined in (\ref{pilot-def}). Suppose that model (\ref{model}) holds with $x \stackrel{iid}{\sim} \mathcal{E}(0,\Sigma, F_v)$ and assume that
	\begin{enumerate}
	    \item $v$ is subgaussian with $\|\nu\|_{\psi_2} \leq \sigma_x$
		\item $z_i = y_i -\langle x_i, \beta \rangle$ is subgaussian with $\|z_i\|_{\psi_2} \leq \sigma_z$ for all $1 \leq i \leq n$,
		\item $\beta$ is a $s$-sparse vector, i.e. $|\{j: \beta_j \neq 0\} |\leq s$,
		\item $\Sigma^{\frac{1}{2}}$ satisfies the restricted eigenvalue condition with parameters $(s,9, 2\alpha)$ for some $\alpha > 0$, and that $\alpha$ and $\lambda_{\max}(\Sigma)$ are bounded away from $0,\infty$.
	\end{enumerate}
	Then there exists an absolute constant $c_0 > 0$ such that for $\lambda > c_0 \sigma_z \sigma_x \sqrt{\log(p)/n}$ and $n \geq c_0 (1 \lor \sigma_x^4) s \log(p/s)$ we have
	\begin{align*}
	\|\hat{\beta} - \beta \|_1 &\leq \frac{12 s \lambda}{\alpha^2},  \\
	\| \hat{\beta} - \beta \|_2 &\leq \frac{3 \sqrt{s}\lambda}{\alpha^2},\\
	 \frac{\|X(\hat{\beta} - \beta )\|_2}{\sqrt{n}} &\leq \frac{3 \sqrt{s} \lambda}{\alpha},
	\end{align*}
	with probability no less than $1 - 2p^{-1}- \exp(-c_0 n^2/\sigma_x^4)$.
\end{prop}

\begin{theorem}\label{thm_known_sig_ellip}
	Let $\tilde{\beta}_1$ be the estimator defined by (\ref{est-def-known}). Suppose $(x_i, y_i)_{i=1}^n$ follow the model (\ref{model}) with $x_i \stackrel{iid}{\sim} \mathcal{E}(0,\Sigma, F_\mu)$ and let 
	\begin{align*}
	\nu^2 =\frac{\E [r_n^2 z_n^2]}{(\E [r_n^2])^2}.
	\end{align*}
	Assume also that the following conditions are satisfied:
	
	\begin{enumerate}
		\item $\|x_i\|_{\psi_2} \leq \sigma_x$ with $\sigma_x$ uniformly (over $n$) bounded above.
		\item The sparsity of $\beta$ satisfies $s = o(\frac{n}{\log^2(p)})$.
		\item There exist $0 < c, C < \infty$ such that $c \leq \lambda_{\min}(\Sigma) \leq \lambda_{\max}(\Sigma) \leq C$. 
		\item  $\E|y_n|^{2 + \alpha} \leq M < \infty$ for some $\alpha, M > 0$ and all $n \geq 1$.
        \item The pilot estimator $\hat{\beta}$ satisfies
        \begin{align*}
            \|\hat{\beta} - \beta\|_2 \lesssim \sqrt{\frac{s \log(p)}{n}},\quad \text{ with probability } 1-o(1).
        \end{align*}
	\end{enumerate}	
	
	Then
	\begin{align*}
	\sqrt{n}(\tilde{\beta}_1 - \beta_1) = \nu \cdot \Xi_n, \quad \text{ where }\quad \Xi_n \rightarrow_d N(0,1).
	\end{align*}
	
\end{theorem}

	\begin{theorem}(Unknown $\Sigma$) 
	\label{thm-unknown-sig-ellip}
		Let $\tilde{\beta}_1$ be the estimator defined in (\ref{est-def-unknown}). Suppose $(x_i, y_i)_{i=1}^{2n}$ follow the model (\ref{model}) with $x_i \stackrel{iid}{\sim} \mathcal{E}(0,\Sigma, F_\mu)$ and assume that the following conditions are satisfied:
		
		\begin{enumerate}
			\item $\|x_i\|_{\psi_2} \leq \sigma_x$ with $\sigma_x$ uniformly (over $n$) bounded above.
		\item We have $s := \|\beta\|_0 \lor \|\gamma\|_0 = o(\frac{\sqrt{n}}{\log(p)})$.
			\item There exists an estimate $\hat{\gamma}$ of $\gamma$ that is independent of $(x_i, y_i)_{i=1}^n$ and satisfies
			\begin{align*}
			\p(\|\hat{\gamma} - \gamma \|_1 \leq c_\gamma s\sqrt{\frac{ \log p }{n}}) \rightarrow 1,
			\end{align*}
			for a constant $c_\gamma$ not dependent on $n$.
			\item There exist $0 < c, C < \infty$ such that $c \leq \lambda_{\min}(\Sigma) \leq \lambda_{\max}(\Sigma) \leq C$. 
		\item $z_n$ is subgaussian with $\| z_n \|_{\psi_2} \leq \sigma_z$ for some $\sigma_z<\infty$ and all $n\geq 1$.
			
		\end{enumerate}	
		
		Then
	\begin{align*}
	\sqrt{n}(\tilde{\beta}_1 - \beta_1) = \nu \cdot \Xi_n, \quad \text{ where }\quad \Xi_n \rightarrow_d N(0,1).
	\end{align*}
		
	\end{theorem}

\begin{remark}
    The remarks that followed Proposition \ref{prop1} and Theorems \ref{thm_known_sig}, \ref{thm-unknown-sigma} remain valid in the more general context of Proposition \ref{prop1'} and Theorems \ref{thm_known_sig_ellip} and \ref{thm-unknown-sig-ellip}. For the sake of brevity they are not restated here.
\end{remark}

\textbf{Extension of Theorem \ref{thm-efficient}}. In order to extend Theorem \ref{thm-efficient} to elliptically symmetric designs, one would first need to construct an orthogonal basis with respect to the distribution of $\langle x, \tau \rangle$. It is possible to adapt the Hermite polynomial basis to this setting (whenever $N(0,1)$ is absolutely continuous with respect to the distribution of $\langle x, \tau \rangle$) as follows. Let $f_e$ be the density of $\langle x, \tau \rangle$, for $ x \sim  \mathcal{E}(0,\Sigma, F_\mu)$ and let $\varphi$ be the density of the standard normal distribution. Then the functions 

\begin{align*}
h^{e}_j = h_j \cdot \sqrt{\frac{\varphi}{f_e}}
\end{align*}
are orthonormal with respect to the distribution of $\langle x, \tau \rangle$, that is,
\begin{align*}
\E[ h^{e}_j (\langle x, \tau \rangle) h^{e}_k(\langle x, \tau \rangle) ] = \mathbf{1}\{ j = k \}.
\end{align*}
Thus under an elliptically symmetric design one could proceed as in Section \ref{Efficient-Inference} and replacing $h_j$ with $h^{e}_j$ in the construction of the estimator. We leave the rigorous analysis of this extension to future work.

\section{(Tail bounds)} \label{tailbounds}
We collect in this appendix some facts about sub-Gaussian and sub-exponential random variables. Proofs can be found in chapter 2 of \citet{vershynin2018high}.

Denote by $\| \cdot \|_{\psi_2}$ and $\| \cdot \|_{\psi_1}$ the sub-gaussian and the sub-exponential norms, respectively.
\begin{prop}\label{sub-g-e}
	
	There exists an absolute constant $C > 0$ such that the following are true:
	
	\begin{enumerate}
		\item If $X \sim N(0, \sigma^2)$, then $\| X \|_{\psi_2} \leq C \sigma$.
		
		\item If $X$ is sub-guassian, then $X^2$ is sub-exponential and $\|X^2\|_{\psi_1} = \| X \|_{\psi_2}^2$.
		
		\item If $X,Y$ are sub-gaussian, then $XY$ is sub-exponential and $\| XY \|_{\psi_1} \leq \| X \|_{\psi_2}\cdot \| Y \|_{\psi_2}$.
		
		\item If $X$ is sub-exponential then $\|X - \E X\|_{\psi_1} \leq C \| X\|_{\psi_1}$.
		
		\item (Bernstein's Inequality). Let $x_1, \dots, x_n$ be independent, mean zero sub-exponential random variables and $a = (a_1, \dots, a_n) \in \mathbf{R}^n$. Then for every $t>0$ we have
		\begin{align*}
		\p\{ | \sum_{i=1} a_i x_i | \geq t \} \leq 2 \exp\left[ - c \min\left( \frac{t^2}{K^2\|a\|_2^2}, \frac{t}{K \|a\|_\infty} \right) \right],
		\end{align*} 	
		where $K = \max_i \| x_i\|_{\psi_1}$ and $c$ is an absolute constant.
	\end{enumerate}
\end{prop}

The following corollary will be used multiple times in the text.
\begin{corollary}\label{subexp-cor}
	Suppose that $x_i \in \mathbf{R}^p$ are i.i.d. random vectors with $\|x_{ij}\|_{\psi_1} \leq \rho_x$ for $1 \leq i \leq n$ and $1 \leq j \leq p$. Assume also that $\log(p)/n \rightarrow 0$. Then
	\begin{enumerate}
		\item for any $ 1 \leq i \leq n$ and $1 \leq j \leq p$, the variable $x_{ij} - \E x_{ij}$ is sub-exponential with $\|x_{ij} - \E x_{ij} \|_{\psi_1} \leq C \rho_x$, for some absolute constant $C$.
		
		\item 
		We have $\p( \max_{1 \leq j \leq p} |\frac{1}{\sqrt{n}}\sum_{i=1}^n x_{ij} - \E x_{ij}| < C \rho_x \sqrt{\log p} ) \rightarrow 1$ for an absolute constant $C > 0$ 
%
	\end{enumerate}
	
\end{corollary}
\begin{proof}
	1. This follows immediately from Proposition \ref{sub-g-e}.
	
	2. Apply Bernstein's inequality with $a_i = 1/\sqrt{n}$ and $t = \kappa \rho_x \sqrt{\log p}$ for a constant $\kappa$ that will be determined shortly. We obtain for each $1 \leq j \leq p$,
	\begin{align*}
	\p\left\{ \frac{1}{\sqrt{n}} \left|\sum_{i=1}^n x_{ij} - \E x_{ij} \right| \geq \kappa \rho_x \sqrt{\log p} \right\} \leq 2 e^{-c \min \left( \frac{\kappa^2 \rho_x^2 \log p}{C^2 \rho_x^2}, \frac{\kappa \rho_x\sqrt{\log p} \sqrt{n}}{C \rho_x} \right)}.
	\end{align*}
	
	In order to get sub-gaussian tail in Bernstein's inequality, we need 
	\begin{align*}
	\frac{\kappa^2 \log p}{C^2} \leq \frac{\kappa\sqrt{n} \sqrt{\log p}}{C},
	\end{align*}
	
	which is equivalent to $\sqrt{\log p} / \sqrt{n} \leq C / \kappa$, and holds for large enough $n$ (and any fixed value of $\kappa$) as in our asymptotic regime, $\sqrt{\log p} / \sqrt{n} \rightarrow 0$. For such large $n,p$, apply a union bound to the above inequality to get
	
	\begin{align*}
	\p\{ \frac{1}{\sqrt{n}} \max_{1 \leq j \leq p}|\sum_{i=1}^n x_{ij} - \E x_{ij}| \geq \kappa \rho_x \sqrt{\log p} \} &\leq 2p \exp\left[ -c \left( \frac{\kappa^2 \log p}{C^2} \right)  \right]\\
	&= 2 p \cdot p^{-\frac{c \kappa^2}{C^2}} \\
	&= 2 p ^{1 -\frac{c \kappa^2}{C^2} } \rightarrow 0,
	\end{align*} 
	as long as $\kappa^2 > C^2/c$, where $C, c$ are absolute constants.

\end{proof}
\section{}

\subsection{Orthogonality of $X$ and $z$}
\label{orthoproof} 
Let $\xi = \langle x_i, \tau \rangle \sim N(0,1)$. Gaussianity of $x_i$ implies that $\E[x_i \mid \xi] = \xi b$ for some (non-random) $b \in \mathbf{R}^p$. Using the definition of $z_i$ and the tower property of conditional expectations,
\begin{align*}
\E [z_i x_i] &= \E [(f_i(\xi) - \mu \xi) x_i] \\
&= \E[\E[(f_i(\xi) - \mu \xi) x_i \mid \xi, f_i]] \\
&= \E[(f_i(\xi) - \mu \xi)\E[x_i \mid \xi ]] \tag{$f_i \indep x_i$}\\
&= (\E[f_i(\xi)\xi] - \mu \E \xi^2 ) b^T\\
&= (\mu - \mu \cdot 1) b \\
&= 0.
\end{align*}

\subsection{Identifiability}\label{identifiability} In this appendix we discuss the identifiability of the parameters the model (\ref{model}). Even though the parameter $\tau$ is not identifiable (since its norm can be absorbed into the link function $f$), the parameter $\beta = \mu \tau$ is in fact identifiable when $\Sigma$ is non-singular. This follows since as shown before, $X$ and $z$ are uncorrelated, and so
\begin{align*}
\frac{1}{n}\E\|Y - X\beta'\|_2^2 &= \frac{1}{n}\E \|X\beta + z - X\beta'\|_2^2 \\
&= \frac{1}{n} \E \|X (\beta - \beta') \|_2^2 + \frac{1}{n} \E\|z\|_2^2 \\
&= \|\Sigma^{\frac{1}{2}} (\beta - \beta')\|_2^2 + \frac{1}{n}\E\|z\|_2^2.
\end{align*}
Thus when $\Sigma$ is non-singular, $\beta$ is the unique minimizer of $\E\|Y - X\beta'\|_2^2$. Since the latter only depends on the distribution of $(X,Y)$, identifiability follows.

\subsection{Departure From Elliptical Symmetry}
\label{departure}
When the design is not elliptically symmetric, $X$ and $z$ are not necessarily uncorrelated. In this case the expectation of the quadratic loss is equal to
\begin{align*}
    L(\beta') &= \E(y - \langle x, \beta' \rangle)^2 \\
    &= \E \langle x, \beta - \beta'\rangle^2 + 2 \E [\langle \beta - \beta', x \rangle (y - \langle x, \beta \rangle)] + \E(y - \langle x, \beta \rangle)^2\\
    &= (\beta' - \beta)^T \Sigma (\beta' - \beta) - 2 (\beta' - \beta)^T \E[(y - \langle x, \beta \rangle) x] + \E (y - \langle x, \beta \rangle)^2.
\end{align*}
Setting the gradient to zero we find the minimizer 
\begin{align*}
    \beta^\star = \beta + \Sigma^{-1} \E[(y - \langle x, \beta \rangle) x].
\end{align*}
Elliptical symmetry guarantees that the bias term is zero as argued in Appendix \ref{orthoproof}.

\begin{remark}
	The proofs of Proposition \ref{prop1} and Theorems \ref{thm_known_sig}, \ref{thm-unknown-sigma} only use the facts that $x_i$ is a subgaussian vector (with a uniformly bounded subgaussian constant) and that $\E[x_i \mid \langle x_i, \tau \rangle]$ is linear in $\langle x_i, \tau \rangle$. It is well-known and easy to verify that both of these conditions are satisfied for Gaussian vectors $x_i \sim N(0, \Sigma)$ when the extreme eigenvalues of $\Sigma$ are uniformly bounded away from zero and $\infty$. In particular, for any unit vector $u\in \mathbf{R}^p$ we have $\| \langle u, x_i \rangle\|_{\psi_2} \leq C \lambda_{\max}(\Sigma) < \infty$. 
The validity of these properties for elliptically symmetric random vectors has been discussed in section \ref{Elliptical-Design}. 
\end{remark}

\subsection{Proof of Proposition \ref{prop1}}

\begin{lemma}\label{RE_condition}
    Suppose that $X \in \mathbf{R}^{n \times p}$ is a random matrix with rows $x_i^T$ that are iid samples from $\mathcal{E}(0, \Sigma, F_v)$ with $\|v\|_{\psi_2} \leq \sigma_x$ for $i = 1,\dots, n$. Assume that $\Sigma^{\frac{1}{2}}$ satisfies the RE condition with parameters $(s, 9, 2\alpha)$ and that $0 < c \leq \alpha, \lambda_{\max}(\Sigma) \leq C < \infty$ for some $c,C$ not depending on $n,p$. Then, as long as $n \geq c'\sigma_x^4 s \log(p/s)$, there exist constants $c',C' > 0$ not depending on $n$ such that with probability at least $1 - \exp(-c'n^2/\sigma_x^4)$ the matrix $X$ satisfies the RE condition with parameters $(s, 3, \alpha)$.
\end{lemma}
\begin{proof}
    By elliptical symmetry, $x_i^T$ can be decomposed as 
    \begin{align*}
        x_i^T = v_i u_i^T \Sigma^{\frac{1}{2}},
    \end{align*}
    where $v_i \indep u_i$, the random vector $u_i$ is uniformly distributed on the sphere and $\E v_i^2 = p$. It is easy to verify that for any unit vector $a \in \mathbb{S}^{p-1}$ we have
    \begin{itemize}
        \item $\E \langle a, v_i u_i \rangle^2 = 1$, and,
        \item $\| \langle a, v_i u_i \rangle \|_{\psi_2} \leq \sigma_x$.
    \end{itemize}
    In other words, the random vector $v_i u_i$ is isotropic and subgaussian with constant $\sigma_x$. Since $\Sigma^{\frac{1}{2}}$ satisfies the RE condition with parameters $(s, 9, 2\alpha)$ and $\lambda_{\max}(\Sigma) \leq C$, it follows from Theorem 6 of \cite{rudelson2013reconstruction} that for some constants $c',C'$ depending only on $c,C$ and all $n \geq c'\sigma_x^4 s \log(p/s)$, the matrix $X/\sqrt{n}$ satisfies the restricted eigenvalue condition with parameters $(s, 3, \alpha)$ with probability at least $1 - \exp(-c'n^2/\sigma_x^4)$.
\end{proof}

\begin{proof}[Proposition \ref{prop1}]
	Using the definition of $\hat{\beta}$ we can write
	\begin{align*}
	\frac{1}{2n} \| Y - X\hat{\beta}\|_2^2 + \lambda \|\hat{\beta}\|_1 &\leq \frac{1}{2n}\| Y - X\beta \|_2^2 + \lambda \|\beta\|_1 \\
	& = \frac{1}{2n}\|z\|_2^2 + \lambda \|\beta\|_1.
	\end{align*}
	Expand $\|Y - X \hat{\beta}\|_2^2$ in the above inequality and rearrange to get
	\begin{align}\label{basic-ineq}
	\frac{1}{n} \|X (\hat{\beta} - \beta) \|_2^2 \leq \frac{2}{n} z^TX(\hat{\beta} - \beta) + 2 \lambda (\|\beta\|_1 - \| \hat{\beta}\|_1).
	\end{align} 
	Using H\"older's inequality we can write
	\begin{align*}
	|z^T X(\hat{\beta} - \beta) | \leq \| z^T X \|_\infty \cdot \| \hat{\beta} - \beta \|_1.
	\end{align*}
	For any $i,j$, $z_i X_{ij}$ is a subexponential random variable with 
	\begin{align*}
	\| z_i X_{ij} \|_{\psi_1} \leq \| z_i \|_{\psi_2} \cdot \| X_{ij} \|_{\psi_2} \leq \sigma_x \sigma_z
	\end{align*}
	So for each $1 \leq j \leq p$, the variable $(z^T X)_j$ is a sum of iid, mean-zero sub-exponential random variables, and thus Bernstein's inequality implies
	\begin{align*}
	\p\{ |\sum_{i=1}^n z_i X_{ij} | \geq t \} \leq 2 \exp\left[ -c \min\left( \frac{t^2}{n \sigma_x^2 \sigma_z^2} , \frac{t}{\sigma_x \sigma_z} \right) \right],
	\end{align*}
	for an absolute constant $c > 0$. As long as $t \leq n \sigma_x \sigma_z$, the subgaussian tail bound prevails. For such $t$ and using a union bound, we get
	\begin{align*}
	\p\{ \| z^T X \|_\infty \geq t \} \leq 2p \exp\left( -c  \frac{t^2}{n \sigma_x^2 \sigma_z^2} \right).
	\end{align*}
	Setting $t = \sigma_x \sigma_z \sqrt{2 c^{-1}n \log p}$, the last inequality reads
	\begin{align*}
	\p\{ \| z^T X \|_\infty \geq t \} \leq 2 p \cdot p^{-2} = 2p^{-1},
	\end{align*}
	As long as $ \log(p)/n \leq c/2$. Choosing $\lambda \geq 2 \sigma_x \sigma_z \sqrt{2c^{-1} \log(p)/n}$, we have shown that the event
	\begin{align*}
	\mathcal{T}_1 := \left[ \frac{2}{n} \|z^T X\|_\infty \leq \lambda \right]
	\end{align*}
	has probability no less than $1 - 2p^{-1}$. On this event $\mathcal{T}_1$ we can continue with inequality (\ref{basic-ineq}) to get
	\begin{align*}
	\frac{1}{n} \|X (\hat{\beta} - \beta) \|_2^2 &\leq \lambda \|\hat{\beta} - \beta\|_1 + 2 \lambda (\|\beta\|_1 - \| \hat{\beta}\|_1).
	\end{align*}
	Let $S$ be the support of $\beta$. Adding $\lambda \|\hat{\beta} - \beta \|_1$ to both sides yields
	\begin{align}\label{ineq2}
	\frac{1}{n} \|X (\hat{\beta} - \beta) \|_2^2 + \lambda \|\hat{\beta} - \beta \|_1 &\leq 2\lambda \left( \|\hat{\beta} - \beta\|_1 + \|\beta\|_1 - \|\hat{\beta}\|_1 \right) \\
	&= 2\lambda \left( \| \hat{\beta}_S - \beta_S \|_1 + \| \hat{\beta}_{S^c} \|_1 + \| \beta_S \|_1 - \| \hat{\beta}_S \|_1 - \| \hat{\beta}_{S^c} \|_1 \right) \\
	&\leq 4\lambda \|\hat{\beta}_S - \beta_S \|_1, \label{upper-bound-lasso} \\
	\end{align}
	where in the last step we used the triangle inequality. It follows from the last inequality that
	\begin{align*}
	\hat{\beta} - \beta \in \mathcal{C}(S,3) = \{ \delta \in \mathbf{R}^p : \|\delta_{S^c}\|_1 \leq 3\|\delta_S\|_1 \}.
	\end{align*}
	Let $\mathcal{T}_2$ be the event that $X/\sqrt{n}$ satisfies the RE condition with parameters $(s,3,\alpha)$. By Lemma \ref{RE_condition}, there exist constants $c',C' > 0$ such that for all $n \geq C' \sigma_x^4 s \log(p/s)$ we have $\p(\mathcal{T}_2) \geq 1 - \exp(-c'n^2/\sigma_x^4)$. Using the Cauchy-Schwartz inequality and the RE condition on $T_1 \cap T_2$,
	\begin{align*}
	\alpha^2 \| \hat{\beta} - \beta \|_2^2 \leq \frac{1}{n} \|X(\hat{\beta} - \beta )\|_2^2 \leq 3 \lambda \| \hat{\beta}_S -\beta_S\|_1 \leq 3 \lambda \sqrt{s} \|\hat{\beta} - \beta\|_2.
	\end{align*}
	Cancel $\| \hat{\beta} - \beta \|_2$ on both sides to obtain the $\ell_2$ error bound
	\begin{align*}
		\| \hat{\beta} - \beta \|_2 \leq \frac{3\lambda \sqrt{s}}{\alpha^2}.
	\end{align*}
	To obtain the $\ell_1$ error bound, note that
	\begin{align*}
		\| \hat{\beta} - \beta \|_1 \leq 4 \| \hat{\beta}_S - \beta_S \|_1 \leq 4 \sqrt{s} \| \hat{\beta} - \beta\|_2  \leq \frac{12 \lambda s}{\alpha^2}.
	\end{align*}
	Finally, the prediction error bound is obtained fom
	\begin{align*}
		\frac{1}{n} \| X(\hat{\beta} - \beta) \|_2^2 \leq 3 \lambda \|\hat{\beta}_S - \beta_S \|_1 \leq  3 \lambda \sqrt{s} \left( \frac{3\lambda \sqrt{s}}{\alpha^2} \right) = \frac{9 \lambda^2 s}{\alpha^2}
	\end{align*}
	Note that $\p(\mathcal{T}_1 \cap \mathcal{T}_2) \geq 1 - 2p^{-2} - \exp(-c_0 n^2/\sigma_x^4)$ as long as $n \geq c_0 (1 \lor \sigma_x^4) s\log(p/s)$, where $c_0 := c' \lor C' \lor (2/c) \lor 2\sqrt{2/c}$.
\end{proof}

\begin{lemma}\label{residual_prop}
	Suppose that $x_n$ is a subgaussian vector with variance proxy $\sigma_x$. The projection $r_n$ of $x_{n,1}$ on the ortho-complement of the span of $x_{n,-1}$ satisfies
	\begin{enumerate}
		\item $\lambda_{\min}(\Sigma) \leq \E r_n^2 \leq \lambda_{\max}(\Sigma)$
		\item $\| r_n \|_{\psi_2} \leq \frac{\E r_n^2}{\lambda_{\min}(\Sigma)} \cdot \sigma_x \leq \frac{\lambda_{\max}(\Sigma)}{\lambda_{\min}(\Sigma)} \cdot \sigma_x$
	\end{enumerate}
\end{lemma}
\begin{proof}
1. We have 
\begin{align*}
\E r_n^2 \leq \E x_{n,1}^2 = \Sigma_{11} \leq \lambda_{\max}(\Sigma),
\end{align*}
which proves the upper bound. For the lower bound, note that by definition, $r_n = (1, -\gamma^T) x_n$, and that $\E r_n x_n^T = (\E r_n^2) \mathbf{e}_1^T$, where $\mathbf{e}_k$ is that $k$-th standard basis vector in $\mathbf{R}^p$.  From the last equality and $\Sigma = \E x_n x_n^T$, we get
\begin{align*}
\E r_n^2 &= \| (\E r_n^2) \mathbf{e}_1^T \|_2 \\
&=\|(1, -\gamma^T) \cdot \Sigma \|_2 \\
&\geq \|(1, -\gamma^T) \|_2 \cdot \lambda_{\min}(\Sigma) \tag{$\star$}\\
&\geq \lambda_{\min}(\Sigma),
\end{align*}	
proving the lower bound.

2. From the definition of subgaussian vectors and the inequality $(\star)$ in the proof of part (1) we have
\begin{align*}
\|r_n\|_{\psi_2} &= \|(1, -\gamma^T)\|_2 \cdot \left\| \frac{(1, -\gamma^T) x_n}{\|(1, -\gamma^T)\|_2} \right\|_{\psi_2}\\
&\leq \|(1, -\gamma^T)\|_2 \cdot \sigma_x\\
&\leq \frac{\E r_n^2}{\lambda_{\min}(\Sigma)} \sigma_x,
\end{align*}
where we used $(\star)$ in the last inequality. Using the upper bound $\E r_n^2 \leq \lambda_{\max}(\Sigma)$ obtained in part (1) completes the proof.
\end{proof}

\subsection{Proof of Theorem \ref{thm_known_sig}}
\begin{proof} Without loss of generality, assume that $k = 1$. Use the representation $y_i = \langle x_i, \beta \rangle + z_i$ to rewrite $\tilde{\beta}_1$ as
	\begin{align*}
	\tilde{\beta}_1 &= \hat{\beta}_1 + \frac{\sum_{i=1}^n r_i z_i + r_i \langle x_i, \beta - \hat{\beta} \rangle}{\sum_{i=1}^n r_i x_{i,1}} \\
	&=\beta_1 + \frac{\sum_{i=1}^n r_i z_i + r_i \langle x_{i,-1}, \beta_{-1} - \hat{\beta}_{-1} \rangle}{\sum_{i=1}^n r_i x_{i,1}}\\
	\end{align*}
	Subtracting $\beta_1$ from both sides and multiplying by $\sqrt{n}$ yields
	\begin{align*}
	\sqrt{n}  (\tilde{\beta}_1 - \beta_1) = \frac{\overbrace{ \frac{1}{\sqrt{n}} \sum_{i=1}^n r_i z_i}^{B} + \overbrace{ \frac{1}{\sqrt{n}} \sum_{i=1}^n r_i\langle x_{i,-1}, \beta_{-1} - \hat{\beta}_{-1} \rangle }^{C}}{\underbrace{\frac{1}{n} \sum_{i=1}^n r_i x_{i,1}}_{A}}.
	\end{align*}
	
	We start by showing $A / \E r_n^2 \rightarrow_p 1$. That $\E A/\E r_n^2 = 1$ follows from the definition of $r_n$. By the second part of Lemma \ref{residual_prop}, $\|r_i / \E r_n^2 \|_{\psi_2} \leq 1/\lambda_{\min}(\Sigma)$. Also, we have by assumption that $x_n$ is subgaussian with variance proxy $\sigma_x$, implying that $\|x_{i,1}\|_{\psi_2} \leq \sigma_x$. Thus $(r_i x_{i,1})/\E r_n^2$ is subexponential with 
	\begin{align*}
	\left\| \frac{r_i x_{i,1}}{\E r_n^2} \right\|_{\psi_1} &\leq	\left\| \frac{r_i}{\E r_n^2}\right\|_{\psi_2} \cdot \left\| x_{i,1} \right\|_{\psi_2}\\
	&\leq \frac{\sigma_x^2}{\lambda_{\min}(\Sigma)},
	\end{align*}
	which is uniformly (in $n$) bounded above by assumption. Bernstein's inequality now implies that $\E A/\E r_n^2 \rightarrow_p 1$.

Next we bound $C$. Conditioning on the second subsample $\mathcal{S}_2=(x_i,y_i)_{i=n+1}^{2n}$, $\Delta:= \hat{\beta} - \beta$ becomes a deterministic vector.Thus conditionally we have

\begin{align*}
\left \| \frac{r_i}{\E r_n^2}  \langle x_{i,-1} , \frac{\Delta}{\|\Delta\|_2} \rangle \right\|_{\psi_1} &\leq \left \| \frac{r_i}{\E r_n^2}  \right\|_{\psi_2} \cdot \left \|  \langle x_{i,-1} , \frac{\Delta}{\|\Delta\|_2} \rangle \right\|_{\psi_2}\\
&\leq \frac{\sigma_x^2}{\lambda_{\min}(\Sigma)}.
\end{align*}
From Bernstein's inequality (Proposition \ref{sub-g-e}),
\begin{align*}
\p_{|\mathcal{S}_2} \left(   \left| \frac{1}{\sqrt{n}} \sum_{i=1}^n   \frac{r_i}{\E r_n^2}  \langle x_{i,-1} , \frac{\Delta}{\|\Delta\|_2} \rangle \right| > \frac{\sigma_x^2 \sqrt{\log(p) }}{c\lambda_{\min}(\Sigma) }  \right) & \leq 2 \exp\left[ - \min \left( \log(p), \sqrt{n\log(p)}  \right) \right] \\
&\leq \frac{2}{p},
\end{align*}
where the last inequality follows since by assumption $\log(p) = o(n)$, so that the subgaussian tail bound prevails for large enough $n,p$. Since the RHS does not depend on the second subsample (or $\Delta$), the same bound is also valid for the unconditional distribution:
\begin{align*}
\p \left(   \left| \frac{1}{\sqrt{n}} \sum_{i=1}^n   \frac{r_i}{\E r_n^2}  \langle x_{i,-1} , \hat{\beta} - \beta \rangle \right| > \frac{\sigma_x^2 \sqrt{\log(p) }}{c\lambda_{\min}(\Sigma) } \| \hat{\beta} - \beta \|_2  \right)  \leq \frac{2}{p}.
\end{align*}
Using the $\ell_2$ error bound of the Lasso estimator $\hat{\beta}$, namely $\| \hat{\beta} - \beta \|_2 \lesssim \frac{\sigma_x \sigma_z \sqrt{s \log(p)}} {\lambda_{\min}(\Sigma) \sqrt{n}}$ with probability $1 - o(1)$, we obtain
\begin{align*}
 \left| \frac{1}{\sqrt{n}} \sum_{i=1}^n   \frac{r_i}{\E r_n^2}  \langle x_{i,-1} , \hat{\beta} - \beta \rangle \right| \lesssim  \frac{\sigma_x^3 \sigma_z }{c \lambda_{\min}^2(\Sigma) } \left( \log(p)  \sqrt{\frac{s}{n}} \right), \quad \text{ with probability } 1 - o(1).
\end{align*}
Given that the multiplying constants are bounded away from $0$ and $\infty$, the sparsity assumption $s = o(n / \log^2(p))$ implies that $C / \E r_n^2 \rightarrow_p 0$.

	Finally, we show that for $\nu^2 := \E r_n^2 z_n^2 / \E r_n^2$, the term $B/\nu$ converges to $N(0,1)$ in distribution. In light of the Lyapunov condition for the central limit theorem, it is sufficient to show that $|r_n z_n|$ has a finite and bounded $(2+\delta)$-th moment for some $\delta > 0$. Let us argue that this follows from $\E |y_n|^{2 + \alpha} < M < \infty$. 
	
	For $q\geq1$, denote the $L^q$ norm of random variables by $\|\cdot \|_{L^q} := \sqrt[q]{\E|\cdot|^q}$. Let $q = 2 + \alpha$ and use the triangle inequality to write
	\begin{align*}
	\| z_n \|_{L^q} &= \| y_n - \mu \langle x_n, \tau \rangle \|_{L^q}\\
	&\leq \| y_n \|_{L^q} + | \mu | \cdot \| \langle x_n, \tau \rangle \|_{L^q}.
	\end{align*}
	
	By assymption, $\| y_n \|_{L^q} \leq M^{1/q}$. Using the Cauchy-Shwartz inequality, 
	\begin{align*}
	|\mu| &= |\E y_n \langle x_n, \tau \rangle | \\
	&\leq \sqrt{\E|y_n|^2} \cdot \sqrt{\E \langle x_n, \tau \rangle^2 }\\
	&\leq \sqrt[q]{M}, 
	\end{align*}
	where the last inequality uses the normalization of $\tau$ and the fact that $\|y_n\|_2 \leq \|y_n\|_q$ for $q = 2+\alpha \geq 2$. Next, note that $\langle x_n, \tau \rangle$ is a subgaussian random variable with $\| \langle x_n, \tau \rangle \|_{\psi_2} \leq \sigma_x \cdot \|\tau\|_2$. By the properties of subgaussian random variables \cite[Proposition 2.5.2]{vershynin2018high}, the $L^q$ norms of subgaussian random variables are bounded by their $\psi_2$ norms, so we have
	\begin{align*}
		\| \langle x_n, \tau \rangle \|_{L^q} \leq c \sigma_x \cdot \|\tau\|_2 \cdot \sqrt{q},
	\end{align*}
	for an absolute constant $c > 0$. Noting that $\|\Sigma^{1/2}\tau\|_2 = 1$ implies $\|\tau \|_2 \leq 1/\sqrt{\lambda_{\min}(\Sigma)}$, we obtain
\begin{align}
\|z_n\|_{L^q} \leq M^{1/q} (1+ c \sigma_x  \sqrt{q/\lambda_{\min}(\Sigma)}) \label{boundz},
\end{align}
\end{proof}
proving that $\|z_n\|_{L^q}$ is bounded (uniformly in $n$) away from infinity. Lemma \ref{subexp-cor} shows that $\|r_n\|_{\psi_2} \leq \sigma_x\lambda_{\max}(\Sigma)/\lambda_{\min}(\Sigma)$. Using H\"older's inequality and the bound on the moments of subgaussian random variables, for $q' = 2 + \alpha/2 < q$, 
\begin{align*}
\E|r_n z_n|^{q'} &\leq \| |z_n|^{q'} \|_{L^{q/q'}} \cdot \| |r_n|^{q'} \|_{L^{q/(q-q')}}\\
&\leq \| z_n \|_{L^q}^{q'}  \cdot \| r_n\|_{L^{q''}}^{q'} \tag{$q'' = (q - q')/(qq')$}\\
& \leq M^{q'/q} \left(1+ c \sigma_x  \sqrt{\frac{q}{\lambda_{\min}(\Sigma)}}\right)^{q'}  \cdot
 \left(c \sigma_x \frac{\lambda_{\max}(\Sigma)}{\lambda_{\min}(\Sigma)} \sqrt{q''}\right)^{q'}\\
&\leq c'
\end{align*}
for some $c' < \infty$ that does not depend on $n$. It follows that for $q' = 2 + \alpha/2$ the Lyapunov condition is satisfied:
\begin{align*}
\frac{n\E|r_n z_n|^{2+\alpha/2}}{(\sqrt{n\E r_n^2 z_n^2 } )^{2+\alpha/2}} \leq \frac{c'}{n^{\alpha/4}c_{\text{rz}}} \rightarrow 0,
\end{align*}
and the proof is complete.
\subsection{Proof of Theorem \ref{thm-unknown-sigma}}
\begin{proof}
First note that the assumptions of Theorem \ref{thm-unknown-sigma} include the asusmptions of Proposition \ref{prop1} proving $\| \hat{\beta} - \beta \|_1 \lesssim s \sqrt{\log(p) / n}$ with high probability.
	Let $R, \hat{R} \in \mathbf{R}^p$ be the vectors with $r_i / (n \E r_n^2) $ and $\hat{r}_i / \sum_k \hat{r}_k x_{k,1}$  in their $i$-th poisition respectively. Using $y_i = \langle x_i, \beta \rangle + z_i$ in the the definition of $\tilde{\beta}_1$ and rearranging yields
\begin{align*}
	\tilde{\beta}_1 &= \hat{\beta}_1 + \hat{R}^T( y - X \hat{\beta})\\
&= \tilde{\beta}_1 + \hat{R}^TX (\beta - \hat{\beta}) + \hat{R}^T z \\
&= \beta_1 + (R^T X - \e_1^T) (\beta - \hat{\beta}) + R^T z + (\hat{R}  - R)^T X(\beta - \hat{\beta} )+ (\hat{R} - R)^T z
	\end{align*}
	Subtracting $\beta_1$ from both sides and multiplying by $\sqrt{n}$,
	\begin{align}
	\sqrt{n} (\tilde{\beta} - \beta_1) = A +  B + C
	\end{align}  
	where 
\begin{align*}
A &= \sqrt{n} \left( (R^T X - \e_1^T) (\beta - \hat{\beta}) + R^T z \right) \\
B &= \sqrt{n} (\hat{R}  - R)^T X(\beta - \hat{\beta} ) \\
C &= \sqrt{n} (\hat{R} - R)^T z.
\end{align*}
As in the case of known $\Sigma$, it can be shown that $\sqrt{n} R^T z / \nu \rightarrow_d N(0,1)$. Also as in the case of known $\Sigma$, using the subexponential property of $ R^T X - \mathrm{e}_1^T$ we can show that $\sqrt{n} \| R^T X - \mathrm{e}_1^T \|_\infty = \mathcal{O}_p( \sqrt{\log(p)})$. Thus using an $\ell_1 - \ell_\infty$ bound it follows that $| (R^T X - \e_1^T) (\beta - \hat{\beta})  | \leq \| R^T X - \e_1^T \|_\infty \cdot \| \beta- \hat{\beta} \|_1 = \mathcal{O}_p(s \log(p) / \sqrt{n} )$ which is negligible by assumption. So it suffices to show that $B,C \rightarrow_p 0$.

By the definition of $R$ and $\hat{R}$,
\begin{align*}
|B| &= \left| \sqrt{n} \left( \frac{n}{\sum_i r_i x_{i,1}} (1 , -\hat{\gamma}^T) - \frac{1}{\E r_n^2} (1, -\gamma^T) \right) \hat{\Sigma} (\beta - \hat{\beta}) \right|\\
&\leq \sqrt{n} \left\| \frac{n}{\sum_i r_i x_{i,1}} (1 , -\hat{\gamma}) - \frac{1}{\E r_n^2} (1, -\gamma) \right\|_1 \cdot \left\| \hat{\Sigma} (\beta - \hat{\beta}) \right\|_\infty.
\end{align*}
For the first term we can use the triangle inequality to write
\begin{align*}
 \left\| \frac{n}{\sum_i r_i x_{i,1}} (1 , -\hat{\gamma}) - \frac{1}{\E r_n^2} (1, -\gamma) \right\|_1 &\leq \frac{n}{| \sum_{i} r_i x_{i1} |} \| \hat{\gamma} - \gamma \|_1 \\
  &+ \left| \frac{1}{\E[r_n^2]}   - \frac{n}{ \sum_{i} r_i x_{i1} } \right| \cdot (1 + \| \gamma\|_1 )\\
 &= \mathcal{O}_p\left( s\sqrt{\frac{\log(p)}{n}} + \sqrt{\frac{s}{n}}  \right) = \mathcal{O}_p\left( s\sqrt{\frac{\log(p)}{n}}  \right).
  \end{align*}
where the first equality follows because $\E[r_n^2] \geq \lambda_{\max}(\Sigma)$ is bounded away from zero and $\sum_i r_i x_{i1}  / n - \E[r_n^2] = \mathcal{O}_p(1/\sqrt{n})$, and $\| \gamma \|_1 = \mathcal{O}( \sqrt{s})$, since we have $\| \gamma\|_1 \leq \sqrt{s} \| \gamma\|_2$ and furthermore $(1, -\gamma^T) \Sigma = \E[r_n^2] \mathrm{e}_1^T$ so that $\|\gamma\|_2 \leq \E[r_n^2] / \lambda_{\min}(\Sigma) \leq \lambda_{\max}(\Sigma)/\lambda_{\min}(\Sigma)$.

For the second term, use the KKT conditions on the definition of $\hat{\beta}$ to obtain a vector $u$ in the subgradient of $\| \hat{\beta}\|_1$ that satisfies
\begin{align*}
-\frac{1}{n} X^T ( Y - X \hat{\beta}) + \lambda u = 0.
\end{align*}
Rearranging and using $\| u \|_{\infty} \leq 1$ gives for $\lambda \gtrsim \sigma_x \sigma_z \sqrt{\log(p)/n}$,
\begin{align*}
\| \hat{\Sigma} (\beta - \hat{\beta} ) \|_{\infty} \leq \frac{1}{n} \|X^T z \|_{\infty} + \lambda.
\end{align*}
In the proof of Proposition \ref{pilot-estimate} it is proved\footnote{Note that the assumptions of Proposition \ref{pilot-estimate} are included in the assumptions of Theorem \ref{thm-unknown-sigma}.}
 that $\frac{1}{n} \|X^T z \|_\infty = \mathcal{O}_p(\lambda)$. Thus we have $\| \hat{\Sigma} (\beta - \hat{\beta}) \|_\infty = \mathcal{O}_p (\lambda)$.
Combining these bounds gives
\begin{align*}
|B| &= \mathcal{O}_p \left( \sqrt{n} \cdot s \frac{\log(p)}{\sqrt n}  \cdot  \sigma_x \sigma_z \sqrt{\frac{\log(p)}{n}} \right)\\
&= \mathcal{O}_p \left( s\sqrt{ \frac{\log(p)}{n} }\right).
\end{align*}
The assumption $s =o( \sqrt{ n}/ \log(p))$ now implies that $B \rightarrow_p 0$.

Finally, $C$ can be bounded by
\begin{align*}
|C| &= \left| \left( \frac{n}{\sum_i r_i x_{i,1}} (1 , -\hat{\gamma}^T) - \frac{1}{\E r_n^2} (1, -\gamma^T) \right) \frac{1}{\sqrt n}X^T z \right| \\
& \leq \left\|  \frac{n}{\sum_i r_i x_{i,1}} (1 , -\hat{\gamma}^T) - \frac{1}{\E r_n^2} (1, -\gamma^T)  \right\|_1 \cdot \left\| \frac{1}{ \sqrt n} X^T z \right\|_\infty.
\end{align*}
The first term we have already shown to be $\mathcal{O}_p\left( s \sqrt{\log(p)/n} \right)$.  
The second term $\left\| \frac{1}{ \sqrt n} X^T z \right\|_\infty$ is $\mathcal{O}_p (\sqrt{n} \lambda ) = \mathcal{O}_p(\sqrt{\log(p)} )$ as established in the proof of Proposition \ref{prop1}. Thus we obtain
\begin{align*}
|C| = \mathcal{O}_p \left( s \frac{\log(p)}{\sqrt{n}} \right),
\end{align*}
which implies $C \rightarrow_p 0$ by the assumption $s = o \left(\frac{ \sqrt{n}}{\log(p)} \right)$.

\end{proof}

\subsection{Proof of Theorem \ref{thm-efficient}}

Some well-known properties of Hermite polynomials are collected in the following proposition. Definitions and proofs of the first two statements can be found in Section 11.2 of  \cite{o2014analysis}. Statements 2 and 3 are easy to verify from the definition. The last statement is proved in Theorem 2.1 of \cite{larsson-cohn2002}.

\begin{prop}\label{hermite-prop}
	For Hermite polynomials $\{h_j\}_{j=0}^\infty$ defined by (\ref{hermite_def}) the following are true:
	\begin{enumerate}
		\item $\{h_j\}_{j=0}^\infty$ forms an orthonormal basis of $L^2(\mathbf{R}, N(0,1))$.
		
		\item For (deterministic) unit vectors $\tau,\hat{\tau}\in \mathbf{R}^p$ and $x \sim N(0,I_p)$ we have 
		\begin{align}
		\E [h_j(\langle \tau, x \rangle) h_k(\langle \hat{\tau}, x \rangle)] = \langle \tau, \hat{\tau} \rangle^j \cdot \mathbf{1}[j = k].
		\end{align}
		
		\item For all $j \geq 1$, $h_j' = \sqrt{j} h_{j-1}$.
		\item For all $j \geq 1$, $\xi  h_j(\xi) = \sqrt{j+1} h_{j+1} + \sqrt{j} h_{j-1}(\xi)$
		\item For $q > 2$ the $L^q$ norms (w.r.t. the Gaussian measure) of Hermite polynomials satisfy
		\begin{align*}
		\|h_j\|_{L_q} = \frac{c(q)}{j^{1/4}} (q-1)^{j/2} \left( 1 + \mathcal{O}(\frac{1}{j}) \right),
		\end{align*}
		as $j \rightarrow \infty$, where $c(q) = (1/\pi)^{1/4} ((q-1)/(2q-4))^{(q-1)/(2q)}.$
	\end{enumerate}  
\end{prop}

The following lemma relates the smoothness of a function $g =\sum_{j=0}^\infty \mu_j h_j$ to the decay of the sequence $\{\mu_j\}_{j}$. The result and its proof are direct analogues of Lemma A.3 in the work of \cite{Tsybakov:2008:INE:1522486} which concerns the trigonometric basis.

\begin{lemma}\label{decay}
	Suppose that $g(\xi) = \sum_{j=0}^\infty \mu_j h_j(\xi)$. Assume also that $g$ is $k$-times continuously differentiable and that 
	\begin{align*}
	\underset{\xi \sim N(0,1)}{\E}|g^{(k)}(\xi)|^2 \leq L^2.
	\end{align*}
	Then we have
	\begin{align*}
	\sum_{j=0}^\infty \frac{(j+k)!}{j!} \mu_{j+k}^2 \leq L^2.
	\end{align*}
	
\end{lemma}
\begin{proof}
	Let $\mu_j(k) := \E g^{(k)}(\xi) h_j(\xi)$ for $k \geq 1$ and $\mu_j(0) = \mu_j$. Using integration by parts, for $k\geq 1$ we have
	\begin{align*}
	\mu_j(k) &= \frac{1}{\sqrt{2\pi}} \int_{-\infty}^{+\infty} g^{(k)}(\xi) h_j(\xi) e^{-\xi^2/2} d\xi\\
	&=\frac{1}{\sqrt{2\pi}} \left[ g^{(k-1)}(\xi) h_j(\xi)e^{-\xi^2/2} \right]_{-\infty}^{+\infty} -\frac{1}{\sqrt{2\pi}} \int_{-\infty}^{\infty}  g^{(k-1)}(\xi)(\sqrt{j}h_{j-1}(\xi) - \xi h_j(\xi)) e^{-\xi^2/2} d\xi\\
	&= 0 - \sqrt{j+1} \E g^{(k-1)}(\xi) h_{j+1}(\xi)\\
	&= \sqrt{j+1} \mu_{j+1}(k-1)
	\end{align*} 
	where in the second and third equalities we used the parts (3) and (4) of Proposition \ref{hermite-prop}.
	From the recursion $\mu_j(k) = \sqrt{j+1}\mu_{j+1}(k-1)$ it follows that $\mu_j^2(k) = \frac{(j+k)!}{j!} \mu_{j+k}^2$. Using the latter and Parseval's identity.
	\begin{align*}
	\sum_{j=0}^\infty \frac{(j+k)!}{j!} \mu_{j+k}^2 = \sum_{j=0}^\infty \mu_j(k)^2 = \underset{\xi \sim N(0,1)}{\E}|g^{(k)}(\xi)|^2 \leq L^2.
	\end{align*}
\end{proof}

Before we present the proof of Theorem \ref{thm-efficient}, we state and prove a fact involving Gaussian decompositions of random variables that will be used several times in the proof. 
\begin{lemma}
Let $\tilde{\omega}^T = (1, -\gamma^T)$ so that $r_n = \langle x_n, \tilde{\omega} \rangle$. One can write $r_n = \alpha \langle x_n, \tau \rangle + \hat{\alpha} \langle x_n, \hat{\tau} \rangle + u$ such that conditional on $\hat{\tau}$, the random variable $u$ is independent of $\langle x_n, \tau \rangle$ and $\langle x_n, \hat{\tau} \rangle$. Furthermore, $\alpha$ and $\hat{\alpha}$ are both $\mathcal{O}_p(1)$.
\end{lemma}
\begin{proof}
Without loss of generality assume that $\Sigma = I$. Then $\| \tau \|_2 = \| \hat{\tau} \|_2 = 1$. The conditional independence follows if we have
\begin{align*}
\alpha = \frac{ \langle \tilde{\omega} , \hat{\tau} \rangle \langle \tau, \hat{\tau} \rangle - \langle \tilde{\omega}, \tau \rangle }{ \langle \hat{\tau}, \tau \rangle^2 - 1 }, \text{ and } \, \hat{\alpha} =  \frac{ \langle \tilde{\omega} , \tau \rangle \langle \tau, \hat{\tau} \rangle - \langle \tilde{\omega}, \hat\tau \rangle }{ \langle \hat{\tau}, \tau \rangle^2 - 1 }.
\end{align*}
Next we show that $\alpha = \mathcal{O}_p(1)$. Note that $\tilde{\omega} = \E[r_n^2] \mathrm{e}_1$. Therefore, we have
\begin{align*}
\alpha &= \frac{ \E[r_n^2] ( \hat{\tau}_1  \langle \tau, \hat{\tau} \rangle - \tau_1) }{ \langle \hat{\tau}, \tau \rangle^2 - 1 }\\
&= \frac{ \E[r_n^2] \hat{\tau}_1 }{ \langle \hat{\tau}, \tau \rangle + 1} + \frac{ \E[r_n^2] ( \hat{\tau}_1 - \tau_1 )}{ ( \langle \hat{\tau}, \tau \rangle + 1 )(\langle \hat{\tau}, \tau \rangle - 1)} = \mathcal{O}_p(1),
\end{align*}
since $| \hat{\tau}_1 | \leq 1$ and $\langle \hat{\tau} , \tau \rangle \rightarrow_p 1$ (see below) and $\E[r_n^2] \leq \lambda_{\max}(\Sigma)$, and by assumption we have $(\hat{\tau}_1 - \tau_1) = \mathcal{O}_p \left(\langle \hat{\tau}, \tau \rangle - 1 \right) $. The argument for $\hat{\alpha} = \mathcal{O}_p(1)$ is similar.
\end{proof}

\textbf{Proof of Theorem \ref{thm-efficient}}. First we show that $\hat{\tau}$ has the same rate of convergence as $\hat{\beta}$. To see this, use the triangle inequality to write
\begin{align*}
\|\hat{\tau} - \tau\|_2 &= \left\|\frac{\hat{\beta}}{\hat{\mu}_1} - \frac{\beta}{\mu_1}\right\|_2 \\
&\leq \left\| \frac{\hat{\beta}}{\hat{\mu}_1} - \frac{\hat{\beta}}{\mu_1} \right\|_2 + \left\| \frac{\hat{\beta}}{\mu_1} - \frac{\beta}{\mu_1} \right\|_2\\
&=\frac{1}{\mu_1} \left( |\hat{\mu}_1 - \mu_1 | \cdot \frac{\|\hat{\beta}\|_2}{\hat{\mu}_1}   + \|\hat{\beta} -\beta\|_2 \right).
\end{align*}

We can now write 
\begin{align*}
|\hat{\mu}_1 - \mu| &= \left| \|\Sigma^{\frac{1}{2}} \hat{\beta} \|_2 - \| \Sigma^\frac{1}{2} \beta \|_2 \right| \\
&\leq \|\Sigma^{\frac{1}{2}} (\hat{\beta} - \beta) \|_2\\
&\leq \lambda_{\max}(\Sigma^{\frac{1}{2}}) \cdot \|\hat{\beta} - \beta\|_2.
\end{align*}
On the other  hand,
\begin{align*}
\frac{\|\hat{\beta}\|_2}{\hat{\mu}_1} = \frac{\|\hat{\beta}\|_2}{\| \Sigma^{\frac{1}{2}} \hat{\beta} \|_2} \leq \max_{\theta \in \mathbf{R}^p}   \frac{\|\theta\|_2}{\| \Sigma^{\frac{1}{2}} \theta \|_2} 
= \left(\min_{\theta \in \mathbf{R}^p}   \frac{\|\Sigma^{\frac{1}{2}} \theta \|_2}{\|\theta\|_2 } \right)^{-1} = \lambda_{\min}^{-1}(\Sigma^{\frac{1}{2}}).
\end{align*}
The last three bounds put together yield
\begin{align*}
\| \hat{\tau} - \tau \|_2 \leq \frac{1 + \kappa(\Sigma)^{\frac{1}{2}}}{\mu_1} \|\hat{\beta} - \beta\|_2,
\end{align*}
where $\kappa(\Sigma) = \lambda_{\max}(\Sigma) / \lambda_{\min}(\Sigma)$ is the condition number of $\Sigma$. By assumptions (2) and(1), the parameters $\mu_1$ and $\kappa(\Sigma)$ are bounded away from zero and infinity, respectively, showing that $\hat{\tau}$ has the same rate of convergence as $\hat{\beta}$.

We can now write 

\begin{align*}
\sqrt{n} (\frac{1}{n} \sum_{i=1}^n r_i x_{i1})\frac{(\tilde{\beta}_1 - \beta_1)}{\sqrt{\E r_n^2 e_n^2}} &= \frac{1}{\sqrt{n\E r_n^2 e_n^2}} \sum_{i=1}^n r_i e_i  \tag{A}\\
&+ \frac{1}{\sqrt{n\E r_n^2 e_n^2}} \sum_{i=1}^n r_i (\mu_0 - \hat{\mu}_0) \tag{B} \\
&+ \frac{1}{\sqrt{n\E r_n^2 e_n^2}} \sum_{i=1}^n r_i  \langle \mu_1\tau_{-1} - \hat{\mu}_1 \hat{\tau}_{-1}, x_{i,-1}\rangle \tag{C} \\
&+\frac{1}{\sqrt{n\E r_n^2 e_n^2}} \sum_{i=1}^n r_i  \sum_{j=2}^m \mu_j [h_j(\langle \tau, x_i \rangle) - h_j(\langle \hat{\tau}, x_i \rangle)] \tag{D}\\
&+\frac{1}{\sqrt{n\E r_n^2 e_n^2}} \sum_{i=1}^n \sum_{j=2}^m r_i(\mu_j - \hat{\mu}_j) h_j(\langle \hat{\tau}, x_i \rangle) \tag{E}\\
&+ \frac{1}{\sqrt{n\E r_n^2 e_n^2}} \sum_{i = 1}^n r_i\sum_{j = m+1}^\infty \mu_j h_j(\langle \tau, x_i \rangle). \tag{F}
\end{align*}

We will show that the first term (A) converges in law to a normal distribution and the other terms (B-F) converge to zero in probability.

\textbf{(A)}. As argued in the proof of Theorem \ref{thm_known_sig}, an application of H\"older's inequality shows that $\E|e_n|^{2 + \alpha} < M$ implies that $\E|r_n e_n|^{2 + \alpha}$ is uniformly bounded above. Therefore the Lyapunov condition for the central limit theorem is satisfied and the first term $(A)$ converges to $N(0,1)$ in distribution.

\textbf{(B)}. Since $\E[B | \mathcal{S}_2] = 0$, and $\E \hat{\mu}_0 = \mu_0$, the variance of $(B)$ evaluates to
\begin{align*}
\V(B) &= \V[\E[B | \mathcal{S}_2]] + \E[\V[B|\mathcal{S}_2]]\\
&=\frac{\E(\mu_0 - \hat{\mu}_0)^2 \E r_n^2}{\E r_n^2 e_n^2}\\
&\leq \frac{\lambda_{\max}(\Sigma) \E y_n^2}{\lceil n/2 \rceil \E r_n^2 e_n^2} \rightarrow 0.
\end{align*}

\textbf{(C)}. The ``linear'' term $(C)$ has been handled in the proof of Theorem \ref{thm_known_sig}, as by definition $\hat{\mu}_1 \hat{\tau} = \hat{\beta}$ and $\mu_1 \tau = \beta$.

\textbf{(D) . } From Proposition \ref{hermite-prop} it is clear that $\E[D \mid \mathcal{S}_2] = 0$. Therefore it suffices to show that $\E[D^2 \mid \mathcal{S}_2] = o_p(1)$.
Let $\tilde{\gamma}^T = (1, -\gamma^T)$ so that $\Sigma \tilde{\gamma} = \E[r_n^2] \mathrm{e}_1$.  Then we can write $\xi_n := \langle \tau, x_n \rangle = \tau_1 r_n + u_n$ and $\hat{\xi}_n := \langle \hat{\tau}, x_n \rangle = \hat{\tau}_1 r_n + \hat{u}_n$, where after conditioning on $\mathcal{S}_2$, the random variables $u_n, \hat{u_n}$ are independent of $r_n$. For each value of $u_n,\hat{u}_n$ define $\psi_{u_n,\hat{u}_n}$ as a random function of $r_n$ as follows:
\begin{align*}
\psi_{u_n, \hat{u}_n} := \sum_{j=2}^m \mu_j [ h_j( \tau_1 r_n + u_n) - h_j( \hat{\tau}_1 r_n + \hat{u}_n) ]
\end{align*}
Condition on $\mathcal{S}_2, u_n, \hat{u}_n$ and use Stein's lemma to obtain
\begin{align*}
\E[ r_n^2 \psi_{u_n, \hat{u}_n}^2 \mid \mathcal{S}_2, u_n, \hat{u}_n] &= 
\E[r_n^2 \mid \mathcal{S}_2, u_n, \hat{u}_n ] \E[  \psi_{u_n, \hat{u}_n}^2 \mid \mathcal{S}_2, u_n, \hat{u}_n] \\
&+ 2 (\E[r_n^2 \mid \mathcal{S}_2, u_n, \hat{u}_n ])^2  \E[  \psi_{u_n, \hat{u}_n}'^2 \mid \mathcal{S}_2, u_n, \hat{u}_n] \\
&+ 2 (\E[r_n^2 \mid \mathcal{S}_2, u_n, \hat{u}_n ])^2  \E[  \psi_{u_n, \hat{u}_n}\psi_{u_n, \hat{u}_n}'' \mid \mathcal{S}_2, u_n, \hat{u}_n].
\end{align*}
Taking another expectation with respect to the distribution of $u_n, \hat{u}_n$ given $\mathcal{S}_2$ and using the tower property of conditional expectations yields\footnote{This trick will be used several times in the proofs and we will refer to this discussion for details.}
\begin{align}
\E[ D^2 \mid \mathcal{S}_2] &= \E[ r_n^2 \psi_{u_n, \hat{u}_n}^2 \mid \mathcal{S}_2] \nonumber \\
&=\E[r_n^2  ] \E[  \psi_{u_n, \hat{u}_n}^2 \mid \mathcal{S}_2] 
+ 2 (\E[r_n^2  ])^2  \left( \E[  \psi_{u_n, \hat{u}_n}'^2 \mid \mathcal{S}_2] 
+  \E[  \psi_{u_n, \hat{u}_n}\psi_{u_n, \hat{u}_n}'' \mid \mathcal{S}_2] \right). \label{stein_trick}
\end{align}
We need to show that the three terms in the above sum are negligible in probability. Let us consider $\E[ \psi_{u_n, \hat{u}_n}^2 \mid \mathcal{S}_2]$ first. Using the orthonormality of Hermite polynomials (Proposition \ref{hermite-prop}), we obtain
\begin{align*}
\E[ \psi_{u_n, \hat{u}_n}^2 \mid \mathcal{S}_2] &= \sum_{j=2}^m \mu_j^2 \E[ (h_j(\xi) - h_j(\hat\xi))^2 \mid \mathcal{S}_2]\\
&= \sum_{j=2}^m \mu_j^2 ( 2 - 2 \langle \Sigma^\frac12 \tau,  \Sigma^\frac12 \hat{\tau} \rangle^j)\\
&= 2 \sum_{j=2}^m \mu_j^2 \left( 1 - \left( 1 - \frac{\| \Sigma^\frac12 ( \hat\tau - \tau )\|_2^2}{2} \right)^j \right)\\
&\leq \lambda_{\max}(\Sigma)^\frac12 \|  \hat\tau - \tau \|_2^2 \sum_{j=2}^m j \mu_j^2,
\end{align*}
where the last inequality follows because $1 - (1-\theta)^j \leq j \theta$ for $\theta \in [0,1]$. Since $\sum_j j \mu_j^2 \lesssim L^2$ and $\| \hat\tau - \tau \|_2^2 \rightarrow_p 0$, it follows that $\E[\psi_{u_n, \hat{u}_n}^2 \mid \mathcal{S}_2] \rightarrow_p 0$ as well.

Next we consider $ \E[  \psi_{u_n, \hat{u}_n}'^2 \mid \mathcal{S}_2] $. Using Proposition \ref{hermite-prop} to compute derivatives and inner products of Hermite polynomials, we obtain
\begin{align*}
 \E[  \psi_{u_n, \hat{u}_n}'^2 \mid \mathcal{S}_2] &= \E\left[ \left(  \sum_{j=2}^m \mu_j \sqrt{j} ( \tau_1 h_{j-1}(\xi) - \hat{\tau}_1 h_{j-1}(\hat{\xi})) \right)^2  \given \mathcal{S}_2 \right]\\
 &= \sum_{j=2}^m j \mu_j^2  (\tau_1^2 + \hat{\tau}_1^2 - 2 \tau_1 \hat{\tau}_1 \langle \Sigma^\frac12 \tau, \hat{\tau} \rangle^{j-1})\\
 &= \sum_{j=2}^m j \mu_j^2  \left((\tau_1 - \hat{\tau}_1)^2 + 2 \tau_1 \hat{\tau}_1 (1 -  \langle \Sigma^\frac12 \tau, \Sigma^\frac12 \hat{\tau} \rangle^{j-1}) \right)\\
 &\leq (\tau_1 - \hat{\tau}_1)^2 \sum_{j=2}^m j\mu_j^2 + 2 \tau_1 \hat{\tau}_1 \lambda_{\max}(\Sigma)^\frac12 \| \hat{\tau} - \tau \|_2^2 \sum_{j=2}^m j(j-1) \mu_j^2.
\end{align*}
That the first summand is $o_p(1)$ is clear. For the second summand, use the Cauchy-Schwarz inequality to obtain
\begin{align*}
\sum_{j=2}^m j(j-1) \mu_j^2 &\leq \left( \sum_{j=2}^m (j\mu_j^2)^2 \right)^\frac12  \left( \sum_{j=2}^m (j-1)^2\right)^\frac12\\
&\leq \left( \sum_{j=2}^m (j\mu_j^2) \right)  \left( \sum_{j=2}^m (j-1)^2\right)^\frac12\\
&\lesssim L^2 \cdot m^\frac32.
\end{align*}
Since $m^\frac32 \|\hat{\tau} - \tau \|_2^2 \leq \log(p) \| \hat{\tau} - \tau \|_2^2  \rightarrow_p 0$, we have proved that $\E[\psi_{u_n, \hat{u}_n}^2 \mid \mathcal{S}_2] = o_p(1)$.

Next, the term $\E[  \psi_{u_n, \hat{u}_n}\psi_{u_n, \hat{u}_n}'' \mid \mathcal{S}_2]$ is calculated as follows:
\begin{align*}
& \E\left[  \left( \sum_{k=2}^m \mu_k(h_k(\xi) - h_k(\hat\xi))  \right) \left( \sum_{j=2}^m \mu_j \sqrt{j(j-1)} ( \tau_1^2 h_{j-2}(\xi) - \hat{\tau}_1^2 h_{j-2}(\hat\xi))  \right)  \given \mathcal{S}_2 \right]\\
&= \sum_{j=2}^{m-2} \sqrt{(j+2)(j+1)} \mu_j \mu_{j+2} ( \tau_1^2 + \hat\tau_1^2 )( 1 - \langle \Sigma^\frac12 \tau, \Sigma^\frac12 \hat{\tau} \rangle^j )\\
&\leq   (\tau_1^2 + \tau_1^2)  \lambda_{\max}(\Sigma)^\frac12\| \hat{\tau} - \tau \|_2^2 \sum_{j=2}^{m-2} j (j+2) \mu_j \mu_{j+2}\\
&\leq (\tau_1^2 + \tau_1^2)  \lambda_{\max}(\Sigma)^\frac12\| \hat{\tau} - \tau \|_2^2 \sum_{j=2}^m j^2 \mu_j^2
\end{align*}
where the last inequality follows from the elementary inequality $2 j(j+2) \mu_j \mu_{j+2} \leq j^2 \mu_j^2 + (j+2) \mu_{j+2}^2$. As before, it is easy to show that
\begin{align*}
\| \hat{\tau} - \tau \|_2^2 \sum_{j=2}^m j^2 \mu_j^2 \lesssim \| \hat{\tau} - \tau \|_2^2 L^2 m^\frac32 = o_p(1).
\end{align*}
This completes the proof of $\E[D^2 \mid \mathcal{S}_2] = o_p(1)$.

\textbf{(E)}. The same technique can be applied to compute the conditional variance of $(E)$. Use a Gaussian decomposition $r_i = \alpha' \hat{\xi}_i + v_i$ where after conditioning on $\mathcal{S}_2$, $v_i$ is independent of $\hat{\xi}_i$, to rewrite $(E)$ as 
\begin{align*}
E &= \frac{1}{\sqrt{n\E r_n^2 e_n^2}} \sum_{i=1}^n \sum_{j=2}^m \alpha' \hat{\xi}_i (\mu_j - \hat{\mu}_j) h_j(\langle \hat{\tau}, x_i \rangle) \tag{$E_1$} \\ &+ \frac{1}{\sqrt{n\E r_n^2 e_n^2}} \sum_{i=1}^n \sum_{j=2}^m v_i (\mu_j - \hat{\mu}_j) h_j(\langle \hat{\tau}, x_i \rangle). \tag{$E_2$}
\end{align*}

The conditional variance of $E_2$ is
\begin{align*}
\E[E_2^2 | \mathcal{S}_2] &= \frac{\E[v_n^2 | \mathcal{S}_2]}{\E r_n^2 e_n^2} \sum_{j=2}^m (\mu_j - \hat{\mu}_j)^2.
\end{align*}
From the definition of $v_n$, it can be seen that $v_n$ depends on $\mathcal{S}_2$ only through $\hat{\tau}$, so that we have $\E[v_n^2|\mathcal{S}_2] = \E[v_n^2|\hat{\tau}]$. Also note that $\hat{\tau}$ is computed on $\mathcal{S}_{21} \subset \mathcal{S}_2$, so that $\mathcal{S}_2$ contains all the information \footnote{In terms of $\sigma$-algebras we have $\sigma(\hat{\tau}) \subset \sigma(\mathcal{S}_2)$.} about $\hat{\tau}$. This allows us to use the tower property of conditional expectations to write
\begin{align*}
\E[E_2^2 | \hat{\tau}] &= \E[\E[E_2^2 | \mathcal{S}_2] | \hat{\tau}] \\
&= \frac{\E[v_n^2 | \hat{\tau}]}{\E r_n^2 e_n^2} \sum_{j=2}^m \E[(\mu_j - \hat{\mu}_j)^2 | \hat{\tau}]\\
&= \frac{\E[v_n^2 | \hat{\tau}]}{\E r_n^2 e_n^2} \sum_{j=2}^m \E[(\hat{\mu}_j - \E[\hat{\mu}_j | \hat{\tau}])^2 | \hat{\tau}] + (\E[\hat{\mu}_j | \hat{\tau}] - \mu_j)^2
\end{align*}

The conditional mean and variance of $\hat{\mu}_j$ are 
\begin{align*}
\E [\hat{\mu}_j \mid \hat{\tau} ] &= \E[y_{2n} h_j(\langle x_{2n} , \hat{\tau})]  = \mu_j \langle \Sigma \tau, \hat{\tau} \rangle^j,\\
\V[\hat{\mu}_j \mid \hat{\tau}] &= \frac{1}{\lceil n/2 \rceil} (\E y_{2n}^2 h_j^2(\langle x_{2n}, \hat{\tau}\rangle) - \mu_j^2 \langle \Sigma \tau, \hat{\tau} \rangle^j).
\end{align*}

To bound the conditional variance, we use the Cauchy-Schwarz inequality and the last part of Proposition \ref{hermite-prop} to obtain
\begin{align*}
\E y_{2n}^2 h_j^2(\langle x_{2n}, \hat{\tau}\rangle) &\leq \|y_{2n}\|_4^2 \cdot \|h_j\|_4^2 \\
&\leq C_y^2 \cdot C_h^2 \cdot \frac{3^j}{j^{1/2}} ,
\end{align*}
for some absolute constant $C_h$. Using the latter to bound the conditional variance of $E_2$ we obtain
\begin{align*}
\E[E_2^2\mid \hat{\tau}] &\leq \frac{\E[v_n^2 | \hat{\tau}]}{\E r_n^2 e_n^2} \left( C_y^2 C_h^2 \sum_{j=2}^m \frac{3^j}{j^{1/2}\lceil n/2\rceil} + \mu_j^2(1 - \langle \Sigma \tau, \hat{\tau} \rangle^j)^2 \right)\\
&\leq \frac{\E[v_n^2 | \hat{\tau}]}{\E r_n^2 e_n^2} \left( C_y^2 C_h^2 \frac{3^{m+1}}{\lceil n/2\rceil} + \frac{\lambda_{\max}^2(\Sigma) \|\tau - \hat{\tau}\|_2^4}{4}  \sum_{j=2}^m j^2\mu_j^2  \right).
\end{align*}

Let us now consider the term $(E_1)$. The conditional variance is 
\begin{align*}
\E[E_1^2 | \mathcal{S}_2] =  \frac{(\alpha')^2}{\E r_n^2 e_n^2} (\E[\hat{\xi}_n^2(\sum_{j=2}^m (\mu_j - \hat{\mu}_j) h_j(\hat{\xi}_n))^2| \mathcal{S}_2]).
\end{align*}
Applying Stein's lemma (see the derivation of equation (\ref{stein_trick}) for details) to $\psi(\hat{\xi}_n) = \sum_{j=2}^m (\mu_j - \hat{\mu}_j) h_j(\hat{\xi}_n)$, we get
\begin{align*}
\E[E_1^2 | \mathcal{S}_2] = \frac{(\alpha')^2}{\E r_n^2 e_n^2} &(  \E[(\sum_{j=2}^m (\mu_j - \hat{\mu}_j) h_j(\hat{\xi}_n))^2 | \mathcal{S}_2] \\
&+  \E[ (\sum_{j=2}^m \sqrt{j}(\mu_j - \hat{\mu}_j) h_{j-1}(\hat{\xi}_n))^2 | \mathcal{S}_2 ] \\
&+ \E[(\sum_{j=2}^m (\mu_j - \hat{\mu}_j) h_j(\hat{\xi}_n)) (\sum_{j=2}^m \sqrt{j(j-1)} (\mu_j - \hat{\mu}_j) h_{j-2}(\hat{\xi}_n))] ) \\
\end{align*}

By the orthonormality of Hermite polynomials, this simplifies to
\begin{align*}
\E[E_1^2 | \mathcal{S}_2] &=  \frac{(\alpha')^2}{\E r_n^2 e_n^2} (\sum_{j=2}^m (\mu_j - \hat{\mu}_j)^2 + j(\mu_j - \hat{\mu}_j)^2 \\
&+\sum_{j=2}^{m-2}\sqrt{(j+1)(j+2)} (\mu_j - \hat{\mu}_j)(\mu_{j+2} - \hat{\mu}_{j+2}) ) \\
&\leq  \frac{3(\alpha')^2}{\E r_n^2 e_n^2} \sum_{j=2}^m j(\mu_j - \hat{\mu}_j)^2.
\end{align*}

Using our previous calculations for the conditional mean and variance of $\hat{\mu}_j$, we obtain
\begin{align*}
\E[E_1^2 | \hat{\tau}] &\leq  \frac{3(\alpha')^2}{\E r_n^2 e_n^2}
\sum_{j=2}^m j(\E[(\hat{\mu}_j - \E[\hat{\mu}_j|\hat{\tau}])^2 \mid \hat{\tau}]) + j (\E[\hat{\mu}_j|\hat{\tau}] - \mu_j) \\
&\leq \frac{3(\alpha')^2}{\E r_n^2 e_n^2} (C_y^2 C_h^2 \sum_{j=2}^m \frac{j 3^j}{j^{1/2}\lceil n/2\rceil }) + \frac{3(\alpha')^2}{\E r_n^2 e_n^2}(\sum_{j=2}^m j \mu_j^2 (1 - \langle \Sigma \tau, \hat{\tau} \rangle^j )^2 ) \\
&\leq \frac{3(\alpha')^2 C_y^2 C_h^2}{\E r_n^2 e_n^2} \cdot \frac{m 3^{m+1}}{\lceil n/2\rceil} + \frac{3(\alpha')^2}{\E r_n^2 e_n^2}(\lambda_{\max}^2(\Sigma)\frac{\|\tau - \hat{\tau}\|_2^4}{4} \sum_{j=2}^m j^3 \mu_j^2).
\end{align*}

\textbf{(F).} Finally, we consider the term $(F)$ by using once again the Gaussian decomposition $r_i = \alpha'' \xi_i + v_i'$ with $v_i'$ independent of $\xi_i$ to rewrite $(F)$ as
\begin{align*}
F &= \frac{1}{\sqrt{n\E r_n^2 e_n^2}} \sum_{i = 1}^n \alpha'' \xi_i \sum_{j = m+1}^\infty \mu_j h_j(\xi_i) \tag{$F_1$}\\
&+ \frac{1}{\sqrt{n\E r_n^2 e_n^2}} \sum_{i = 1}^n v_i' \sum_{j = m+1}^\infty \mu_j h_j(\xi_i) \tag{$F_2$}.
\end{align*}
The second term $(F_2)$ has variance
\begin{align*}
\E[F_2^2] = \frac{\E(v_n')^2}{\E r_n^2 e_n^2} \sum_{j=m+1}^\infty \mu_j^2. 
\end{align*}
The first term has variance
\begin{align*}
\E[F_1^2] \leq \frac{(\alpha'')^2}{\E r_n^2 e_n^2} \E \xi_n^2(\sum_{j=m+1}^\infty \mu_j h_j(\xi_n))^2.
\end{align*}
Applying Stein's lemma (see the derivation of equation (\ref{stein_trick}) for details) to the function $\psi(\xi_n) = \sum_{j=m+1}^\infty \mu_j h_j(\xi_n)$, we obtain
\begin{align*}
\E \xi_n^2(\sum_{j=m+1}^\infty \mu_j h_j(\xi_n))^2 
&= \E(\sum_{j=m+1}^\infty \mu_j h_j(\xi_n))^2 \\
&+ 2\E(\sum_{j=m+1}^\infty \sqrt{j} \mu_j h_{j-1}(\xi_n))^2\\
&+ 2\E(\sum_{j=m+1}^\infty \mu_j h_j(\xi_n))(\sum_{j=m+1}^\infty \sqrt{j(j-1)}\mu_j h_{j-2}(\xi_n))\\
&= \sum_{j=m+1}^\infty (2j + 1) \mu_j^2 + 2\sqrt{(j+1)(j+2)} \mu_j  \mu_{j+2}\\
&\leq \sum_{j=m+1}^\infty (4j+1) \mu_j^2\\
&\leq 5 \sum_{j=m+1}^\infty j \mu_j^2.
\end{align*}

Let us consider the random variables such as $\alpha, \alpha', u_n$ etc. that result from the Gaussian decomposition of $r_n$. All these term can be shown to have bounded variance. For example, let us take a closer look at $\alpha, \hat{\alpha}, u_i$ appearing in $r_n = \alpha \xi_n + \hat{\alpha} \hat{\xi}_n + u_n$. Note that by construction, $u_n$ is independent of $\xi_n$ and $\hat{\xi}_n$ given $\mathcal{S}_2$, so that
\begin{align*}
\E[r_n^2] &= \E[\E[(\alpha \xi_n + \hat{\alpha} \hat{\xi}_n + u_n)^2 | \mathcal{S}_2]] \\
&= \E[\alpha^2 + \hat{\alpha}^2 + u_n^2 + 2 \langle \Sigma \tau, \hat{\tau} \rangle].
\end{align*}
Since $\E r_n^2 \leq \E x_{n,1}^2$ and $|\langle \Sigma \tau,\hat{\tau} \rangle | \leq 1$, we have
\begin{align*}
\E \alpha^2 + \E \hat{\alpha}^2 + \E u_n^2 \leq \Sigma_{11} + 2 \leq \lambda_{\max}(\Sigma) + 2.
\end{align*}
Thus the variances of all these terms is bounded above by $C + 2$ where $C$ is by assumption (1) a constant independent of $n$, implying that these variables are all $\mathcal{O}_p(1)$. 

Ignoring constants and $\mathcal{O}_p(1)$ terms, it suffices to show that the following dominant terms converge to zero:
\begin{align}
\|\tau - \hat{\tau}\|_2^2 \sum_{j=2}^m j^2\mu_j^2 , \quad 
\|\tau - \hat{\tau} \|_2^4 \sum_{j=2}^m j^3 \mu_j^2, \quad
\sum_{j= m+1}^\infty j \mu_j^2, \quad
\frac{m 3^{m+1}}{\lceil n/2 \rceil}.
\end{align}

The last term converges to zero by the choice of $m = \lfloor \log^{\frac{2}{3}}(n) \rfloor$. For the first three terms to be negligible we need the smoothness of $g$. 
Since by assumption $\| g' \|_{L^2}^2 \leq L^2$ , Lemma \ref{decay} implies that $\sum_{j=1}^\infty j \mu_{j}^2 < L^2$, which immediately proves the term $\sum_{j= m+1}^\infty j \mu_j^2$ converges to zero.
Using the Cauchy-Schwarz inequality, we have
\begin{align*}
\sum_{j=2}^m j^2 \mu_j^2 &\leq (\sum_{j=2}^m (j\mu_j^2)^2)^\frac{1}{2} (\sum_{j=1}^m j^2)^\frac{1}{2} = \mathcal{O}(L^2 m^{\frac{3}{2}}), \quad \text{ and, } \\
\sum_{j=2}^m j^3 \mu_j^2 &\leq (\sum_{j=2}^m (j\mu_j^2)^2)^\frac{1}{2} (\sum_{j=1}^m j^4)^\frac{1}{2} = \mathcal{O}(L^2 m^{\frac{5}{2}}),
\end{align*}

Now using $\log(n) \leq \log(p)$ we obtain
\begin{align*}
m^{\frac{3}{2}} \cdot \|\hat{\tau} - \tau \|_2^2 &\leq \log(p) \cdot \|\hat{\tau} - \tau \|_2^2 \lesssim \log(p) \cdot \frac{s \log(p)}{n} \rightarrow 0.
\end{align*}
Finally, we can write
\begin{align*}
m^{\frac{5}{2}} \|\hat{\tau} - \tau \|_2^4 \leq \left( m^{\frac{3}{2}} \|\hat{\tau} - \tau \|_2^2 \right)^2 \rightarrow_p 0.
\end{align*}

\subsection{Proof of Theorem \ref{thm-efficient-unknown}}
 Before we provide the proof of Theorem \ref{thm-efficient-unknown}, we note that the assumptions of Theorem \ref{thm-efficient-unknown} (specifically, assumptions 1, 5 and 8) include the assumptions of Proposition \ref{prop1} providing error rates for the lasso estimators $\hat{\beta}, \check{\beta}$. Therefore the conclusion of Proposition \ref{prop1} is assumed throughout this section. That is, we assume that with probability $1 - o(1)$ we have
 \begin{align*}
 \| \hat{\beta} - \beta \|_1 \lesssim s \sqrt{\frac{\log(p)}{n}} , \quad \text{and} \quad \| \hat{\beta} - \beta \|_2 \lesssim \sqrt{\frac{s\log(p)}{n}},
 \end{align*}
with $\check{\beta}$ satisfying the same inequalities.

When $\Sigma$ is not known, $\hat{\mu}_1, \hat{\tau}$ are defined using $\hat{\Sigma}$ instead of $\Sigma$:
\begin{align*}
\hat{\mu}_1 := \| \hat{\Sigma}^\frac{1}{2} \hat{\beta}\|_2, \, \hat{\tau} := \hat{\mu}_1^{-1} \hat{\beta}.
\end{align*}
This normalization via $\hat{\Sigma}$ introduces a difficulty, namely that $\Sigma^{\frac{1}{2}} \hat{\tau}$ is no longer a unit vector, which also implies that $\langle x_i, \hat{\tau} \rangle$ is no longer a standard normal random variable, as opposed to the case of known $\Sigma$. This also means that the Hermite polynomials in $\langle x_i, \hat{\tau} \rangle$ are no longer orthogonal, and part 2 of Proposition \ref{hermite-prop} is no longer applicable. The following proposition generalizes the latter when $\tau, \hat{\tau}$ are not unit vectors. We use $i \equiv_2 j$ to mean $i - j$ is an even integer.

\begin{proposition}\label{non-unit-hermite-prop}
Suppose that $\{h_j\}_{j=0}^\infty$ are Hermite polynomials defined by (\ref{hermite_def}) and suppose $x \sim N(0, I)$. Then for nonrandom vectors $\alpha, \hat{\alpha}$  we have 
\begin{align*}
\E[ h_j( \langle \alpha, x \rangle ) h_k( \langle \hat{\alpha}, x \rangle ) ] = \sqrt{j! k! }  \sum_{ \stackrel{ 0 \leq i \leq \min \{ j,k \} }{ j \equiv_2 i \equiv_2 k}} \frac{ \langle \alpha, \hat{\alpha} \rangle^i \left(\frac{\| \alpha \|_2^2 - 1}{2}\right)^\frac{j-i}{2}\left(\frac{\| \hat\alpha \|_2^2 - 1}{2}\right)^\frac{k-i}{2}}{i! (\frac{j-i}{2})! (\frac{k-i}{2})! }.
\end{align*}
In particular, when $\| \alpha \|_2= 1$, we have 
\begin{align*}
\E[ h_j( \langle \alpha, x \rangle ) h_k( \langle \hat{\alpha}, x \rangle ) ] = \sqrt{\frac{k!}{j!}} \frac{ \langle \alpha, \hat{\alpha} \rangle^j \left(\frac{\| \hat\alpha \|_2^2 - 1}{2}\right)^\frac{k-j}{2}  }{(\frac{k-j}{2})!  } \mathbf{1}[ j \leq k, j \equiv_2 k].
\end{align*}
\end{proposition}
\begin{proof}
We extend O'Donnell's argument \cite[Proposition 11.31]{o2014analysis} to non-standard Gaussians, i.e. when $\alpha$ and $\hat{\alpha}$ are not unit vectors. The idea is to calculate the coefficients of the power series of a joint moment generating function in two ways and match the resulting coefficients. Let $\xi = \langle \alpha, x \rangle$ and $\hat{\xi} = \langle \hat{\alpha}, x \rangle$. Then we have
\begin{align*}
\E[\exp( s \xi + t \hat{\xi})] = \exp\left( \frac{1}{2} ( \|\alpha\|_2^2 s^2 + 2 \langle \alpha, \hat{\alpha} \rangle st + \|\hat{\alpha} \|_2^2 t^2 ) \right).
\end{align*}
Multiplying both sides by $\exp(-(s^2 + t^2)/2)$, we obtain
\begin{align*}
\E[\exp( s \xi - \frac{s^2}{2}) \exp( t \hat{\xi} -\frac{t^2}{2})] = \exp\left( \frac{1}{2} ( (\|\alpha\|_2^2-1) s^2 + 2 \langle \alpha, \hat{\alpha} \rangle st  + (\|\hat{\alpha} \|_2^2-1)t^2 ) \right).
\end{align*}
It can be shown that $\exp(s\xi - s^2/2) = \sum_{j=0}^\infty h_j(\xi)s^j / \sqrt{j!}$ \cite[\S 11.2]{o2014analysis}, and similarly for $\exp(t\hat{\xi} -t^2/2)$. 
Plugging this power series, the left-hand side is equal to
\begin{align*}
\sum_{j,k\geq0} \frac{1}{\sqrt{j!k!}} \E[h_j(\xi) h_k(\hat{\xi})] s^j t^k.
\end{align*}
To simplify notation write $\theta^2 = (\| \alpha \|_2^2 - 1)/2$ and  $\hat{\theta}^2 =( \|\hat{\alpha}\|_2^2 -1)/2$. The power series of the right-hand side is then equal to
\begin{align*}
 \sum_{l\geq 0} \frac{\theta^{2l} s^{2l}}{l!} \cdot \sum_{m\geq 0} \frac{\langle \alpha, \hat{\alpha} \rangle^m s^m t^m}{m!} \cdot \sum_{n\geq 0} \frac{\hat{\theta}^{2n} s^{2n}}{n!} &= \sum_{l,m,n\geq 0} \frac{ \langle \alpha, \hat{\alpha} \rangle^m  \theta^{2l} \hat{\theta}^{2n} }{l! m! n! } s^{m + 2l} t^{m + 2n}  \\
&= \sum_{j,k\geq 0} \left( \sum_{\stackrel{ 0 \leq i \leq \min\{j,k\}}{j \equiv_2 i \equiv_2 k} }  \frac{\langle \alpha, \hat{\alpha} \rangle^i \theta^{j-i}  \hat{\theta}^{k-i} }{ i! \left(\frac{j-i}{2} \right)! \left( \frac{k-i}{2} \right)!} \right) s^j t^k
\end{align*}
where in the last equality we use the convention that an empty sum is equal to zero (which happens when $ j \not\equiv_2 k$). Equating the corresponding coefficients of $s^j t^k$ in the last two displays finishes the proof of the first assertion, from which the second part also follows immediately.
\end{proof}
\begin{lemma}\label{inner-prod-bound-lemma}
Suppose that $j \equiv_2 k$ and $\omega\geq0$. We have
\begin{align*}
\sum_{\stackrel{0 \leq j \leq k -2 }{j \equiv_2 k }} \sqrt{\frac{k!}{j!}} \frac{\omega^\frac{k-j}{2}  }{(\frac{k-j}{2})!  } \leq 4 \omega e^{\omega k} \quad \text{and} \quad
\sum_{\stackrel{0 \leq j \leq k -2 }{j \equiv_2 k }} \frac{k!}{j!} \frac{\omega^{k-j}}{( \frac{k-j}{2})!^2} \leq 16 \omega^2 e^{\omega^2 k^2}.
\end{align*}
\end{lemma}
\begin{proof}
It is easy to verify that $(j+2)(j+1)/(k-j)^2 \leq 4 $ for all $j \leq k -2$, and therefore,
\begin{align}\label{ineq:factorial-ratio}
\frac{k!}{j!} \cdot  \frac{4}{(k-j)^2 } \leq 4  k\cdot (k-1) \cdots (j+3) \cdot \frac{(j+2)(j+1)}{(k-j)^2} \leq 16 k^{k - j - 2}, \quad \forall j \leq k - 2.
\end{align}
Using this inequality we can write
\begin{align*}
\sum_{\stackrel{0 \leq j \leq k -2 }{j \equiv_2 k }} \sqrt{\frac{k!}{j!}} \frac{\omega^\frac{k-j}{2}  }{(\frac{k-j}{2})!  } &\leq
4 \omega  \sum_{\stackrel{0 \leq j \leq k -2 }{j \equiv_2 k }} \frac{(k \omega)^\frac{k-2-j}{2} }{ (\frac{k-2-j}{2})! } \leq 4 \omega e^{k\omega},
\end{align*}
which completes the proof of the first inequality. The second inequality is proved similarly using (\ref{ineq:factorial-ratio}).
\end{proof}

\begin{lemma}\label{lemma:mean-squares}
Suppose that $2 \leq j \leq m$ and $x \sim N(0,\Sigma)$. Assume $\|\Sigma^\frac12 \tau \|_2 = 1$ and $\|\Sigma^\frac12 \hat{\tau} \|_2 = 1 + \mathcal{O}_p(1/\sqrt{n})$ and that $m = o(\sqrt{n})$. Then
\begin{align*}
\E[\left(\alpha^q h_j( \langle \tau, x \rangle) - \hat\alpha^q h_j(\langle \hat{\tau}, x \rangle) \right)^2 \mid \hat\tau, \alpha, \hat\alpha] &= \mathcal{O}_p \left( j \sqrt{\frac{s \log(p)}{n}}  \right), \quad \forall q \in \{ 0, 1, 2\}.
\end{align*}
\end{lemma}
\begin{proof}
We first prove the equality for $q = 0$. Using Proposition \ref{non-unit-hermite-prop} and inequality (\ref{ineq:factorial-ratio}) we have
\begin{align*}
\bigl| \E[ h_j^2 (\langle \hat{\tau}, x \rangle) \mid \hat\tau]  -& \|\Sigma^\frac12 \hat{\tau} \|_2^{2j}  \bigr| =\left|  j! \sum_{\stackrel{0 \leq i \leq j -2}{i \equiv_2 j}}  \frac{ \|\Sigma^\frac12 \hat{\tau} \|_2^{2i}  \left(\frac{\| \Sigma^\frac12 \hat{\tau} \|_2^2 - 1}{2}\right)^{j-i} }{i! (\frac{j-i}{2})!^2 } \right| \\
\leq&  (1 \lor \|\Sigma \hat{\tau}\|_2^{2(j-2)} ) \left| \frac{\| \Sigma^\frac12 \hat{\tau} \|_2^2 - 1}{2}\right|^2 \sum_{\stackrel{0 \leq i \leq j-2}{i \equiv_2 j}}  \frac{ 16 j^{ j - i -2 } \left|\frac{\| \Sigma^\frac12 \hat{\tau} \|_2^2 - 1}{2}\right|^{j-2-i} }{(\frac{j-2-i}{2})! }\\
\leq& 16 (1 \lor \|\Sigma \hat{\tau}\|_2^{2(j-2)} ) \left| \frac{\| \Sigma^\frac12 \hat{\tau} \|_2^2 - 1}{2}\right|^2 \exp\left( j^2  \left| \frac{\| \Sigma^\frac12 \hat{\tau} \|_2^2 - 1}{2}\right|^2 \right)\\
=& \mathcal{O}_p\left( \frac{1}{n} \right).
\end{align*}
Similarly, using the second part of Proposition \ref{non-unit-hermite-prop} we have
\begin{align*}
\E[h_j(\langle \tau, x \rangle) h_j(\langle \hat{\tau}, x \rangle) \mid \hat{\tau}] &= \langle \Sigma^\frac12 \tau, \Sigma^\frac12 \hat{\tau} \rangle^j \\
&= \left( \langle \Sigma^\frac12\tau , \Sigma^\frac12 \tau \rangle + \langle  \Sigma^\frac12 \tau,  \Sigma^\frac12 (\hat{\tau} - \tau) \rangle \right)^j \\
&= \left( 1 + \mathcal{O}_p\left( \sqrt{\frac{s \log(p)}{n}} \right) \right)^j = 1 + \mathcal{O}_p \left( j \sqrt{\frac{s \log(p)}{n}}  \right).
\end{align*}
Putting together these approximations yields
\begin{align*}
\E[(h_j( \langle \tau, x \rangle) - h_j(\langle \hat{\tau}, x \rangle) )^2 \mid \hat{\tau} ] &= 1 + \| \Sigma^\frac12 \hat{\tau} \|_2^{2j} - 2 ( 1 +  \mathcal{O}_p \left( j \sqrt{\frac{s \log(p)}{n}}  \right) )\\
&=\mathcal{O}_p \left( j \sqrt{\frac{s \log(p)}{n}}  \right).
\end{align*}

To prove the result for $q = 1,2$, we use the result we just proved for $q = 0$. Add and subtract $\hat\alpha^q h_j(\langle \tau, x \rangle)$ and use the inequality $(a+b)^2 \leq 2a^2 + 2b^2$ to write
\begin{align*}
\E[\left(\alpha^q h_j( \langle \tau, x \rangle) - \hat\alpha^q h_j(\langle \hat{\tau}, x \rangle) \right)^2 \mid \hat\tau, \alpha, \hat\alpha] 
&\leq 2 \hat\alpha^{2q} \E[ \left(h_j( \langle \tau, x \rangle) - h_j(\langle \hat{\tau}, x \rangle) \right)^2]\\
&+ 2(\alpha^q - \hat\alpha^q)^2 \E[h_j^2(\langle \tau, x \rangle) \mid \hat{\tau}]\\
&= 2\hat{\alpha}^{2q} \mathcal{O}_p\left( j\sqrt{\frac{s \log(p)}{n}} \right) + 2 \mathcal{O}_p\left( \frac{s \log(p)}{n} \right)\\
&=  \mathcal{O}_p\left( j\sqrt{\frac{s \log(p)}{n}} \right).
\end{align*}
\end{proof}

Even though $\Sigma^\frac12 \hat{\tau}$ is not a unit vector, it is easy to see that $T = \| \hat{\Sigma}^\frac12 \hat{\beta} \|_2^2 / \| \Sigma^\frac12 \hat{\beta} \|_2^2 - 1$ has mean zero (given $\hat{\tau}$), and that $\E[T^2 \mid \hat{\tau}] \lesssim 1/ n$. Thus $T = \mathcal{O}_p(1/\sqrt{n})$ and therefore 
\begin{align}\label{tau-hat-norm}
\| \Sigma^\frac{1}{2} \hat{\tau} \|_2  = 1 + \mathcal{O}_p(1/\sqrt n). 
\end{align}
Furthermore, it can be shown as before that $\| \hat{\tau} - \tau \|_2$ has the same rate of convergence as $\| \hat{\beta} - \beta \|_2$, i.e. both errors are $\mathcal{O}_p( \sqrt{s \log(p) / n})$ (This follows from two applications of the triangle inequality and the assumption that $|\mu_1| = \| \Sigma^\frac12 \beta\|_2$ is bounded away from zero).

1. We start by calculating the mean and variance of $\mu_k$ for $k\geq 2$. 
The mean can be computed using Proposition \ref{non-unit-hermite-prop} and noting that $\Sigma^\frac12 \tau$ is a unit vector as follows:
\begin{align*}
\E [\hat{\mu}_k \mid \hat{\tau} ] &= \E [y_n h_k( \langle x_n, \hat{\tau} \rangle) \mid \hat{\tau}]\\
&= \sum_{j=0}^\infty \mu_j \E[ h_j(\langle \Sigma^\frac12 \tau, \Sigma^{-\frac12} x_i \rangle) h_k( \langle \Sigma^\frac12\hat{\tau}, \Sigma^{-\frac12} x_i \rangle) ]\\
&= \langle \Sigma^\frac12 \tau, \Sigma^\frac12 \hat{\tau} \rangle^k \mu_k +   \sum_{\stackrel{j \leq k - 2}{j \equiv_2 k}} \sqrt{\frac{k!}{j!}} \frac{ \langle \Sigma^\frac12 \tau, \Sigma^\frac12 \hat{\tau} \rangle^j \left(\frac{\| \Sigma^\frac12 \tau' \|_2^2 - 1}{2}\right)^\frac{k-j}{2}  }{(\frac{k-j}{2})!  } \mu_j.
\end{align*}
Therefore using Lemma \ref{inner-prod-bound-lemma} we have
\begin{align*}
|\E [\hat{\mu}_k \mid \hat{\tau} ] - \mu_k| &\leq | (\langle \Sigma^\frac12 \tau, \Sigma^\frac12 \hat{\tau} \rangle^k - 1) \mu_k|
\\ 
&+ (\max_j |\mu_j|) \left| \frac{\| \Sigma^\frac12 \tau' \|_2^2 - 1}{2} \right| (1 \lor | \langle \Sigma^\frac12 \tau, \Sigma^\frac12 \hat{\tau} \rangle|^k )  \exp\left(k \left|\frac{\| \Sigma^\frac12 \tau' \|_2^2 - 1}{2}\right| \right).
\end{align*}
Note that the last sum in the right-hand side of above inequality is of order $\exp(\mathcal{O}_p(k/\sqrt{n})) = \mathcal{O}_p(1)$.
Since $\| \Sigma^\frac12 \hat{\tau} \|_2 = 1 + \mathcal{O}_p(1/\sqrt{n})$ and $\| \Sigma^\frac{1}{2} \tau \|_2 = 1$, it follows that 
\begin{align}
\left| \langle \Sigma^\frac12 \tau, \Sigma^\frac12 \hat{\tau} \rangle^k - 1 \right| &\leq \left|  \frac{1}{2} \left( \| \Sigma^\frac12 \tau \|_2^2 + \| \Sigma^\frac12 \hat{\tau} \|_2^2 - \| \Sigma^\frac12 (\tau - \hat{\tau} )\|_2^2 \right)^k - 1 \right|\\
&\leq  \frac{k}{2} (1 -   \| \Sigma^\frac12 \hat{\tau} \|_2^2 + \| \Sigma^\frac12 (\tau - \hat{\tau} )\|_2^2 ) \quad (\text{with high probability})\\
&= \mathcal{O}_p\left(k\sqrt{\frac{s\log(p)}{n}} \right),
\end{align}
where the second inequality follows from $1 - (1 - \theta)^k \leq k \theta$ for $\theta \in [0,1]$. Note that by Lemma \ref{decay} we have $|\mu_j| \leq L$ for all $j \geq 1$ since by assumption $\| g' \|_{L_2(N(0,1))} \leq L$.
Putting together these inequalities yields
\begin{align*}
|\E [\hat{\mu}_k \mid \hat{\tau} ] - \mu_k| =  \mathcal{O}_p\left(L k\sqrt{\frac{s\log(p)}{n}}+ L \frac{1}{\sqrt{n}}  \right).
\end{align*}
Next we turn to the variance of $\hat{\mu}_k$. Use the Cauchy-Schwarz inequality to write
\begin{align*}
\V(\hat{\mu}_k \mid \hat{\tau} ) &\lesssim \frac{1}{n} \E[ y_n^2 h_k^2( \langle x_n, \hat{\tau} \rangle)] \leq \frac{1}{n} \sqrt{\E [y_n^4]}  \cdot \sqrt{\E[ h_k^4(\langle x_n, \hat{\tau} \rangle)]}. 
\end{align*}
Since\footnote{See the first remark following Theorem \ref{thm-unknown-sigma} for details.} $\| y_n\|_{\psi_2} \lesssim \| z_n \|_{\psi_2} \leq \sigma_z <\infty$ with $\sigma_z$ bounded away from $+\infty$ by assumption, the fourth moment of $y_n$ is also uniformly (in $n$) bounded. Thus we need to bound the fourth moment of $h_k(\langle x_n, \hat{\tau} \rangle )$. Since $\langle x_n, \hat{\tau} \rangle$ is not a standard normal random variable, we can not directly use Proposition \ref{hermite-prop} to bound the moments of $h_k(\langle x_n, \hat{\tau} \rangle)$. However, it is possible to relate these moments to those of $h_k(\xi)$ for $\xi \sim N(0,1)$ as follows. Let us write $\omega^2 = \| \Sigma^\frac12 \hat{\tau} \|_2^2 $ for the conditional variance of $\langle x_n, \hat{\tau} \rangle$ given $\hat{\tau}$. For large $n$ we have shown that $\omega^2$ is close to one with high probability. On the high-probability event $[\omega^2 < 1.5]$ we can write
\begin{align*}
\E[ h_k^4(\langle x_n, \hat{\tau} \rangle)] &= \frac{1}{\sqrt{2\pi}\omega} \int_{-\infty}^{+\infty} h_k^4( \xi) e^\frac{-\xi^2}{2\omega^2} \diff \xi\\
&= \frac{1}{\sqrt{2\pi}\omega} \int_{-\infty}^{+\infty} \left( h_k^4( \xi) e^\frac{-\xi^2}{4} \right) \left( e^{\frac{-\xi^2}{2 \omega^2} +\frac{\xi^2}{4} } \right) \diff \xi\\
&\leq  \frac{1}{\sqrt[4]{2\pi}\omega} \left(  \int_{-\infty}^{+\infty}  h_k^8( \xi) \frac{e^\frac{-\xi^2}{2}}{\sqrt{2\pi}} \diff \xi \right)^\frac12 \cdot  \left( \int_{-\infty}^{+\infty}  e^{-\frac{\xi^2}{2}\left( \frac{2}{\omega^2} - 1 \right)} \diff \xi \right)^\frac12\\
&= (2 - \omega^2)^{-\frac{1}{2}}\| h_k \|_{L_8(N(0,1))}^4,
\end{align*}
where the inequality uses the Cauchy-Schwarz inequality. By the fifth part of Proposition \ref{hermite-prop}, the $L_8$ norms of Hermite polynomials satisfy $\| h_k \|_{L_8} \lesssim 7^k / \sqrt[4]{k}$. Putting together the last two displays, we obtain (with high probability)
\begin{align*}
\V(\hat{\mu}_k \mid \hat{\tau} ) &\leq \frac{1}{n} C_y^2 \cdot (2 - \omega^2)^{-\frac{1}{4}} \| h_k \|_{L_8}^2 = \mathcal{O}_p\left( \frac{7^{2k}}{n\sqrt{k}} \right).
\end{align*}
To summarize our calculations, we have
\begin{align}\label{hermite-coef-bias-variance}
|\E [\hat{\mu}_k \mid \hat{\tau} ] - \mu_k| =  \mathcal{O}_p\left(L k\sqrt{\frac{s\log(p)}{n}}\right) \quad \text{and} \quad \V(\hat{\mu}_k \mid \hat{\tau} ) = \mathcal{O}_p\left( \frac{7^{2k}}{n\sqrt{k}} \right).
\end{align}
Given the (conditional) bias and variance of the $\mu_k$'s, we next consider the various terms in the following error decomposition of $\tilde{\beta}_1$:
\begin{align*}
\sqrt{n} (\frac{1}{n} \sum_{i=1}^n \hat r_i x_{i1})(\tilde{\beta}_1 - \beta_1) &= \frac{1}{\sqrt{n }} \sum_{i=1}^n \hat r_i e_i  \tag{A}\\
&+ \frac{1}{\sqrt{n }} \sum_{i=1}^n \hat r_i (\mu_0 - \hat{\mu}_0) \tag{B} \\
&+ \frac{1}{\sqrt{n }} \sum_{i=1}^n \hat r_i  \langle \mu_1\tau_{-1} - \check{\beta}_{-1}, x_{i,-1}\rangle \tag{C} \\
&+\frac{1}{\sqrt{n }} \sum_{i=1}^n \hat r_i  \sum_{j=2}^m \mu_j [h_j(\langle \tau, x_i \rangle) - h_j(\langle \hat{\tau}, x_i \rangle)] \tag{D}\\
&+\frac{1}{\sqrt{n }} \sum_{i=1}^n \sum_{j=2}^m \hat r_i(\mu_j - \hat{\mu}_j) h_j(\langle \hat{\tau}, x_i \rangle) \tag{E}\\
&+ \frac{1}{\sqrt{n }} \sum_{i = 1}^n \hat r_i\sum_{j = m+1}^\infty \mu_j h_j(\langle \tau, x_i \rangle). \tag{F}
\end{align*}

 \textbf{Step 1: (A). } The first term $A$ can be shown to converge in distribution to $N(0,\E[r_n^2 e_n^2])$ using the same argument as in the proof of Theorem \ref{thm-unknown-sigma}. More precisely, we have
 \begin{align*}
  A = \frac{1}{\sqrt{n }} \sum_{i=1}^n \hat r_i e_i = \sqrt{\E[r_n^2 e_n^2]} W_n,  \text{ where } W_n \rightarrow_d N(0,1).
 \end{align*}
 
 \textbf{Step 2: (B).} Note that $\E[\hat{r}_i \mid \hat{\gamma} ] = 0$. Let $\tilde{\gamma}^T := (1, - \hat{\gamma}^T)$ and $\tilde{\omega}^T:= (1, -\gamma^T) = \E[r_n^2] \Omega_1^T$ where $\Omega_1$ is the first column of $\Sigma^{-1}$. We have
 \begin{align*}
 \V[ \frac{1}{\sqrt{n}} \sum_{i=1}^n \hat{r}_i \mid \hat{\gamma} ] &= \E[ \langle \tilde{\gamma}, x_n \rangle^2 \mid \tilde{\gamma}]\\
 &= \| \Sigma^\frac12  \tilde{\gamma} \|_2^2
 \end{align*}
Use the triangle inequality and the inequality $\E[r_n^2] \leq \lambda_{\max}(\Sigma)$ from Lemma \ref{residual_prop}  to write
\begin{align*}
 \| \Sigma^\frac12  \tilde{\gamma} \|_2&\leq  \| \Sigma^\frac12 \tilde{\omega} \|_2 + \| \Sigma^\frac12 (\tilde{\gamma} - \tilde{\omega}) \|_2\\
 &\leq \frac{1}{\sqrt{\lambda_{\min}(\Sigma)}} \| \Sigma \tilde{\omega} \|_2 + 
 \sqrt{\lambda_{\max}(\Sigma)} \| \hat{\gamma} - \gamma \|_2 \\
 &\lesssim \frac{\lambda_{\max}(\Sigma) }{\sqrt{\lambda_{\min}(\Sigma)}} + 
 \sqrt{\lambda_{\max}(\Sigma)} \sqrt{ \frac{s \log(p)}{n}},
 \end{align*}
which is uniformly bounded above. This shows that $\sum_1^n \hat{r}_i / \sqrt{n} = \mathcal{O}_p(1)$. Since $y_n$ has a finite fourth moment by assumption, it follows that $\hat{\mu}_0 - \mu_0 = \mathcal{O}_p(1/\sqrt{n})$. Therefore $B = \mathcal{O}_p(1/\sqrt{n})$ as well.

 \textbf{Step 3: (C).} Note that since $\mu_1 \tau = \beta$, this term is shown to be $o_p(1)$ in the proof of Theorem \ref{thm-unknown-sigma}. (Note that $\check{\beta}$ was computed on the first subsample, similarly to how $\hat{\beta}$ was computed in the setting of Theorem \ref{thm-unknown-sigma}.)

\textbf{Step 4: (E).} Let $\bar{\gamma} = (1, -\gamma^T)^T$ and $\tilde{\gamma} = (1, -\hat\gamma^T)^T$.
Note that since $\| \tilde{\gamma} - \gamma \|_2 \rightarrow_p 0$, the event $H = [\| \tilde{\gamma} - \gamma \|_2 < 1]$ has probability converging to 1. Therefore, we have $E = o_p(1)$ if and only if $E \mathbf{1}_H = o_p(1)$. Further more, on $H$ we have
\begin{align*}
\sqrt{\E[ \hat{r}_n^2 \mid \mathcal{S}_2] \mathbf{1}_H}  &= \| \Sigma^\frac12 \tilde{\gamma} \|_2 \mathbf{1}_H \\
&\leq \| \Sigma^\frac12 \bar{\gamma} \|_2+ \| \Sigma^\frac12 (\tilde{\gamma} - \bar{\gamma})\|_2 \mathbf{1}_H \\
&\leq \E[r_n^2]
+ \lambda_{\max}(\Sigma)^\frac12 \\
&\leq \lambda_{\max}(\Sigma)  + \lambda_{\max}(\Sigma)^\frac12 
\end{align*}
which is uniformly bounded above. 

We can write
\begin{align*}
E' := E \mathbf{1}_H = \frac{ \mathbf{1}_H}{\sqrt{n }} \sum_{i=1}^n \sum_{j=2}^m \hat{r}_i(\mu_j - \hat{\mu}_j) h_j(\langle \hat{\tau}, x_i \rangle)
\end{align*}
The conditional mean of $E'$ is given by
\begin{align*}
\E[E' \mid \mathcal{S}_2] &= \sqrt{n}  \mathbf{1}_H \sum_{j=2}^m (\mu_j - \hat{\mu}_j) \E[ \hat{r}_n h_j(\langle \hat{\tau}, x_n \rangle) \mid \mathcal{S}_2]\\
 &=\langle \Sigma^\frac12 \tilde{\gamma} ,  \Sigma^\frac12 \hat{\tau} \rangle   \mathbf{1}_H   \sqrt{n}   \sum_{\stackrel{3 \leq j \leq m}{j \equiv_2 1}}  \sqrt{j!} \left( \frac{ \| \Sigma^\frac12 \hat{\tau} \|_2^2 - 1}{2} \right)^{\frac{j-1}{2}}  \left( \left(\frac{j - 1}{2}\right)! \right)^{-1}  (\mu_j - \hat{\mu}_j) .
\end{align*}
Using the tower property of conditional expectations and the independence of $\hat{\gamma}$ an $\hat{\mu}_j$, 
\begin{align*}
\E&[E' \mid \hat{\tau} ] =\E[ \E[ E' \mid \mathcal{S}_2] \mid \hat{\tau}]\\
  &=\E[\langle \Sigma^\frac12 \tilde{\gamma} ,  \Sigma^\frac12 \hat{\tau} \rangle   \mathbf{1}_H \mid \hat{\tau}]   \sqrt{n}   \sum_{\stackrel{3 \leq j \leq m}{j \equiv_2 1}}  \sqrt{j!} \left( \frac{ \| \Sigma^\frac12 \hat{\tau} \|_2^2 - 1}{2} \right)^{\frac{j-1}{2}} \left(\frac{j - 1}{2}\right)!^{-1} \E[ (\mu_j - \hat{\mu}_j) \mid \hat{\tau}]\\
 &= \mathcal{O}_p\left( \sqrt{n}    \sum_{3 \leq j \leq m, j \equiv_2 1}  \sqrt{j!} \left(\frac{1}{\sqrt{n}}\right)^{\frac{j-1}{2}} \left(\left(\frac{j - 1}{2}\right)!\right)^{-1} \left(j \sqrt{\frac{s \log(p)}{n}} \right)
\right)
\end{align*}
Using the inequality $j!  \lesssim m^{j-3}$, we can upper bound the sum by an exponential:
\begin{align*}
\sqrt{n}    \sum_{3 \leq j \leq m, j \equiv_2 1}  j \sqrt{j!} \left(\frac{1}{\sqrt{n}}\right)^{\frac{j-1}{2}} \left(\frac{j - 1}{2}\right)!^{-1}  & \lesssim \sum_{j=3}^m  \left(\frac{m}{\sqrt{n}}\right)^{\frac{j-3}{2}} \left( \left(\frac{j - 3}{2}\right)!\right)^{-1}  \leq e^\frac{m}{\sqrt{n}}.
\end{align*}
It follows that $\E[E' \mid \hat{\tau}] = \mathcal{O}_p( \sqrt{s \log(p) / n} )$.

Next consider the conditional variance of $E'$. Using the tower property of conditional expectations and the inequality $(\sum_1^m a_j)^2 \leq m \sum_1^m a_j^2$, we have
\begin{align*}
\V( E' \mid \hat{\tau}) &\leq \E\left[   \E\left[ \hat{r}_n^2 \left( \sum_{j=2}^m (\mu_j - \hat{\mu}_j) h_j(\hat{\xi}) \right)^2   \mathbf{1}_H \given  \mathcal{S}_2\right] \given \hat{\tau} \right] \\
&\leq  \E\left[   m \sum_{j=2}^m (\mu_j - \hat{\mu}_j)^2  \E\left[ \hat{r}_n^2 h_j^2(\hat{\xi})    \given  \mathcal{S}_2\right]    \mathbf{1}_H \given \hat{\tau} \right].
\end{align*}
Write $\hat{\xi}_n = \hat{\alpha} \hat{r}_n + \hat{u}_n$ where $\hat{u}_n$ is conditionally independent of $\hat{r}_n$ given $\mathcal{S}_2$. The inner expectation can be rewritten via Stein's lemma as follows
\begin{align*}
\E\left[ \hat{r}_n^2 h_j^2(\hat{\xi}) \given  \mathcal{S}_2\right]  &  \mathbf{1}_H=
 \E[\hat{r}_n^2 \mid \mathcal{S}_2] \E[ h_j^2(\hat{\xi}_n) \mid \mathcal{S}_2]   \mathbf{1}_H+
2(\E[\hat{r}_n^2 \mid \mathcal{S}_2])^2 \E[ (h_j'(\hat{\xi}_n))^2 \mid \mathcal{S}_2]   \mathbf{1}_H \\
&+ 2(\E[\hat{r}_n^2 \mid \mathcal{S}_2])^2 \E[ h_j(\hat{\xi}_n)h_j''(\hat{\xi}_n) \mid \mathcal{S}_2]   \mathbf{1}_H \\
&\lesssim \E[ h_j^2(\hat{\xi}_n) \mid \hat{\tau}]+
 j \hat{\alpha}^2 \E[ (h_{j-1}^2(\hat{\xi}_n) \mid \hat{\tau}] + \sqrt{j(j-1)}  \hat{\alpha}^2 \E[ h_j(\hat{\xi}_n)h_{j-2}(\hat{\xi}_n) \mid \hat{\tau}]\\
 &\lesssim
  \E[ h_j^2(\hat{\xi}_n) \mid \hat{\tau}]+
 j \| \Sigma^\frac12 \hat{\tau} \|_2^2 \left(\E[ (h_{j-1}^2(\hat{\xi}_n) \mid \hat{\tau}] + \E[ h_j(\hat{\xi}_n)h_{j-2}(\hat{\xi}_n) \mid \hat{\tau}] \right)
\end{align*}
where the first inequality uses the boundedness of $\E[\hat{r}_n^2 \mid \mathcal{S}_2]$ on $H$ and the second inequality uses the following upper bound on $\hat{\alpha}^2$:
\begin{align*}
\hat{\alpha}^2  =  \left( \frac{\tilde{\gamma}^T \Sigma \hat{\tau}}{\tilde{\gamma}^T \Sigma \tilde{\gamma} } \ \right)^2\leq   \frac{ \| \Sigma^\frac12 \tilde{\gamma} \|_2^2 \| \Sigma^\frac12 \hat{\tau} \|_2^2 }{ \| \Sigma^\frac12 \tilde{\gamma} \|_2^4 }  \leq   \frac{ \| \Sigma^\frac12 \hat{\tau} \|_2^2  }{ \lambda_{\max}(\Sigma) } \lesssim \| \Sigma^\frac12 \hat{\tau} \|_2^2 
\end{align*}
Plugging these back into the upper bounded on the conditional variance of $E'$ yields the following $\lesssim$-upper bound on $\V[E' \mid \hat{\tau}]$:
\begin{align*}
m \sum_{j=2}^m  \E\left[ (\mu_j - \hat{\mu}_j)^2    \given \hat{\tau} \right] 
\left[ \E[ h_j^2(\hat{\xi}_n) \mid \hat{\tau}] +
 j \| \Sigma^\frac12 \hat{\tau} \|_2^2 \left(\E[ (h_{j-1}^2(\hat{\xi}_n) + h_j(\hat{\xi}_n)h_{j-2}(\hat{\xi}_n) \mid \hat{\tau}] \right) \right].
 \end{align*}
Next we find bounds on each of these terms. Use Proposition \ref{non-unit-hermite-prop} to write
\begin{align*}
\left| \E[h_j^2(\hat{\xi})\mid \hat{\tau}]  - 1 \right|&= \left| -1 +  j! \sum_{\stackrel{0 \leq i \leq j}{i \equiv_2 j}} \frac{ \| \Sigma^\frac12 \hat{\tau} \|_2^{2i} \left( \frac{\| \Sigma^\frac12 \hat{\tau} \|_2^2 - 1}{2} \right)^{j-i}}{i! \left( \frac{j - i}{2} \right)!^2} \right|\\
\leq&  \left| \| \Sigma^\frac12 \hat{\tau} \|_2^{2j} - 1 \right| +  \left( 1 \lor \| \Sigma^\frac12 \hat{\tau} \|_2^{2j}  \right)     \sum_{\stackrel{0 \leq i \leq j-2}{i \equiv_2 j}} \frac{ j! \left| \frac{\| \Sigma^\frac12 \hat{\tau} \|_2^2 - 1}{2} \right|^{j-i}}{i! \left( \frac{j  - i}{2} \right)!^2}\\
\leq&   \left| \| \Sigma^\frac12 \hat{\tau} \|_2^{2j} - 1 \right| +  \left( 1 \lor \| \Sigma^\frac12 \hat{\tau} \|_2^{2j}  \right)   
 \left| \frac{\| \Sigma^\frac12 \hat{\tau} \|_2^2 - 1}{2} \right|^2 \exp\left( j^2  \left| \frac{\| \Sigma^\frac12 \hat{\tau} \|_2^2 - 1}{2} \right|^2 \right)\\
 =& \mathcal{O}_p\left( \frac{j}{\sqrt{n}} \right)
\end{align*}
where we used Lemma \ref{inner-prod-bound-lemma} to obtain the last inequality. A similar argument and using Proposition \ref{non-unit-hermite-prop} and Lemma \ref{inner-prod-bound-lemma} shows that  $\E[h_j(\hat{\xi}_n) h_{j-2}(\hat{\xi}_n)\mid \hat{\tau} ] = \mathcal{O}_p( 1/ n )$. Using the bias-variance result \ref{hermite-coef-bias-variance} for $\hat{\mu}_j$ we have
\begin{align*}
\E\left[ (\mu_j - \hat{\mu}_j)^2    \given \hat{\tau} \right]  = \E\left[ ( \hat{\mu}_j - \E[ \hat{\mu}_j \mid \hat{\tau}] )^2    \given \hat{\tau} \right]  + ( \E[ \mu_j \mid \hat{\tau} ] - \mu_j)^2 = 
\mathcal{O}_p \left( \frac{ 7^{2j}}{n \sqrt{j} } + j^2 \frac{s \log(p)}{n} \right).
\end{align*}
Plugging these back into the upper bound on the conditional variance of $E'$ yields
\begin{align*}
\V[E' \mid \hat{\tau}] &= \mathcal{O}_p\left( m \sum_{j=2}^m \left(\frac{ 7^{2j}}{n \sqrt{j} } + j^2 \frac{s \log(p)}{n}\right) 
\left( 1 + \frac{j}{\sqrt{n} }+ j (1 + \frac{j}{\sqrt{n}}) + \frac{j}{n} \right) \right)\\
&= \mathcal{O}_p\left(  m \sum_{j=2}^m \frac{  7^{3j}}{n } + j^3 \frac{s \log(p)}{n} \right) = \mathcal{O}_p\left( \frac{m 7^{3m}}{n} + \frac{m^4 s \log(p)}{n} \right),
\end{align*}
which is negligible for $m = \lfloor \log^\frac{1}{4}(n) \rfloor$ since $m^4 \leq \log(n) \leq \log(p)$.

\textbf{Step 5: (D).} 
To compute the bias and variance of $D$ we consider the cases $j=3$ and $j\neq 3$ separately as they have different rates and require different approaches. 
Let us write $\tilde{\gamma}^T = (1, -\hat{\gamma}^T)$ so that $\hat{r}_i = \langle \tilde{\gamma}, x_i \rangle$ and $\E [\hat{r}_i^2 \mid \mathcal{S}_2] = \|\Sigma^\frac12 \tilde{\gamma} \|_2^2$. Also define
\begin{align*}
D^{(3)} = \frac{\mu_3}{\sqrt{n }} \sum_{i=1}^n \hat r_i  [h_3(\xi_i) - h_3(\hat{\xi}_i)] .
\end{align*}
First we show that $D^{(3)} - \E[D^{(3)} \mid \mathcal{S}_2] = o_p(1)$ by showing that its second conditional moment $\E[(D^{(3)} - \E[D^{(3)} \mid \mathcal{S}_2] )^2 \mid \mathcal{S}_2] =\V[D^{(3)} \mid \mathcal{S}_2]$ is  $o_p(1)$. 

Write $\xi_n = \alpha \hat{r}_n + u_n$ and $\hat{\xi}_n = \hat{\alpha} \hat{r}_n + \hat{u}_n$ where 
\begin{align*}
\alpha = \frac{\tau^T \Sigma \tilde{\gamma}}{\tilde{\gamma}^T \Sigma \tilde{\gamma}} , \text{ and } \, \hat\alpha = \frac{\hat\tau^T \Sigma \tilde{\gamma}}{\tilde{\gamma}^T \Sigma \tilde{\gamma}}.
\end{align*}
With this choice of $\alpha, \hat\alpha$ it is easy to see that $u_n, \hat u_n$ are conditionally independent of $\hat r_n$ given $\mathcal{S}_2$. Also both $\alpha$ and $\hat{\alpha}$ are $\mathcal{O}_p(1)$ as by the Cauchy-Schwarz inequality we have
\begin{align*}
|\alpha| \leq \frac{\|\Sigma^\frac12 \tilde{\gamma}\|_2 \| \Sigma^\frac12 \tau \|_2}{\|\Sigma^\frac12 \tilde{\gamma} \|_2^2} \leq \frac{1}{ \lambda_{\min}(\Sigma)^{\frac12} \| \tilde{\gamma} \|_2} \leq \lambda_{\min}(\Sigma)^{-\frac12},
\end{align*}
since the first coordinate of $\tilde{\gamma}$ is $1$. A similar argument shows that $\hat{\alpha} = \mathcal{O}_p(1)$ and furthermore, that $|\alpha - \hat\alpha | \lesssim \| \tau - \hat\tau \|_2 = \mathcal{O}_p(\sqrt{s \log(p)/n})$.

Using Stein's lemma twice after conditioning on $\mathcal{S}_2, u_n, \hat u_n$ and using the tower property of conditional expectations, we obtain
\begin{align*}
\E[ \hat r_n^2 (h_3(\xi_n) - & h_3(\hat{\xi}_n) )^2 \mid \mathcal{S}_2]  = \E[r_n^2 \mid \mathcal{S}_2] \E[ (h_3(\xi_n) - h_3(\hat{\xi}_n) )^2 \mid \mathcal{S}_2]\\
&+2(\E[r_n^2 \mid \mathcal{S}_2])^2  \E[ (\sqrt{3} \alpha h_2(\xi_n) - \sqrt{3} \hat \alpha h_2(\hat{\xi}_n) )^2 \mid \mathcal{S}_2]\\
&+2(\E[r_n^2 \mid \mathcal{S}_2])^2 \E[ (h_3(\xi_n) - h_3(\hat{\xi}_n) )(\sqrt{6} \alpha^2  h_1(\xi_n) - \sqrt{6} \hat\alpha^2 h_1(\hat{\xi}_n) ) \mid \mathcal{S}_2]
\end{align*}
The first two summands are $\mathcal{O}_p( \sqrt{s \log(p) / n})$ by Lemma \ref{lemma:mean-squares}. The third term is also seen to have the same order of magnitude after applying the inequality $2ab \leq a^2 + b^2$ and using Lemma \ref{lemma:mean-squares}. Thus showing that $\V[D^{(3)} \mid \mathcal{S}_2] = o_p(1)$.

Thus $D^{(3)}$ has negligible conditional variance, and therefore we have $D^{(3)} = \E[D^{(3)} \mid \mathcal{S}_2] 
+ o_p(1)$. Next we consider $ \E[D^{(3)} \mid \mathcal{S}_2]$.
Using the second part of Proposition \ref{non-unit-hermite-prop},
\begin{align*}
\E[D^{(3)} \mid \mathcal{S}_2 ] &= \mu_3 \|\Sigma^\frac12 \tilde{\gamma} \|_2    \sqrt{n}
 \E\left[  \frac{\langle \tilde{\gamma} , x_n \rangle }{\|\Sigma^\frac12 \tilde{\gamma} \|_2} \left[h_3(\xi_n) - h_3(\hat{\xi}_n)\right] \given \mathcal{S}_2 \right]  \\
&= -\mu_3 \|\Sigma^\frac12 \tilde{\gamma} \|_2   \sqrt{n} \E\left[  \frac{\langle \tilde{\gamma} , x_n \rangle }{\|\Sigma^\frac12 \tilde{\gamma} \|_2}  h_3(\hat{\xi}_n) \given \mathcal{S}_2 \right] \\
&=  -\mu_3   \sqrt{n} \left( \sqrt{3!} \langle  \Sigma^\frac12 \tilde{\gamma},  \Sigma^\frac12 \hat{\tau} \rangle  \left( \frac{\|\Sigma^\frac12 \hat{\tau} \|_2^2 - 1}{2}  \right) \right).
\end{align*}
First note that conditional on $\hat{\beta}$, the inverse of $\|\Sigma^\frac12 \hat{\tau} \|_2^2$ has a scaled $\chi^2_{\lfloor n/3 \rfloor}$ distribution:
\begin{align*}
\| \Sigma^\frac{1}{2} \hat{\tau} \|_2^{-2} &= \frac{ \| \hat{\Sigma}^\frac12 \hat{\beta} \|_2^2}{\| \Sigma^\frac12 \hat{\beta}\|_2^2} \mid \hat{\beta} \sim \frac{\chi^2_{\lfloor n/3 \rfloor}}{\lfloor n / 3 \rfloor}
\end{align*}
Using a conditional central limit theorem \citep{Bulinski2017ConditionalCL}, we have
\begin{align*}
\sqrt{\lfloor n/3 \rfloor} (\| \Sigma^\frac12 \hat{\tau} \|_2^{-2} - 1) \rightarrow_d N(0,2).
\end{align*}
Use the delta method on the function $u \mapsto u^{-1}$ and the above asymptotic result gives
\begin{align}\label{D3CLT}
\sqrt{n} (\| \Sigma^\frac12 \hat{\tau} \|_2^{2} - 1) \rightarrow_d N(0,\frac{2}{3}).
\end{align}

Next let us consider the multiplier $ \langle  \Sigma^\frac12 \tilde{\gamma},  \Sigma^\frac12 \hat{\tau} \rangle $. Recall that $(1 , -\gamma^T)^T = \Omega_1 \E r_n^2$ where $\Omega_1$ is the first column of $\Omega = \Sigma^{-1}$.  We can write
\begin{align*}
\left| \langle  \Sigma^\frac12 \tilde{\gamma},  \Sigma^\frac12 \hat{\tau} \rangle -   \frac{\hat{\beta}_1}{ \| \hat{\Sigma}^\frac{1}{2} \hat{\beta} \|_2} \E r_n^2 \right|&=\left| \langle  \Sigma^\frac12 \tilde{\gamma},  \Sigma^\frac12 \hat{\tau} \rangle - \langle \Sigma \Omega_1, \hat{\tau} \rangle \E r_n^2  \right| \\
&=  \left| \langle   \tilde{\gamma} - (1 , -\gamma^T)^T, \Sigma \hat{\tau} \rangle  \right| \\
&\leq \| \hat{\gamma} - \gamma \|_2 \cdot \| \Sigma \hat{\tau} \|_2  = o_p(1).
\end{align*}
On the other hand, we have
\begin{align*}
\left| \frac{\hat{\beta}_1}{ \| \hat{\Sigma}^\frac12 \hat{\beta} \| } - \frac{ \beta_1}{ \| \Sigma^\frac12 \beta \|_2}  \right|&\leq 
\left| \frac{\hat{\beta}_1}{ \| \hat{\Sigma}^\frac12 \hat{\beta} \| } - \frac{ \hat\beta_1}{ \| \Sigma^\frac12 \hat\beta \|_2}  \right| + 
\left| \frac{\hat{\beta}_1}{ \| \Sigma^\frac12 \hat{\beta} \|_2 } - \frac{ \hat\beta_1}{ \| \Sigma^\frac12 \beta \|_2}  \right| +
\left| \frac{\hat{\beta}_1}{ \| \Sigma^\frac12 \beta \|_2 } - \frac{ \beta_1}{ \| \Sigma^\frac12 \beta \|_2}  \right|\\
& \leq \frac{|\hat{\beta}_1|    \| \Sigma^\frac12 \hat{\tau}\|_2 - 1}{\| \hat{\beta} \|_2 \sqrt{\lambda_{\min}(\Sigma)} } + 
\frac{ |\hat{\beta}_1 |}{\| \hat{\beta} \|_2 \lambda_{\min}(\Sigma^\frac12)  } \frac{ \lambda_{\max}(\Sigma^\frac12) \| \hat{\beta} - \beta \|_2  }{\| \Sigma^\frac12 \beta \|_2} 
+ \frac{| \hat{\beta}_1 - \beta_1 |}{ \| \Sigma^\frac12 \beta \|_2}\\
&= o_p(1),
\end{align*}
since by assumption $|\mu_1| = \| \Sigma^\frac12 \beta \|_2$ is bounded away from zero.
The last two displays together imply that $ \langle  \Sigma^\frac12 \tilde{\gamma},  \Sigma^\frac12 \hat{\tau} \rangle  = \beta_1 / |\mu_1| + o_p(1)$. Using this and the asymptotic distribution (\ref{D3CLT}) we obtain
\begin{align*}
D^{(3)} = \E[D^{(3)} \mid \mathcal{S}_2] + o_p(1) = \eta_n \cdot Z_n + o_p(1), \text{ where } Z_n \rightarrow_d N\left(0, 1 \right) \text{ and } \eta_n^2 = \frac{ \mu_3^2   \beta_1^2 }{ \mu_1^2  }.
\end{align*}
Note that since the term $\| \Sigma^\frac12 \hat{\tau} \|_2^2 - 1$ is independent of the term (A) in the error expansion of $\tilde{\beta}_1$, we can also take $Z_n$ to be independent of $(A)$ (modulo an $o_p(1)$ term) .  It then follows that 
\begin{align*}
\left(\E[r_n^2 e_n^2] + \frac{\mu_3^2 \beta_1^2}{\mu_1^2} \right)^{-\frac12} \left( A + D^{(3)}  \right) \rightarrow_d N(0,1).
\end{align*}

Next we consider $D' = D - D^{(3)}$. Note that $\V[ \hat{r}_n / \|\Sigma^\frac12 \tilde{\gamma}  \|_2 \mid \mathcal{S}_2] =1$. Using Proposition \ref{non-unit-hermite-prop}, the conditional mean is given by
\begin{align*}
| \E[D' \mid \mathcal{S}_2] | &= \sqrt{n} \|\Sigma^\frac12 \tilde{\gamma}  \|_2   \left| \sum_{\stackrel{j=2}{j\neq 3}}^m \mu_j \E\left[ \frac{ \hat{r}_n}{ \| \Sigma^\frac12 \tilde{\gamma}\|_2} (h_j(\xi_n) - h_j(\hat{\xi}_n)) \given \mathcal{S}_2 \right] \right| \\
&=   \sqrt{n}  \left| \sum_{\stackrel{5 \leq j \leq m}{j \equiv_2 1}} \mu_j  \sqrt{j!} \langle \Sigma^\frac12 \tilde{\gamma} ,  \Sigma^\frac12 \hat{\tau} \rangle \left( \frac{ \| \Sigma^\frac12 \hat{\tau} \|_2^2 - 1}{2} \right)^{\frac{j-1}{2}}  \left( \left(\frac{j - 1}{2}\right)! \right)^{-1} \right|\\
&\leq \sqrt{n} ( \max_{j\geq 5} |\mu_j | ) \| \Sigma^\frac12 \tilde{\gamma} \|_2 \| \Sigma^\frac12 \hat{\tau} \|_2  
\left| \frac{ \| \Sigma^\frac12 \hat{\tau} \|_2^2 - 1}{2} \right|^2
\sum_{\stackrel{5 \leq j \leq m}{j \equiv_2 1}}  \sqrt{j!} \frac{ \left| \frac{ \| \Sigma^\frac12 \hat{\tau} \|_2^2 - 1}{2} \right|^{\frac{j-5}{2}} }{ \left( \left(\frac{j - 1}{2}\right)! \right) }.
\end{align*}
Using the inequality $j! / [(j-1)^2 (j-3)^2] \lesssim m^{j-5}$ for $5 \leq j \leq m$, we obtain
\begin{align*}
\sum_{\stackrel{5 \leq j \leq m}{j \equiv_2 1}}  \sqrt{j!} \frac{ \left| \frac{ \| \Sigma^\frac12 \hat{\tau} \|_2^2 - 1}{2} \right|^{\frac{j-5}{2}} }{ \left( \left(\frac{j - 1}{2}\right)! \right) } &\lesssim \sum_{\stackrel{5 \leq j \leq m}{j \equiv_2 1}} m^{\frac{j-5}{2}} \frac{ \left| \frac{ \| \Sigma^\frac12 \hat{\tau} \|_2^2 - 1}{2} \right|^{\frac{j-5}{2}} }{ \left( \left(\frac{j - 5}{2}\right)! \right) } \leq \exp\left( m \left| \frac{ \| \Sigma^\frac12 \hat{\tau} \|_2^2 - 1}{2} \right|  \right).
\end{align*}
Thus we have
\begin{align*}
| \E[D' \mid \mathcal{S}_2] |  &\lesssim \sqrt{n} ( \max_{j\geq 5} |\mu_j | ) \| \Sigma^\frac12 \tilde{\gamma} \|_2 \| \Sigma^\frac12 \hat{\tau} \|_2  
\left| \frac{ \| \Sigma^\frac12 \hat{\tau} \|_2^2 - 1}{2} \right|^2
\exp\left( m \left| \frac{ \| \Sigma^\frac12 \hat{\tau} \|_2^2 - 1}{2} \right|  \right)\\
&= \mathcal{O}_p\left( \sqrt{n} \frac{1}{n} \exp\left(\frac{m}{\sqrt{n}} \right) \right) = \mathcal{O}_p\left( \frac{1}{\sqrt{n}} \right).
\end{align*}

Next consider the the conditional variance of $D'$. Using the inequality $(\sum_1^m a_j)^2\leq m \sum_1^m a_j^2$ we can write
\begin{align*}
\V[ D' \mid \mathcal{S}_2] &= \V\left[  \sum_{\stackrel{j=2}{j\neq 3}}^m \mu_j \hat{r}_n (h_j(\xi_n) - h_j(\hat{\xi}_n)) \given \mathcal{S}_2 \right]  \\
&\leq  m \sum_{\stackrel{j=2}{j\neq 3}}^m \mu_j^2 \E[ \hat{r}_n^2 (h_j(\xi_n) - h_j(\hat{\xi}_n))^2 \mid \mathcal{S}_2] .
\end{align*}
Recall the Gaussian decompositions $\xi_n = \alpha \hat{r}_n + u_n$ and $\hat{x}_n = \hat{\alpha} \hat{r}_n + \hat{u}_n$ where $u_n, \hat{u}_n$ are conditionally independent of $r_n$ given $\mathcal{S}_2$. Condition on $\mathcal{S}_2, u_n, \hat{u}_n$ and use Stein's lemma and the tower property of conditional expectations to obtain
\begin{align*}
 \E[ \hat{r}_n^2 (h_j(\xi_n) - & h_j(\hat{\xi}_n))^2 \mid \mathcal{S}_2]  = \E[\hat{r}_n \mid \mathcal{S}_2] \E[( h_j (\xi_n) - h_j(\hat{\xi}_n))^2  \mid \mathcal{S}_2]  \\
 &+ 2 (\E[\hat{r}_n \mid \mathcal{S}_2])^2 j \E[( \alpha h_{j-1} (\xi_n) - \hat{\alpha} h_{j-1}(\hat{\xi}_n))^2  \mid \mathcal{S}_2] \\
 &+ 2 (\E[\hat{r}_n \mid \mathcal{S}_2])^2 \sqrt{j (j-1)} \E[ ( h_j(\xi) - h_j(\hat{\xi}) )(  \alpha^2h_{j-2}(\xi) -\hat{\alpha}^2 h_{j-2}(\hat{\xi}) ) \mid \mathcal{S}_2] 
\end{align*}
In light of Lemma \ref{lemma:mean-squares}, the first summand is $\mathcal{O}_p(j\sqrt{s\log(p) / n})$ while the second summand is $\mathcal{O}_p( j^2 \sqrt{s \log(p) / n})$. After applying the inequality $2ab \leq a^2 + b^2$ we find that the third summand is also $\mathcal{O}_p(j^2 \sqrt{s \log(p) / n})$ and therefore, plugging these into the upper bound on $\V[D' \mid \mathcal{S}_2]$ yields
\begin{align*}
\V[D' \mid \mathcal{S}_2] &=\mathcal{O}_p\left(  m \sqrt{\frac{s\log(p)}{n}}  \sum_{j=2, j\neq 3}^m j^2 \mu_j^2 \right).
\end{align*}
Using the Cauchy-Schwarz inequality and Lemma \ref{decay}, we obtain
\begin{align*}
 \sum_{j=2}^m j^2 \mu_j^2 &\leq  \left( \sum_{j=2}^m (j \mu_j^2)^2 \right)^\frac12 \left( \sum_{j=2}^m j^2 \right)^\frac12
 \lesssim \left(  \sum_{j=2}^m j \mu_j^2 \right) m^\frac32 \leq L^2 m^\frac32.
 \end{align*}
 By assumption $m^\frac52 \sqrt{s \log(p)} \leq \sqrt{\log(n)} \sqrt{s \log(p)} = o(\sqrt{n})$ and $L^2$ is bounded away from infinity. This implies that $\V[D' \mid \mathcal{S}_2] = o_p(1)$ which together with $ \E[D' \mid \mathcal{S}_2] = o_p(1)$ proves $D' = o_p(1)$.

\textbf{Step 6. (F).} Finally, note that the term $F$ does not depend on $\hat{\tau}$ or $\hat{\Sigma}$ and the rate of growth of $m$ is slower than required in Theorem \ref{thm-efficient} and therefore the proof we provided in the known $\Sigma$ case can be applied here to show that $F$ is $o_p(1)$. The only difference is that $r_i$ has to be replaced with $\hat{r}_i$ and the expectations will be conditional on the second subsample $\mathcal{S}_2$.

\begin{proposition}\label{variance_reduction}
Suppose that $ y = g( \langle x, \tau \rangle) + e$ where $\E[e \mid x ] = 0$ and the link function has a Hermite basis expansion $g = \sum_{j = 0}^\infty \mu_j h_j$. Then we have
\begin{align*}
\E[r^2 (y - \langle x, \beta \rangle)^2] - \E[ r^2 e^2] = \sum_{\stackrel{ j=0}{ j \neq 1}}^\infty \left[ \E[r^2] \mu_{j}^2 + 2(\tau_1 \E[r^2])^2 \left( j \mu_{j}^2 +   \sqrt{(j+1)(j+2)} \mu_j \mu_{j+2} \right) \right].
\end{align*}
\end{proposition}
\begin{proof}
It is easy to see that since $\E[e \mid x] = 0$ and $\beta = \mu_1 \tau$, we have
\begin{align*}
\E[r^2 (y - \langle x, \beta \rangle)^2] - \E[ r^2 e^2]  = \E[ r^2 \left(g(\langle x, \tau \rangle) - \mu_1 \langle x, \tau \rangle \right)^2].
\end{align*}
Let $\tilde{g}(\langle x, \tau \rangle) := g(\langle x, \tau \rangle) - \mu_1 \langle x, \tau \rangle  = \mu_0 + \sum_{j=2}^\infty \mu_j h_j(\langle x, \tau \rangle)$. We can write a Gaussian decomposition $ \xi := \langle x, \tau \rangle = \tau_1 r + u$ where $u$ is independent of $r$. Then applying Stein's lemma twice (conditionally given $u$) and then using the tower property of conditional expectations yields
\begin{align*}
\E[ r^2 \tilde g^2(\alpha r + u)] &= \E[r^2] \E[ \tilde g^2(\xi)] + 2 (\tau_1 \E[r^2])^2 \left(   \E[ \tilde g'^2(\xi)] +  \E[ \tilde g(\xi) \tilde g''(\xi)] \right)\\
&= \E[r^2] \sum_{\stackrel{ j=0}{ j \neq 1}}^\infty \mu_{j}^2 + 2(\tau_1 \E[r^2])^2 \sum_{\stackrel{ j=0}{ j \neq 1}}^\infty j \mu_{j}^2 +   \sqrt{(j+1)(j+2)} \mu_j \mu_{j+2}\\
&=  \sum_{\stackrel{ j=0}{ j \neq 1}}^\infty \left[ \E[r^2] \mu_{j}^2 + 2(\tau_1 \E[r^2])^2 \left( j \mu_{j}^2 +   \sqrt{(j+1)(j+2)} \mu_j \mu_{j+2} \right) \right].
\end{align*}
\end{proof}
\section{Simulations}
\label{simulations}
In this subsection we present the coverage rates of confidence intervals based on the debiased estimator $\tilde{\beta}$ defined by (\ref{est-def-known}) and (\ref{est-def-unknown}). We consider the combinations $n \in \{200, 500\}$, $s \in \{5,10\}$ and design covariance matrices $\Sigma_\kappa$ with $(\Sigma_\kappa)_{ij} = \kappa^{|i-j|}$ for $\kappa \in \{0, 0.5\}$ and the number of covariates $p = 2n$. In each case $\tau$ is defined by $\tau = \tilde{\tau}/\|\Sigma_\kappa^{\frac{1}{2}} \tilde{\tau} \|_2$, where
\begin{align*}
\tilde{\tau}_j = \begin{cases}
s - j +1 & : 1 \leq j \leq s,\\
0 & : s < j \leq p.
\end{cases}
\end{align*}

We also consider two different link functions:
\begin{align*}
\textbf{Model 1:} \, \, y_i &=\operatorname{sign}(\langle x_i, \tau \rangle) + \epsilon_i, \, \epsilon_i \sim N(0,1) \\
\textbf{Model 2:} \, \, y_i &= U_i \cdot e^{\langle x_i, \tau\rangle},
\end{align*}
where $U_1, \dots, U_n$ are iid draws from the exponential distribution with rate one, independent of $x_i$.

\textbf{Construction of Confidence Intervals.}\footnote{The R code for simulations is available at \url{https://github.com/ehamid/sim_debiasing}.} The approximate variance of the debiased estimators in Theorems \ref{thm_known_sig} and \ref{thm-unknown-sigma} is equal to $\E r_n^2 z_n^2 / (\E r_n x_{n,k})^2$. We estimate this variance by replacing the expectation with empirical averages of natural estimates of $r_n, z_n$. The $95\%$ confidence intervals for $\beta_k$ are then constructed using
\begin{align*}
\tilde{\beta}_k \pm q_{0.025} \cdot \left(\frac{\sqrt{\sum_{i=1}^n (y_i - x_i^T \hat{\beta})^2 \hat{r}_i^2}}{\sum_{i=1}^n \hat{r}_i x_{ik}} \right),
\end{align*}
where $q_{0.025}$ is the $0.025$ quantile of the standard normal distribution, and $\hat{r_i}$ is computed according to equation (\ref{resid_known_Sig}) or (\ref{resid_unkn_sig}) depending on whether or not $\Sigma$ is assumed known. The pilot estimate $\hat{\beta}$ was computed using the lasso, with the tuning parameter found using ten-fold cross-validation.\footnote{The function \texttt{cv.glmnet} in the R package \texttt{glmnet} \citep{Friedman2010} was used.} In the case of unknown $\Sigma$, the tuning parameter of the node-wise lasso (\ref{nodewise})  was chosen by
\begin{align*}
\lambda_k = \min \left\{ \lambda > 0 : \max_{j\neq k} \left|\frac{\sum_{i=1}^n \hat{r}_i x_{ij}}{\sqrt{\sum_{i=1}^n \hat{r}_i^2}} \right| \geq \sqrt{\log(p)} \right\}.
\end{align*}
This choice of $\lambda_k$ is motivated by the fact that $\max_{j\neq k} |\sum_1^n \hat{r}_i x_{ij}|/\sqrt{\sum_{i=1}^n \hat{r}_i^2} = \mathcal{O}_p(\sqrt{\log p})$, and the value of $\lambda_k$ maintains a trade-off between the bias and variance of the debiased estimator, see Table 2 of \cite{zhang2014confidence} for a similar tuning method and a detailed explanation of the trade-off.
The following measures were computed:
\begin{itemize}
	\item $\overline{cov}(S)$: computed by averaging the coverage rates of confidence intervals for non-zero coefficients.
	
	\item $\overline{cov}(S^c)$: computed by averaging the coverage rates of confidence intervals for 10 randomly chosen (at each of 200 replicates) coefficients in $S^c$.
	
	\item  $\bar{l}(S)$ average length of confidence intervals for coefficients in $S$.
	\item $\bar{l}(S^c)$ average length of confidence intervals for coefficients in $S^c$, computed by averaging the lengths of confidence intervals for 10 randomly chosen (at each of 200 replicates) coefficients in $S^c$.
	\item FPR: The average False Positive Rate corresponding to 10 randomly chosen coefficients in $S^c$. (Proportion of confidence intervals corresponding to $S^c$ that did not include zero.)
	\item TPR: The average True Positive Rate for coefficients in $S$. (Proportion of confidence intervals over $S$ that did not include zero.)
	\item TPR(j): The True Positive Rate for confidence intervals corresponding to $\beta_j$ for $1 \leq j \leq 5$.
\end{itemize}

Finally, the last four tables report simulation results for Model 1 when $x$ has non-zero mean, $\E x = \mathbf{1}$. In this case the sample column means were used to center $X$ before computing the estimators, i.e. $\tilde{X} := (I - \frac{1}{n}\mathbf{1}\mathbf{1}^T)X$ was used in the procedure.

Some general observations are as follows:
\begin{enumerate}
    \item In general, coverage rates are close to the nominal level (95\%) and coverage is improved as the sample size increases from 200 to 500.
    \item Coverage rates for the null coefficients almost always dominate the coverage rate over the support set $S$.
    \item Introducing correlation among covariates (increasing $\kappa$ from 0 to 0.5) leads to longer confidence intervals and decreases the power of the tests (the TPRs). The coverage rates do not necessarily suffer from this correlation.
    \item Similarly, increasing the number of non-null coefficients from 5 to 10 decreases the power.
\end{enumerate}

\begin{table}[!htbp]
	\centering
	\begin{tabular}{|rl|rrrr|}
		\hline
		\multicolumn{2}{|l|}{\diagbox[]{$(n, \kappa, s)$}{Measure}}  & $\overline{cov}(S)$ & $\overline{cov}(S^c)$ & $\bar{l}(S)$ & $\bar{l}(S^c)$ \\
  \hline
1 & (200, 0, 5) & 0.90 & 0.94 & 0.29 & 0.30 \\
  2 & (200, 0, 10) & 0.91 & 0.94 & 0.30 & 0.30 \\
  3 & (200, 0.5, 5) & 0.92 & 0.95 & 0.39 & 0.40 \\
  4 & (200, 0.5, 10) & 0.92 & 0.95 & 0.39 & 0.39 \\
  5 & (500, 0, 5) & 0.93 & 0.94 & 0.19 & 0.20 \\
  6 & (500, 0, 10) & 0.93 & 0.95 & 0.19 & 0.20 \\
  7 & (500, 0.5, 5) & 0.93 & 0.94 & 0.25 & 0.26 \\
  8 & (500, 0.5, 10) & 0.94 & 0.95 & 0.25 & 0.26 \\
   \hline
	\end{tabular}
	\caption{Coverage results for model 1 when $\Sigma$ is known.}
\end{table}
\begin{table}[!htbp]
	\centering
	\begin{tabular}{|rl|rrrrrrr|}
		\hline
		\multicolumn{2}{|l|}{\diagbox[]{$(n, \kappa, s)$}{Measure}} & FPR & TPR & TPR(1) & TPR(2) & TPR(3) & TPR(4) & TPR(5) \\
  \hline
1 & (200, 0, 5) & 0.06 & 0.78 & 1.00 & 1.00 & 0.96 & 0.72 & 0.22 \\
  2 & (200, 0, 10) & 0.06 & 0.61 & 0.99 & 0.98 & 0.96 & 0.88 & 0.79 \\
  3 & (200, 0.5, 5) & 0.05 & 0.57 & 0.98 & 0.79 & 0.57 & 0.40 & 0.11 \\
  4 & (200, 0.5, 10) & 0.05 & 0.34 & 0.78 & 0.61 & 0.56 & 0.41 & 0.34 \\
  5 & (500, 0, 5) & 0.06 & 0.90 & 1.00 & 1.00 & 1.00 & 0.99 & 0.52 \\
  6 & (500, 0, 10) & 0.05 & 0.78 & 1.00 & 1.00 & 1.00 & 0.99 & 0.98 \\
  7 & (500, 0.5, 5) & 0.06 & 0.73 & 1.00 & 1.00 & 0.91 & 0.59 & 0.17 \\
  8 & (500, 0.5, 10) & 0.05 & 0.55 & 0.99 & 0.93 & 0.85 & 0.78 & 0.65 \\
   \hline
	\end{tabular}
	\caption{Average True/False positive rates for model 1 when $\Sigma$ is known.}
\end{table}
\begin{table}[!htbp]
	\centering
	\begin{tabular}{|rl|rrrr|}
		\hline
		\multicolumn{2}{|l|}{\diagbox[]{$(n, \kappa, s)$}{Measure}} & $\overline{cov}(S)$ & $\overline{cov}(S^c)$ & $\bar{l}(S)$ & $\bar{l}(S^c)$ \\
  \hline
1 & (200, 0, 5) & 0.92 & 0.94 & 0.38 & 0.38 \\
  2 & (200, 0, 10) & 0.92 & 0.94 & 0.39 & 0.39 \\
  3 & (200, 0.5, 5) & 0.91 & 0.95 & 0.46 & 0.47 \\
  4 & (200, 0.5, 10) & 0.92 & 0.95 & 0.46 & 0.46 \\
  5 & (500, 0, 5) & 0.95 & 0.94 & 0.25 & 0.26 \\
  6 & (500, 0, 10) & 0.93 & 0.95 & 0.25 & 0.25 \\
  7 & (500, 0.5, 5) & 0.93 & 0.94 & 0.31 & 0.32 \\
  8 & (500, 0.5, 10) & 0.94 & 0.95 & 0.31 & 0.31 \\
   \hline
	\end{tabular}
	\caption{Coverage results for model 1 when $\Sigma$ is unknown.}
\end{table}
\begin{table}[!htbp]
	\centering
	\begin{tabular}{|rl|rrrrrrr|}
		\hline
		\multicolumn{2}{|l|}{\diagbox[]{$(n, \kappa, s)$}{Measure}} & FPR & TPR & TPR(1) & TPR(2) & TPR(3) & TPR(4) & TPR(5) \\
  \hline
1 & (200, 0, 5) & 0.06 & 0.71 & 1.00 & 0.97 & 0.90 & 0.51 & 0.19 \\
  2 & (200, 0, 10) & 0.06 & 0.54 & 0.99 & 0.97 & 0.85 & 0.76 & 0.58 \\
  3 & (200, 0.5, 5) & 0.05 & 0.54 & 0.91 & 0.76 & 0.55 & 0.34 & 0.14 \\
  4 & (200, 0.5, 10) & 0.05 & 0.34 & 0.67 & 0.57 & 0.59 & 0.45 & 0.35 \\
  5 & (500, 0, 5) & 0.06 & 0.84 & 1.00 & 1.00 & 1.00 & 0.90 & 0.31 \\
  6 & (500, 0, 10) & 0.05 & 0.71 & 1.00 & 1.00 & 1.00 & 0.97 & 0.93 \\
  7 & (500, 0.5, 5) & 0.06 & 0.69 & 1.00 & 0.97 & 0.83 & 0.49 & 0.15 \\
  8 & (500, 0.5, 10) & 0.05 & 0.50 & 0.94 & 0.85 & 0.82 & 0.67 & 0.56 \\
   \hline
	\end{tabular}
	\caption{Average True/False positive rates for model 1 when $\Sigma$ is unknown.}
\end{table}
\begin{table}[!htbp]
	\centering
	\begin{tabular}{|rl|rrrr|}
		\hline
		\multicolumn{2}{|l|}{\diagbox[]{$(n, \kappa, s)$}{Measure}} & $\overline{cov}(S)$ & $\overline{cov}(S^c)$ & $\bar{l}(S)$ & $\bar{l}(S^c)$ \\
  \hline
1 & (200, 0, 5) & 0.88 & 0.96 & 1.09 & 0.87 \\
  2 & (200, 0, 10) & 0.89 & 0.95 & 0.98 & 0.87 \\
  3 & (200, 0.5, 5) & 0.92 & 0.96 & 1.18 & 1.14 \\
  4 & (200, 0.5, 10) & 0.93 & 0.95 & 1.12 & 1.10 \\
  5 & (500, 0, 5) & 0.90 & 0.96 & 0.71 & 0.56 \\
  6 & (500, 0, 10) & 0.91 & 0.96 & 0.63 & 0.55 \\
  7 & (500, 0.5, 5) & 0.93 & 0.95 & 0.80 & 0.75 \\
  8 & (500, 0.5, 10) & 0.93 & 0.96 & 0.74 & 0.71 \\
   \hline
	\end{tabular}
	\caption{Coverage results for model 2 when $\Sigma$ is known.}
\end{table}
\begin{table}[!htbp]
	\centering
	\begin{tabular}{|rl|rrrrrrr|}
		\hline
		\multicolumn{2}{|l|}{\diagbox[]{$(n, \kappa, s)$}{Measure}} & FPR & TPR & TPR(1) & TPR(2) & TPR(3) & TPR(4) & TPR(5) \\
  \hline
1 & (200, 0, 5) & 0.04 & 0.63 & 0.98 & 0.93 & 0.70 & 0.38 & 0.15 \\
  2 & (200, 0, 10) & 0.04 & 0.46 & 0.89 & 0.82 & 0.73 & 0.59 & 0.56 \\
  3 & (200, 0.5, 5) & 0.04 & 0.40 & 0.74 & 0.52 & 0.43 & 0.23 & 0.07 \\
  4 & (200, 0.5, 10) & 0.05 & 0.20 & 0.48 & 0.41 & 0.29 & 0.23 & 0.15 \\
  5 & (500, 0, 5) & 0.04 & 0.81 & 1.00 & 0.99 & 0.99 & 0.79 & 0.27 \\
  6 & (500, 0, 10) & 0.04 & 0.67 & 1.00 & 0.99 & 0.98 & 0.94 & 0.90 \\
  7 & (500, 0.5, 5) & 0.05 & 0.61 & 1.00 & 0.90 & 0.69 & 0.35 & 0.10 \\
  8 & (500, 0.5, 10) & 0.04 & 0.38 & 0.81 & 0.74 & 0.57 & 0.51 & 0.41 \\
   \hline
	\end{tabular}
	\caption{Average True/False positive rates for model 2 when $\Sigma$ is known.}
\end{table}
\begin{table}[!htbp]
	\centering
	\begin{tabular}{|rl|rrrr|}
		\hline
		\multicolumn{2}{|l|}{\diagbox[]{$(n, \kappa, s)$}{Measure}} & $\overline{cov}(S)$ & $\overline{cov}(S^c)$ & $\bar{l}(S)$ & $\bar{l}(S^c)$ \\
  \hline
1 & (200, 0, 5) & 0.88 & 0.96 & 1.27 & 1.08 \\
  2 & (200, 0, 10) & 0.91 & 0.95 & 1.20 & 1.09 \\
  3 & (200, 0.5, 5) & 0.94 & 0.95 & 1.34 & 1.22 \\
  4 & (200, 0.5, 10) & 0.94 & 0.95 & 1.28 & 1.25 \\
  5 & (500, 0, 5) & 0.92 & 0.96 & 0.85 & 0.72 \\
  6 & (500, 0, 10) & 0.92 & 0.95 & 0.77 & 0.71 \\
  7 & (500, 0.5, 5) & 0.93 & 0.95 & 0.94 & 0.87 \\
  8 & (500, 0.5, 10) & 0.94 & 0.96 & 0.89 & 0.86 \\
   \hline
	\end{tabular}
	\caption{Coverage results for model 2 when $\Sigma$ is unknown.}
\end{table}
\begin{table}[!htbp]
	\centering
	\begin{tabular}{|rl|rrrrrrr|}
		\hline
		\multicolumn{2}{|l|}{\diagbox[]{$(n, \kappa, s)$}{Measure}} & FPR & TPR & TPR(1) & TPR(2) & TPR(3) & TPR(4) & TPR(5) \\
  \hline
1 & (200, 0, 5) & 0.04 & 0.53 & 0.94 & 0.78 & 0.54 & 0.34 & 0.08 \\
  2 & (200, 0, 10) & 0.05 & 0.35 & 0.75 & 0.66 & 0.51 & 0.41 & 0.39 \\
  3 & (200, 0.5, 5) & 0.05 & 0.34 & 0.67 & 0.54 & 0.30 & 0.15 & 0.07 \\
  4 & (200, 0.5, 10) & 0.05 & 0.21 & 0.37 & 0.40 & 0.39 & 0.24 & 0.28 \\
  5 & (500, 0, 5) & 0.04 & 0.76 & 1.00 & 1.00 & 0.96 & 0.67 & 0.15 \\
  6 & (500, 0, 10) & 0.05 & 0.59 & 0.99 & 0.94 & 0.91 & 0.89 & 0.71 \\
  7 & (500, 0.5, 5) & 0.05 & 0.56 & 0.94 & 0.81 & 0.63 & 0.30 & 0.10 \\
  8 & (500, 0.5, 10) & 0.04 & 0.32 & 0.71 & 0.62 & 0.44 & 0.42 & 0.33 \\
   \hline
	\end{tabular}
	\caption{Average True/False positive rates for model 2 when $\Sigma$ is unknown.}
\end{table}
\begin{table}[!htbp]
	\centering
	\begin{tabular}{|rl|rrrr|}
		\hline
		\multicolumn{2}{|l|}{\diagbox[]{$(n, \kappa, s)$}{Measure}}  & $\overline{cov}(S)$ & $\overline{cov}(S^c)$ & $\bar{l}(S)$ & $\bar{l}(S^c)$ \\
  \hline
1 & (200, 0, 5) & 0.90 & 0.94 & 0.29 & 0.30 \\
  2 & (200, 0, 10) & 0.91 & 0.94 & 0.30 & 0.31 \\
  3 & (200, 0.5, 5) & 0.92 & 0.95 & 0.39 & 0.40 \\
  4 & (200, 0.5, 10) & 0.93 & 0.95 & 0.39 & 0.39 \\
  5 & (500, 0, 5) & 0.93 & 0.94 & 0.19 & 0.20 \\
  6 & (500, 0, 10) & 0.93 & 0.95 & 0.19 & 0.20 \\
  7 & (500, 0.5, 5) & 0.93 & 0.94 & 0.25 & 0.26 \\
  8 & (500, 0.5, 10) & 0.94 & 0.95 & 0.25 & 0.26 \\
   \hline
   \end{tabular}
	\caption{Coverage results for model 1 when $\Sigma$ is known and $\E [x_i] = \mathbf{1}$.}
\end{table}
\begin{table}[!htbp]
	\centering
	\begin{tabular}{|rl|rrrrrrr|}
		\hline
		\multicolumn{2}{|l|}{\diagbox[]{$(n, \kappa, s)$}{Measure}} & FPR & TPR & TPR(1) & TPR(2) & TPR(3) & TPR(4) & TPR(5) \\
  \hline
1 & (200, 0, 5) & 0.06 & 0.78 & 1.00 & 1.00 & 0.97 & 0.68 & 0.24 \\
  2 & (200, 0, 10) & 0.06 & 0.61 & 0.99 & 0.99 & 0.95 & 0.88 & 0.75 \\
  3 & (200, 0.5, 5) & 0.05 & 0.56 & 0.96 & 0.78 & 0.57 & 0.37 & 0.12 \\
  4 & (200, 0.5, 10) & 0.05 & 0.34 & 0.77 & 0.61 & 0.53 & 0.42 & 0.34 \\
  5 & (500, 0, 5) & 0.06 & 0.90 & 1.00 & 1.00 & 1.00 & 0.98 & 0.51 \\
  6 & (500, 0, 10) & 0.05 & 0.78 & 1.00 & 1.00 & 1.00 & 1.00 & 0.99 \\
  7 & (500, 0.5, 5) & 0.06 & 0.73 & 1.00 & 0.99 & 0.91 & 0.58 & 0.17 \\
  8 & (500, 0.5, 10) & 0.05 & 0.55 & 0.99 & 0.93 & 0.86 & 0.78 & 0.65 \\
   \hline
	\end{tabular}
	\caption{Average True/False positive rates for model 1 when $\Sigma$ is known and $\E [x_i] = \mathbf{1}$.}
\end{table}
\begin{table}[!htbp]
	\centering
	\begin{tabular}{|rl|rrrr|}
		\hline
		\multicolumn{2}{|l|}{\diagbox[]{$(n, \kappa, s)$}{Measure}}  & $\overline{cov}(S)$ & $\overline{cov}(S^c)$ & $\bar{l}(S)$ & $\bar{l}(S^c)$ \\
  \hline
1 & (200, 0, 5) & 0.91 & 0.94 & 0.38 & 0.39 \\
  2 & (200, 0, 10) & 0.93 & 0.94 & 0.39 & 0.40 \\
  3 & (200, 0.5, 5) & 0.91 & 0.95 & 0.47 & 0.48 \\
  4 & (200, 0.5, 10) & 0.92 & 0.95 & 0.46 & 0.46 \\
  5 & (500, 0, 5) & 0.95 & 0.95 & 0.25 & 0.26 \\
  6 & (500, 0, 10) & 0.94 & 0.94 & 0.25 & 0.26 \\
  7 & (500, 0.5, 5) & 0.93 & 0.95 & 0.31 & 0.32 \\
  8 & (500, 0.5, 10) & 0.93 & 0.95 & 0.31 & 0.31 \\
   \hline
	\end{tabular}
	\caption{Coverage results for model 1 when $\Sigma$ is unknown and $\E[ x_i] = \mathbf{1}$.}
\end{table}
\begin{table}[!htbp]
	\centering
	\begin{tabular}{|rl|rrrrrrr|}
		\hline
		\multicolumn{2}{|l|}{\diagbox[]{$(n, \kappa, s)$}{Measure}} & FPR & TPR & TPR(1) & TPR(2) & TPR(3) & TPR(4) & TPR(5) \\
  \hline
1 & (200, 0, 5) & 0.06 & 0.71 & 1.00 & 0.97 & 0.88 & 0.54 & 0.18 \\
  2 & (200, 0, 10) & 0.06 & 0.53 & 0.95 & 0.95 & 0.85 & 0.72 & 0.56 \\
  3 & (200, 0.5, 5) & 0.05 & 0.54 & 0.90 & 0.76 & 0.56 & 0.35 & 0.15 \\
  4 & (200, 0.5, 10) & 0.05 & 0.34 & 0.66 & 0.56 & 0.56 & 0.47 & 0.36 \\
  5 & (500, 0, 5) & 0.05 & 0.84 & 1.00 & 1.00 & 1.00 & 0.89 & 0.32 \\
  6 & (500, 0, 10) & 0.06 & 0.71 & 1.00 & 1.00 & 1.00 & 0.96 & 0.94 \\
  7 & (500, 0.5, 5) & 0.05 & 0.68 & 0.99 & 0.96 & 0.83 & 0.47 & 0.15 \\
  8 & (500, 0.5, 10) & 0.05 & 0.50 & 0.93 & 0.83 & 0.80 & 0.67 & 0.57 \\
   \hline
	\end{tabular}
	\caption{Average True/False positive rates for model 1 when $\Sigma$ is unknown and $\E [x_i] = \mathbf{1}$.}
\end{table}

\FloatBarrier
\vskip 0.2in
\bibliography{biblio}

\end{document}